\documentclass[reqno,11pt]{amsart}
\usepackage{hyperref}
\usepackage{graphicx}
\theoremstyle{plain}
\newtheorem{step}{Step}
\newtheorem{thm}{Theorem}[section]
\newtheorem{theorem}[thm]{Theorem}
\newtheorem{cor}[thm]{Corollary}
\newtheorem{corollary}[thm]{Corollary}

\newtheorem{lemma}[thm]{Lemma}

\newtheorem{proposition}[thm]{Proposition}

\theoremstyle{remark}

\newtheorem{remark}[thm]{Remark}

\theoremstyle{definition}

\newtheorem{definition}[thm]{Definition}
\def\al{{\alpha}}
\def\be{{\beta}}
\def\de{{\delta}}
\def\De{{\Delta}}
\def\om{{\omega}}
\def\Om{{\Omega}}

\let\La\Lambda

\def\si{{\sigma}}

\def\ga{{\gamma}}
\def\ep{{\epsilon}}
\def\Ga{{\Gamma}}
\def\th{{\theta}}
\def\Th{{\Theta}}

\def\phi{{\varphi}}

\let\pa\partial
\let\na\nabla

\DeclareMathAlphabet{\doba}{U}{msb}{m}{n}
\gdef\mC{\doba{C}}
\gdef\mH{\doba{H}}
\gdef\mN{\doba{N}}

\gdef\mR{\doba{R}}
\gdef\mS{\doba{S}}

\gdef\mZ{\doba{Z}}

\def\cE{{\mathcal{E}}}

\def\cR{{\mathcal{R}}}
\def\Spin{{\mathop{\rm Spin}}}

\def\Vol{{\mathop{\rm Vol}}}
\let\vol\Vol
\def\Scal{{\mathop{\rm Scal}}}
\let\scal\Scal

\def\mukim{\mu^{(0)}}

\def\diam{{\mathop{\rm diam}}}

\def\Id{{\mathop{\rm Id}}}
\def\End{{\mathop{\rm End}}}
\def\trace{{\mathop{\rm tr}}}
\let\tr\trace

\let\ti\tilde

\def\d2dt{\frac{d^2}{dt^2}}
\def\grad{{\mathop{\rm grad}}}
\def\divergence{{\mathop{\rm div}}}
\let\embed\hookrightarrow

\def\Rmax{R_{\textrm{max}}}

\let\<\langle
\let\>\rangle

\def\dvert{{d_{\rm vert}}}
\def\sidash{$\si$\nobreakdash-{}}

\long\def\ignorethis#1{}

\newdimen\templaenge

\def\Atbox#1#2{\setbox0\hbox{$\displaystyle #1$}\templaenge=\textwidth\advance\templaenge by -\wd0%
\setbox1\hbox{$#2$}\advance\templaenge by -\wd1%
$$#1\hbox{\kern\templaenge$#2$\hss}$$\par\bigbreak}

\newtheorem*{caseI}{Case I}
\newtheorem*{subcaseI.1}{Subcase I.1}
\newtheorem*{subcaseI.2}{Subcase I.2}
\newtheorem*{caseII}{Case II}
\newtheorem*{subcaseII.1}{Subcase II.1}
\newtheorem*{subsubcaseII.1.1}{Subsubcase II.1.1}
\newtheorem*{subsubcaseII.1.2}{Subsubcase II.1.2}
\newtheorem*{subcaseII.2}{Subcase II.2}

\def\an{a}
\def\detwo{\de_0}
\def\WS{$WS$}
\let\WSk\WS
\def\gWS{g_{\rm WS}}

\newcommand{\definedas}{\mathrel{\raise.095ex\hbox{\rm :}\mkern-5.2mu=}}

\begin{document}
%%%%%%%%%%%%%%%%%%%%%%%%%%%%%%%%%%%%%%%%%%%%%%%%%%%%%%%%%%%%%%%%%

%%%%%%%%%%%%%%%%%%%%%%%%%%%%%%%%%%%%%%%%%%%%%%%%%%%%%%%%%%%%%%%%%
%\begin{center}
%\framebox{\framebox{
%\vbox{This is project {\red \Project}\\
%Current version {\blue\Version}, from
%{\blue\Datum}, most recent changes by {\blue\Person}.}
%}}
%\end{center}
%%%%%%%%%%%%%%%%%%%%%%%%%%%%%%%%%%%%%%%%%%%%%%%%%%%%%%%%%%%%%%%%%

\title{Smooth Yamabe invariant and surgery}

\author{Bernd Ammann}
\address{Fakult\"at f\"ur Mathematik \\
Universit\"at Regensburg \\
93040 Regensburg \\
Germany\\}
\email{bernd.ammann@mathematik.uni-regensburg.de}

\author{Mattias Dahl}
\address{Institutionen f\"or Matematik \\
Kungliga Tekniska H\"ogskolan \\
100 44 Stockholm \\
Sweden\\}
\email{dahl@math.kth.se}

\author{Emmanuel Humbert}
\address{Laboratoire de Math\'ematiques et Physique Th\'eorique \\
Universit\'e de Tours \\
Parc de Grandmont \\
37200 Tours \\
France \\}
\email{Emmanuel.Humbert@lmpt.univ-tours.fr}

\begin{abstract}
We prove a surgery formula for the smooth Yamabe invariant~$\sigma(M)$
of a compact manifold~$M$. Assume that~$N$ is obtained from~$M$ by
surgery of codimension at least~$3$. We prove the existence of a
positive constant~$\Lambda_n$, depending only on the dimension~$n$
of~$M$, such that
\[
\si(N) \geq \min\{\si(M),\La_n\}.
\]
\end{abstract}

%
% 35J60   Nonlinear PDE of elliptic type
% 35P30   Nonlinear eigenvalue problems, nonlinear spectral theory for PDO
% 57R65   Surgery and handlebodies
% 58J50   Spectral problems; spectral geometry; scattering theory
% 58C40   Spectral theory; eigenvalue problems
%
% Old codim2 {53C27 (Primary) 55N22, 57R65 (Secondary)}
% Old mu2-MSC%%%% 53A30, 35J60(Primary) 35P30, 58J50, 58C40 (Secondary)

\date{September 2nd, 2012}


\maketitle

\setcounter{tocdepth}{1}
\tableofcontents

%%%%%%%%%%%%%%%%%%%%%%%%%%%%%%%%%%%%%%%%%%%%%%%%%%%%%%%%%%%%%%%%%%%%%%%%%
\section{Introduction}
%%%%%%%%%%%%%%%%%%%%%%%%%%%%%%%%%%%%%%%%%%%%%%%%%%%%%%%%%%%%%%%%%%%%%%%%%

%%%%%%%%%%%%%%%%%%%%%%%%%%%%%%%%%%%%%%%%%%%%%%%%%%%%%%%%%%%%%%%%%%%%%%%%%
\subsection{Main result}
%%%%%%%%%%%%%%%%%%%%%%%%%%%%%%%%%%%%%%%%%%%%%%%%%%%%%%%%%%%%%%%%%%%%%%%%%

The smooth Yamabe invariant, also called Schoen's \sidash{}constant,
of a compact manifold $M$ is defined as
\[
\si(M) \definedas \sup \inf \int_M\scal^g\,dv^g,
\]
where the supremum runs over all conformal classes $[g_0]$ on $M$ and
the infimum runs over all metrics $g$ of volume $1$ in $[g_0]$. The
integral $\cE(g) \definedas \int_M\scal^g\,dv^g$ is the integral of
the scalar curvature of $g$ integrated with respect to the volume element
of $g$ and is known as the Einstein--Hilbert functional.

Let $n = \dim M$. We assume that $N$ is obtained from $M$ by surgery
of codimension $n-k \geq 3$. That is for a given embedding
$S^k \embed M$, with trivial normal bundle, $0\leq k\leq n-3$, we
remove a tubular neighborhood $U_\ep(S^k)$ of this embedding.
The resulting manifold has boundary $S^k \times S^{n-k-1}$. This boundary
is glued together with the boundary of $B^{k+1} \times S^{n-k-1}$, and we
thus obtain the closed smooth manifold
\[
N \definedas
(M\setminus U_\ep(S^k))
\cup_{S^k \times S^{n-k-1}}
(B^{k+1} \times S^{n-k-1}).
\]

Our main result is the existence of a positive constant $\La_n$
depending only on $n$ such that
\[
\si(N)\geq \min\{\si(M),\La_n\}.
\]
This formula unifies and generalizes previous results obtained by
Gromov and Lawson, Schoen and Yau, Kobayashi, Peteana nd Yun. It also 
allows many
conclusions by using bordism theory.

In Section~\ref{sec.history} we give a detailed description of the
background of our result, a stronger version of the main result
follows in Section~\ref{subsec.stronger}, followed by a sketch of topological applications, see Section~\ref{subsec.topapp}.
The construction of a generalization of
surgery is recalled in Section~\ref{joining_man}.
Then, in Section~\ref{sec.La.const} the constant $\La_n$ is described
and it is proven
to be positive. After the proof of some preliminary results on limit
spaces in Section~\ref{sec.limit}, we derive a key estimate in Section
\ref{wsbundles}, namely an estimate for
the $L^2$-norm of solutions of a perturbed Yamabe equation on a
special kind of sphere bundle, called \WS-bundle. The last section
contains the proof of the main theorem, Theorem \ref{main.weak}.

%%%%%%%%%%%%%%%%%%%%%%%%%%%%%%%%%%%%%%%%%%%%%%%%%%%%%%%%%%%%%%%%%%%%%%%%%
\subsection{Basic notions, the Yamabe problem, and
some surgery formulas} \label{sec.history}
%%%%%%%%%%%%%%%%%%%%%%%%%%%%%%%%%%%%%%%%%%%%%%%%%%%%%%%%%%%%%%%%%%%%%%%%%

We denote by $B^n(r)$ the open ball of radius $r$ around $0$ in
$\mR^n$ and we set $B^n \definedas B^n(1)$. The unit sphere in $\mR^n$
is denoted by $S^{n-1}$. By $\xi^n$ we denote the standard flat metric
on $\mR^n$ and by $\sigma^{n-1}$ the standard metric of constant
sectional curvature $1$ on $S^{n-1}$. We denote the Riemannian
manifold $(S^{n-1},\sigma^{n-1})$ by $\mS^{n-1}$. All manifolds in this article 
are manifolds without boundary unless stated differently.

Let $(M,g)$ be a Riemannian manifold of dimension $n$. The Yamabe
operator, or Conformal Laplacian, acting on smooth functions on $M$
is defined by
\[
L^g u = \an \Delta^g u + \scal^g u,
\]
where $\an \definedas \frac{4(n-1)}{n-2}$ and where
$\Delta^g = \divergence^g \grad^g$ is the non-negative Laplacian
associated to the metric $g$. Let $p \definedas \frac{2n}{n-2}$.
Define the functional $J^g$ acting on non-zero compactly supported
smooth functions on $M$ by
\begin{equation} \label{def.J^g}
J^g(u)
\definedas
\frac{\int_M u L^g u \, dv^g}
{\left( \int_M u^p \, dv^g \right)^{\frac{2}{p}} }.
\end{equation}
If $g$ and $\tilde{g} = f^{\frac{4}{n-2}} g = f^{p-2} g$ are conformal
metrics on $M$, then the corresponding Yamabe operators are related by
\begin{equation} \label{confL}
L^{\ti g} u
=
f^{-\frac{n+2}{n-2}} L^{g} (fu)
=
f^{1-p} L^{g} (fu).
\end{equation}
It follows that
\begin{equation} \label{confJ}
J^{\ti g} (u)
=
J^{g} (fu).
\end{equation}
For a compact Riemannian manifold $(M,g)$ the conformal Yamabe
constant is defined by
\[
\mu(M,g) \definedas \inf J^g(u)\in\mR,
\]
where the infimum is taken over all non-zero smooth functions $u$ on
$M$. The same value of $\mu(M,g)$ is obtained if one takes the infimum
over positive smooth functions. {}From \eqref{confJ} it follows that the
invariant $\mu$ depends only on the conformal class $[g]$ of~$g$, and
the notation $\mu(M,[g]) = \mu(M,g)$ is also used. For the standard
sphere we have
\begin{equation} \label{mu(S^n)}
\mu(\mS^n) = n(n-1){\om_n}^{2/n},
\end{equation}
where $\om_n$ denotes the volume of $\mS^n$. This value is a universal
upper bound for~$\mu$.
\begin{theorem}[{\cite[Lemma 3]{aubin:76}}]\label{aubin}
The inequality
\[
\mu(M,g) \leq \mu(\mS^n)
\]
holds for any compact Riemannian manifold $(M,g)$.
\end{theorem}

For $u>0$ the $J^g$-functional is related to the
Einstein-Hilbert-functional via
\[
J^g(u)
=
\frac{\cE(u^{4/(n-2)}g)}{\vol(M,u^{4/(n-2)}g)^{\frac{n-2}{n}}},
\qquad \forall u\in C^\infty(M,\mR^+),
\]
and it follows that $\mu(M,g)$ has the alternative characterization
\[
\mu(M,g)
=
\inf_{\ti g\in [g]}\frac{\cE(\ti g)}{\vol(M,\ti g)^{\frac{n-2}{n}}}.
\]
Critical points of the functional $J^g$ are given by solutions of the
Yamabe equation
\[
L^g u = \mu |u|^{p-2} u
\]
for some $\mu \in \mR$. If the inequality in Theorem \ref{aubin} is
satisfied strictly, that is if $\mu(M,g)< \mu(\mS^n)$, then the
infimum in the definition of $\mu(M,g)$ is attained.

\begin{theorem}[\cite{trudinger:68,aubin:76}]\label{attained}
Let $M$ be connected. If $\mu(M,g) < \mu(\mS^n)$ then there exists a
smooth positive function $u$ with $J^g(u)=\mu$ and $\|u\|_{L^p}=1$.
This implies that $u$ solves \eqref{Yamabe_equation} with
$\mu = \mu(M,g)$. The minimizer $u$ is unique if $\mu\leq 0$.
\end{theorem}

The inequality $\mu(M,g) < \mu(\mS^n)$ was shown by Aubin
\cite{aubin:76} for non-conformally flat, compact manifolds
of dimension at least $6$. Later Schoen \cite{schoen:84} could apply
the positive mass theorem to obtain this strict inequality for all
compact manifolds not conformal to the standard sphere. We thus have a
solution of
\begin{equation} \label{Yamabe_equation}
L^g u = \mu u^{p-1},\qquad u>0.
\end{equation}

To explain the geometric meaning of these results we recall a few
facts about the Yamabe problem, see for example \cite{lee.parker:87}
and \cite[Chapter 5]{schoen.yau:94} for more details on this
material.  The name of Yamabe is associated to the problem, as
Yamabe wrote the first article about this subject \cite{yamabe:60}.

For a given compact Riemannian manifold $(M,g)$ the Yamabe
problem consists of finding a metric of constant scalar curvature
in the conformal class of~$g$. The above results yield a minimizer
$u$ for $J^g$. Equation \eqref{Yamabe_equation} is equivalent to the
fact that the scalar curvature of the metric $u^{4/(n-2)}g$ is
everywhere equal to~$\mu$. Thus, the above Theorem, together with
$\mu(M,g)<\mu(\mS^n)$, resolves the Yamabe problem.

A conformal class $[g]$ on $M$ contains a metric of positive scalar
curvature if and only if $\mu(M,[g])>0$. If $M = M_1\amalg M_2$ is a
disjoint union of $M_1$ and $M_2$ and if $g_i$ is the restriction of
$g$ to $M_i$, then an elementary argument where one rescales the
components with different factors yields that
\[
\mu(M,[g])
=
\min \left\{ \mu(M_1,[g_1]), \mu(M_2,[g_2]) \right\}
\]
if $\mu(M_1,[g_1])\geq 0$ or $\mu(M_2,[g_2])\geq 0$, and otherwise
\[
\mu(M,[g])
=
-\left(
|\mu(M_1,[g_1])|^{n/2}+|\mu(M_2,[g_2])|^{n/2}
\right)^{2/n}.
\]

One now defines the smooth Yamabe invariant of an arbitrary compact manifold $M$ of dimension at least $3$ as
\[
\si(M) \definedas \sup \mu(M,[g])\leq n (n-1)\om_n^{2/n},
\]
where the supremum is taken over all conformal classes $[g]$ on $M$.

The introduction of this invariant was originally motivated by
Yamabe's attempt to find Einstein metrics on a given compact manifold,
see \cite{schoen:87} and \cite{lebrun:99b}.
Yamabe's idea in the early 1960's was to search
for a conformal class $[g_{\rm sup}]$ that attains the supremum. The
minimizer $g_0$ of $\cE$ among all unit volume metrics in
$[g_{\rm sup}]$ exists according to Theorem~\ref{attained}, and Yamabe
hoped that the $g_0$ obtained with this minimax procedure would be a
stationary point of $\cE$ among all unit volume metrics (without fixed
conformal class), which is equivalent to $g_0$ being an Einstein
metric.

Yamabe's approach was very ambitious. If $M$ is a simply connected
compact $3$-manifold, then an Einstein metric on $M$ is necessarily a
round metric on $S^3$, hence the $3$-dimensional Poincar\'e conjecture
would follow.  It turned out, that his approach actually yields an
Einstein metric in some special cases.
For example, LeBrun \cite{lebrun:99a} showed
that if a compact $4$-dimensional manifold~$M$
carries a K\"ahler-Einstein metric
with non-positive scalar curvature,
then the supremum is attained by the conformal class of this metric.
Moreover, in any maximizing conformal class the minimizer is a
K\"ahler-Einstein metric.

Compact quotients $M=\Gamma\backslash\mH^3$ of $3$-dimensional
hyperbolic space $\mH^3$ yield other examples on which Yamabe's approach
yields an Einstein metric. On such quotients the supremum is attained
by the hyperbolic
metric on $M$. The proof of this statement
uses Perelman's proof of the Geometrization conjecture,
see \cite{anderson:06} and \cite[Section II.8]{kleiner.lott:08}.
In particular, $\si(\Gamma\backslash\mH^3)=-6
(v_\Ga)^{2/3}$ where $v_\Ga$ is the volume of $\Gamma\backslash\mH^3$
with respect to the hyperbolic metric.

On a general manifold, Yamabe's approach failed for various reasons.
In dimension $3$ and $4$ obstructions against the existence of
Einstein metrics are known today, see for example
\cite{lebrun:96,lebrun:p08}.  In many cases the supremum is not
attained.

R. Schoen and O. Kobayashi started to study the smooth Yamabe invariant
systematically in the late 1980's,
\cite{schoen:87,schoen:89,kobayashi:85,kobayashi:87}. In particular,
they determined $\si(S^{n-1}\times S^1)$ to be
$\si(S^n)=n(n-1)\om_n^{2/n}$. On $S^{n-1}\times S^1$ the supremum in
the definition of $\si$ is not attained.  Because of
Schoen's important results in these articles, the
smooth Yamabe invariant is also often called Schoen's \sidash{}constant.

The smooth Yamabe invariant determines the existence of positive
scalar curvature metrics. Namely, it follows from above that the
smooth Yamabe invariant $\si(M)$ is positive if and only if the
manifold $M$ admits a metric of positive scalar curvature.  Thus the
value of $\si(M)$ can be interpreted as a quantitative refinement of
the property of admitting a positive scalar curvature metric.

In general calculating the \sidash{}invariant is very difficult.
LeBrun \cite[Section~5]{lebrun:96}, \cite{lebrun:99a} showed that the
\sidash{}invariant of a complex algebraic surfaces is negative (resp. zero)
if and only if it is of general type (resp. of Kodaira dimension $0$ or $1$),
and the value of $\sigma(M)$ can be calculated explicitly in these cases.
As already explained above, the \sidash{}invariant can also
be calculated for hyperbolic $3$-manifolds, they are realized by the
hyperbolic metrics.

There are many manifolds admitting a
Ricci-flat metric, but no metric of positive scalar curvature,
for example tori, K3-surfaces and compact connected
$8$-dimensional manifolds admitting metrics with holonomy~$\Spin(7)$.
These conditions imply $\si(M)=0$, and the supremum is attained.

Conversely, Bourguignon showed that if $\si(M)=0$ and if the supremum
is attained by a conformal class $[g_{\rm sup}]$, then
$\cE:[g_{\rm sup}]\to \mR$ attains its minimum in a Ricci-flat metric
$g_0 \in [g_{\rm sup}]$. Thus Cheeger's splitting principle implies
topological restrictions on $M$ in this case. In particular, a compact
quotient $\Gamma\backslash N$ of a non-abelian nilpotent Lie group~$N$
does not admit metrics of non-negative scalar curvature, but it
admits a sequence of metrics $g_i$ with
$\mu(\Gamma\backslash N,g_i)\to 0$. Thus $\Gamma\backslash N$ is an
example of a manifold for which $\si(\Gamma\backslash N)=0$, for which
the supremum is not attained.

All the examples mentioned up to here have $\si(M)\leq 0$.
Positive smooth Yamabe invariants are even harder to determine.
The calculation of non-positive $\sigma(M)$ often relies on the
formula
\[
|\min\{\si(M),0\}|^{n/2}
=
\inf_g \int_M |\scal^g|^{n/2} \, dv^g
\]
where the infimum runs over all Riemannian metrics $g$ on $M$, see
\cite[Proof of Proposition 2.1]{lebrun:99b}. This formula does not
distinguish between different positive values of $\si(M)$, and thus
it cannot be used in the positive case.

It has been conjectured by Schoen \cite[Page 10, lines 6--11]{schoen:89}
that all finite quotients of round spheres satisfy
$\sigma(S^n/\Gamma)=(\#\Gamma)^{-2/n}Y(\mS^n)$, but this conjecture is
only verified for $\mathbb{R} P^3$ \cite{bray.neves:04}, namely
$\si(\mR P^3)= 6 (\om_3 /2)^{2/3}$. The smooth Yamabe invariant
%\sidash{}invariant
is also known
for connected sums of copies of real projective space~$\mR P^3$ with
copies of $S^2\times S^1$~\cite{akutagawa.neves:07},
for $\mC P^2$ \cite{gursky.lebrun:98} and for connected sums of
$\mC P^2$ with several copies of $S^3 \times S^1$. With similar methods,
it can also be determined for some related manifolds, but for example
the value of $\si(S^2\times S^2)$ is not known. To the knowledge of the
authors there are no manifolds $M$ of dimension $n \geq 5$ for which
it has been shown that $0 < \si(M)< \si(S^n)$, but due to Schoen's
conjecture finite quotients of spheres would be examples of such
manifolds.

As explicit calculation of the Yamabe invariant is difficult, it is
natural to use surgery theory to get estimates for more complicated
examples. Several articles study the behavior of the smooth Yamabe
invariant under surgery. In \cite{gromov.lawson:80b} and
\cite{schoen.yau:79c} it is proven that the existence of a positive
scalar curvature metric is preserved under surgeries of codimension
at least $3$. In terms of the \sidash{}invariant this means that if $N$
is obtained from a compact manifold $M$ by surgery of codimension at
least $3$ and $\si(M)>0$, then $\si(N)>0$.

Later Kobayashi proved in \cite{kobayashi:87} that if $N$ is obtained
from $M$ by $0$-dimensional surgery, then $\si(N)\geq \si(M)$. A first
consequence is an alternative deduction of
$\si(S^{n-1}\times S^1) = \si(S^n)$ using the fact that~$S^{n-1}~\times~S^1$
is obtained from $S^n$ by $0$-dimensional surgery.
More generally one sees that
$\si(S^{n-1} \times S^1 \# \cdots \# S^{n-1}\times S^1) = \si(S^n)$
as this connected sum is obtained from $S^n$ by $0$-dimensional
surgeries as well.

Note that it follows from what we said above that the smooth Yamabe
invariant of disjoint unions $M=M_1\amalg M_2$ satisfies
\[
\si(M)
=
\min \left\{ \si(M_1), \si(M_2) \right\}
\]
if $\si(M_1)\geq 0$ or $\si(M_2)\geq 0$, and otherwise
\[
\si(M)
=
-\left(|\si(M_1)|^{n/2}+|\si(M_2)|^{n/2}\right)^{2/n}.
\]
Kobayashi's result then implies
$\si(M_1\# M_2) \geq \si(M_1\amalg M_2)$,
and thus yields a lower bound for $\si(M_1 \# M_2)$ in terms of
$\si(M_1)$ and $\si(M_2)$.

A similar monotonicity formula for the \sidash{}invariant was proved
by Petean and Yun in \cite{petean.yun:99}. They prove that
$\si(N) \geq \min \{ \si(M), 0 \}$ if~$N$ is obtained from~$M$
by surgery of codimension at least $3$. See also
\cite[Proposition 4.1]{lebrun:99b} and~\cite{akutagawa.botvinnik:03}
for other approaches to this
result. Clearly, this surgery result is particularly interesting in
the case $\si(M) \leq 0$, and it has several fruitful applications. In
particular, any simply connected compact manifold of dimension at
least $5$ has $\si(M)\geq 0$; see \cite{petean:03}. This result has been
generalized to manifolds with certain types of fundamental group in
\cite{botvinnik.rosenberg:02}. Further results in the same spirit for $n=4$
can be found in \cite{sung:09}.

%%%%%%%%%%%%%%%%%%%%%%%%%%%%%%%%%%%%%%%%%%%%%%%%%%%%%%%%%%%%%%%%%%%%%%%%%
\subsection{Stronger version of the main result}\label{subsec.stronger}
%%%%%%%%%%%%%%%%%%%%%%%%%%%%%%%%%%%%%%%%%%%%%%%%%%%%%%%%%%%%%%%%%%%%%%%%%

In the present article we derive a surgery formula which is stronger than
the Gromov-Lawson resp.\ Schoen-Yau surgery formula, the Kobayashi surgery
formula and the Petean-Yun surgery formula described above. Suppose
that $M_1$ and $M_2$ are compact manifolds of dimension $n$ and that
$W$ is a compact manifold of dimension~$k$. Let embeddings $W
\hookrightarrow M_1$ and $W \hookrightarrow M_2$ be given. We assume
further that the normal bundles of these embeddings are trivial.
Removing tubular neighborhoods of the images of~$W$ in~$M_1$
and~$M_2$, and gluing together these manifolds along their common
boundary, we get a new compact manifold~$N$, the connected sum of~$M_1$
and~$M_2$ along~$W$. Strictly speaking $N$ also depends on the
choice of trivialization of the normal bundle. See section
\ref{joining_man} for more details.

Surgery is a special case of this construction: if $M_2=S^n$,
$W=S^k$ and if $S^k\hookrightarrow S^n$ is the standard embedding,
then $N$ is obtained from~$M_1$ via $k$-dimensional surgery along
$S^k\hookrightarrow M_1$.

\begin{theorem} \label{main.weak}
Let $M_1$ and $M_2$ be compact manifolds of dimension $n$.
If~$N$ is obtained as a connected sum of $M_1$ and $M_2$ along a
$k$-dimensional submanifold where $k \leq n-3$, then
\[
\sigma(N) \geq \min\left\{ \sigma(M_1\amalg M_2),\Lambda_{n,k}\right\}
\]
where $\Lambda_{n,k}$ is positive, and only depends on $n$ and $k$.
Furthermore $\Lambda_{n,0}=~\si(S^n)$.
\end{theorem}

{}From Theorem \ref{aubin} we know that $\si(M) \leq \si(S^n)$ and thus we see that
$\si(M \amalg S^n) = \si(M)$ for all compact $M$.  Hence, we obtain
for the special case of surgery the following corollary.

\begin{corollary}
Let $M$ be a compact manifold of dimension $n$.
Assume that~$N$ is obtained from~$M$ via surgery along a $k$-dimensional
sphere $W$, $k \leq n-3$. We then have
\[
\sigma(N) \geq \min\left\{ \sigma(M),\Lambda_{n,k}\right\}
\]
\end{corollary}

The constants $\La_{n,k}$ will be defined in Section
\ref{sec.La.const}. In Subsections \ref{subsec.La1.pos} and
\ref{subsec.La2.pos} we prove that these constants are positive, and
in Subsection \ref{subsec.La.null} we prove that
$\La_{n,0} = \mu(\mS^n)$. Explicit lower bounds for $\La_{n,k}$ can be
found for all $k\not\in \{1,n-3\}$, however these are not optimal, see
Subsection~\ref{subsec.La.k-ge-2}. An explicit calculation of~$\La_{n,k}$
for $k>0$ seems very difficult. The main problem consists
in calculating the conformal Yamabe constant of certain Riemannian
products, which in general is a hard problem. See~\cite{akutagawa.florit.petean:07,ammann.dahl.humbert:13}
for recent progress on this problem.

%%%%%%%%%%%%%%%%%%%%%%%%%%%%%%%%%%%%%%%%%%%%%%%%%%%%%%%%%%%%%%%%%%%%%%%%
\subsection{Topological applications}\label{subsec.topapp}
%%%%%%%%%%%%%%%%%%%%%%%%%%%%%%%%%%%%%%%%%%%%%%%%%%%%%%%%%%%%%%%%%%%%%%%%

The above surgery result can be combined with standard techniques of
bordism theory. As these topological applications are not the main subject
of this article, we will only sketch some typical conclusions as
examples here.

%Such applications will be the subject of a sequel to
%this article, and we will only give some typical conclusions as
%examples here.

The first corollary uses the fact that spin bordism groups and
oriented bordism groups are finitely generated together with
techniques developed for the proof of the $h$-cobordism theorem.
\begin{corollary}
For any $n\geq 5$ there is a constant $C_n>0$, depending only on $n$,
such that
\[
\si(M) \in \{0\} \cup [C_n,\si(S^n)]
\]
for any simply-connected compact manifold $M$ of dimension $n$.
\end{corollary}

We now sketch how interesting bordism invariants can be constructed
using our main result. This construction will be explained here only
for spin manifolds, but similar constructions can also be done for
oriented, non-spin manifolds or for non-oriented manifolds.

Fix a finitely presented group $\Gamma$, and let $B\Gamma$ be the
classifying space of~$\Gamma$. We consider pairs $(M,f)$ where $M$
is a compact spin manifold and where $f:M\to B\Gamma$ is continuous.
Two such pairs $(M_1,f_1)$ and $(M_2,f_2)$ are called spin bordant
over $B\Gamma$ if there exists an $(n+1)$-dimensional spin manifold
$W$ with boundary $-M_1\amalg M_2$ with a map $F:W \to B\Gamma$ such
that the restriction of~$F$ to the boundary yields $f_1$ and $f_2$.
It is implicitly required that the boundary carries the induced
orientation and spin structure and $-M_1$ denotes $M_1$ with reversed
orientation. Being spin bordant over $B\Gamma$ is an equivalence
relation. The equivalence class of $(M,f)$ under this equivalence
relation is denoted by $[M,f]$ and the set of equivalence classes
is called $\Omega_n^{\rm Spin}(B\Gamma)$. Disjoint union of manifolds
defines a sum on $\Omega_n^{\rm Spin}(B\Gamma)$ which turns it into an
abelian group.

We say that a pair $(M,f)$ with $f:M\to B\Gamma$ is a
\emph{$\pi_1$-bijective representative of $[M,f]$}
if $M$ is connected and if the induced map $f_*:\pi_1(M)\to \Gamma$ is
a bijection. Any equivalence class in $\Omega_n^{\rm Spin}(B\Gamma)$
has a $\pi_1$-bijective representative.

Now we define
\[
\Lambda_n\definedas\min\{\Lambda_{n,1},\ldots \Lambda_{n,n-3}\}>0,
\]
\[
\bar\sigma(M)\definedas\min\{\sigma(M),\Lambda_n\}.
\]

\begin{proposition}
Let $n\geq 5$.
Let $(M_1,f_1)$ and $(M_2,f_2)$ be compact spin manifolds with
maps $f_i:M_i\to B\Gamma$. If $(M_1,f_1)$ and $(M_2,f_2)$
are spin bordant over $B\Gamma$ and if $(M_2,f_2)$ is a
$\pi_1$-bijective representative of its class, then
\[
\bar\sigma(M_1)\leq\bar\sigma(M_2).
\]
\end{proposition}

We define $s_\Gamma:\Omega_n^{\rm Spin}(B\Gamma)\to \mathbb{R}$ by
\[
s_\Gamma([M,f])\definedas\sup_{(M_1,f_1)\in [M,f]}\bar\si(M_1).
\]
The proposition states $s_\Ga([M,f])= \bar\si(M)$ if $(M,f)$ is
a $\pi_1$-bijective representative of its class.
The surgery formula further implies
\[
s_\Gamma \Big([M_1,f_1]+[M_2,f_2]\Big)
\geq
\min\Big\{s_\Gamma([M_1,f_1]),s_\Gamma([M_2,f_2])\Big\}
\]
if $s_\Gamma([M_1,f_1])\geq 0$ or $s_\Gamma([M_2,f_2])\geq 0$,
and otherwise
\[
s_\Gamma\Big([M_1,f_1]+[M_2,f_2]\Big)
\geq
- \left(|s_\Gamma([M_1,f_1])|^{n/2}
+|s_\Gamma([M_2,f_2])|^{n/2}\right)^{2/n}.
\]

We conclude, and obtain the following theorem.

\begin{theorem}
Let $t\in\mR$, $t \geq 0$, $n\in \mN$, $n\geq 5$. Then the sets
\[
 \Omega_n^{\rm Spin}(B\Gamma)^{> t}
\definedas
\{[M,f]\in \Omega_n^{\rm Spin}(B\Gamma)\,|\, s_\Gamma([M,f])> t\}
\]
and
\[
 \Omega_n^{\rm Spin}(B\Gamma)^{\geq t}
\definedas
\{[M,f]\in \Omega_n^{\rm Spin}(B\Gamma)\,|\, s_\Gamma([M,f])\geq t\}
\]
are subgroups of $\Omega_n^{\rm Spin}(B\Gamma)$.
\end{theorem}

The theorem admits --- among other interesting conclusions --- the
following application. For a positive integer $p$ we write $p\#M$ for
the connected sum $M\#\cdots \#M$ where $M$ appears $p$ times. We
already know that $\sigma(p\#M) \geq \sigma(M)$ if $\sigma(M)\geq 0$.

\begin{corollary}
Suppose that $M$ is a compact spin
manifold of dimension at least $5$
with $\sigma(M)\in (0,\Lambda_n)$. Let $p$ and $q$ be two
relatively prime positive integers. If $\sigma(p\#M)>\sigma(M)$,
then $\sigma(q\#M)=\sigma(M)$.
\end{corollary}

If Schoen's conjecture about the \sidash{}invariant of quotients of spheres
holds true, then quotients of spheres
by large fundamental groups yield examples of manifolds~$M$
with $\si(M)\in (0,\Lambda_n)$.

%%%%%%%%%%%%%%%%%%%%%%%
%Setting
%$\bar{\si}(M) \definedas \min\{\si(M),\La_{n,1},\ldots,\La_{n,n-3}\}$
%one sees that $\bar{\si}(M)$ is a bordism invariant, where the precise
%meaning of this expression depends on some topological properties of
%the manifold $M$. For example $\bar{\si}(M)$ is a spin-bordism
%invariant of simply connected spin manifolds of dimension $\geq 5$. It
%is an oriented bordism invariant of simply connected oriented non-spin
%manifolds of dimension $\geq 5$. Non-simply connected manifolds can be
%dealt with by considering bordisms with maps to $B\pi_1(M)$.

The determination of manifolds admitting positive scalar curvature
metrics, that is manifolds with $\si(M)>0$, has led to interesting
results and challenging problems in topology, see for example
\cite{rosenberg.stolz:01}. It would be interesting to develop similar
topological tools for manifolds with $\si(M)>\ep$ for $\ep>0$.
As explained above such manifolds form a subgroup on the bordism
level.
%In particular, it would be interesting to find on the bordism
%level a ring structure on manifolds with $\si(M^n)>\ep_n$ where
%$\ep_n>0$ is a given sequence of positive numbers, generalizing the ring
%structure on positive scalar curvature bordism classes.

The subgroups $\Omega_n^{\rm Spin}(B\Gamma)^{> t}$ also provide
interesting algebraic
structures. Any homomorphism $\Gamma_1\to \Gamma_2$ provides
a homomorphism
  $$\Omega_n^{\rm Spin}(B\Gamma_1)^{>t}\to \Omega_n^{\rm Spin}(B\Gamma_2)^{>t}.$$
After introducing some factors and powers depending on the dimension,
these subgroups carry an ideal-like structure. More precisely, it follows
from \cite{ammann.dahl.humbert:13} that for any numbers $t_5>0$ there is a
sequence $t_n>0$, $n\geq 5$, such that 
taking products of manifolds
defines a $\mZ$-bilinear map
  $$\Omega_{*\geq 5}^{\rm Spin}(B\Gamma_1)\times \Omega_{*\geq 5}^{\rm Spin}(B\Gamma_2)^{> t_*}\to \Omega_{*\geq 5}^{\rm Spin}(B(\Gamma_1\times \Gamma_2))^{> t_*}$$
where the index $*\geq 5$ indicates that we consider the spin bordism
ring of manifolds whose dimension is at least $5$.
In particular $\Omega_{*\geq 5}^{\rm Spin}(B\Gamma)^{> t_*}$ is a module
over the ring $\Omega_{*\geq 5}^{\rm Spin}:=\Omega_{*\geq 5}^{\rm Spin}(B\{1\})$
and  $\Omega_{*\geq 5}^{{\rm Spin}\;> t_*}$ is an ideal in
$\Omega_{*\geq 5}^{\rm Spin}$.
Analogous structures exist for $\Omega_n^{\rm Spin}(B\Gamma)^{\geq  t}$.

%%%%%%%%%%%%%%%%%%%%%%%%%%%%%%%%%%%%%%%%%%%%%%%%%%%%%%%%%%%%%%%%%%%%%%%%%
\subsection{Comparison to other results}
%%%%%%%%%%%%%%%%%%%%%%%%%%%%%%%%%%%%%%%%%%%%%%%%%%%%%%%%%%%%%%%%%%%%%%%%%

At the end of the section we want to mention some similar constructions
in the literature.
An analogous surgery formula holds if we replace the Conformal
Laplacian by the Dirac operator, see~\cite{ammann.dahl.humbert:09b}
for details and applications. D.~Joyce \cite{joyce:03}, followed by
L.~Mazzieri \cite{mazzieri:08,mazzieri:p06b}, considered a problem
tightly related to our result: their goal is to construct a metric on
a manifold obtained via a connected sum along a $k$-dimensional submanifold. For these metrics they construct a solution of the Yamabe
equation on the new manifold that is close to solutions of the Yamabe
equations on the original pieces. Such a construction was achieved by
D.~Joyce for $k=0$ and by L.~Mazzieri for $k\in \{1,\ldots,n-3\}$
provided that the embeddings defining the connected sum are isometric.
In contrast to our article their solutions on the new manifold are not
necessarily minimizers of the volume-normalized Einstein-Hilbert
functional. Similar constructions have also been developed by
R.~Mazzeo, D.~Pollack, and K.~Uhlenbeck
\cite{mazzeo.pollack.uhlenbeck:95} in order to glue together metrics
of constant scalar curvature. Recently, J.~Corvino, M.~Eichmair, and
P.~Miao showed how to glue together metrics while preserving constant
scalar curvature and volume, see \cite{corvino.eichmair.miao:p12}.
Further, P.~T.~Chrusciel, J.~Isenberg, and D.~Pollack
\cite{chrusciel.isenberg.pollack:05}, found methods to glue together
solutions of the vacuum Einstein constraint equations.

Other authors studied the equivariant analogues. In this setting
one assumes that a compact Lie group $G$ acts on the manifolds
before and after surgery and that the surgery is compatible with this actions.
Furthermore all metrics are assumed to be $G$-invariant and the
Yamabe constant and Yamabe invariant are replaced by their
equivariant analogues. The equivariant Yamabe problem is solved in many cases,
in particular on spin manifolds or in the case that all orbits have
positive dimension, see \cite{hebey.vaugon:93}, \cite{madani:10,madani:12}.
An equivariant analogue of the Petean-Yun surgery formula was provided
in \cite{sung:06}. B.~Hanke \cite{hanke:08} proved that the existence of
$G$-invariant positive scalar curvature metrics is preserved under
equivariant surgeries of the appropriate dimensions, which is the
equivariant generalization of the result by Gromov and Lawson,
respectively Schoen and Yau, cited above.

%%%%%%%%%%%%%%%%%%%%%%%%%%%%%%%%%%%%%%%%%%%%%%%%%%%%%%%%%%%%%%%%%%%%%%%%%
\subsection*{Acknowledgments}
%%%%%%%%%%%%%%%%%%%%%%%%%%%%%%%%%%%%%%%%%%%%%%%%%%%%%%%%%%%%%%%%%%%%%%%%%

The authors want to thank the Max Planck Institute for Gravitational
Physics in Potsdam for its hospitality, its support, and its friendly
working conditions which had an important impact on this article.
We thank Andreas Hermann for the numerical computation of the
equivariant version of $\La_{4,1}$ mentioned in Section
\ref{subsec.La.speculation}. We also thank Kazuo Akutagawa for interesting
discussions and insightful comments. Finally, we want to express our
deep thanks to the anonymous referee whose many valuable remarks have
helped us greatly to improve the paper.

%%%%%%%%%%%%%%%%%%%%%%%%%%%%%%%%%%%%%%%%%%%%%%%%%%%%%%%%%%%%%%%%%%%%%%%%%
\section{The connected sum along a submanifold}
\label{joining_man}
%%%%%%%%%%%%%%%%%%%%%%%%%%%%%%%%%%%%%%%%%%%%%%%%%%%%%%%%%%%%%%%%%%%%%%%%%

In this section we are going to describe how two manifolds are joined
along a common submanifold with trivialized normal bundle. Strictly
speaking this is a differential topological construction, but since we
work with Riemannian manifolds we will make the construction adapted
to the Riemannian metrics and use distance neighborhoods defined by
the metrics etc.

Let $(M_1,g_1)$ and $(M_2,g_2)$ be complete Riemannian manifolds of
dimension~$n$. Let $W$ be a compact manifold of dimension~$k$, where
$0 \leq k \leq n$. Let $\bar{w}_i: W \times \mR^{n-k} \to TM_i$,
$i=1,2$, be smooth embeddings. We assume that $\bar{w}_i$ restricted
to $W \times \{ 0 \}$ maps to the zero section of $TM_i$ (which we
identify with~$M_i$) and thus gives an embedding $W \to M_i$. The
image of this embedding is denoted by $W_i'$. Further we assume that
$\bar{w}_i$ restrict to linear isomorphisms $\{ p \} \times \mR^{n-k}
\to N_{\bar{w}_i(p,0)} W_i'$ for all $p \in W_i$, where $N W_i'$
denotes the normal bundle of $W_i'$ defined using $g_i$.

By setting $w_i \definedas \exp^{g_i} \circ \bar{w}_i$ we obtain the
embeddings $w_i: W \times B^{n-k}(\Rmax)\to M_i$ for some
$\Rmax > 0$ and $i=1,2$. We have $W_i' = w_i(W \times \{ 0 \})$ and we
define the disjoint union
\[
(M,g) \definedas (M_1 \amalg M_2, g_1 \amalg g_2),
\]
and
\[
W' \definedas W_1' \amalg W_2'.
\]
Let $r_i$ be the function on $M_i$ giving the distance to $W_i'$.
Then $r_1 \circ w_1 (p,x) = r_2 \circ w_2(p,x) = |x|$ for $p \in W$,
$x \in B^{n-k}(\Rmax)$. Let $r$ be the function on $M$ defined by
$r(x) \definedas r_i(x)$ for $x \in M_i$, $i=1,2$. For $0 < \ep$ we
set $U_i(\ep) \definedas \{ x \in M_i \, : \, r_i(x) < \ep \}$ and
$U(\ep) \definedas U_1(\ep) \cup U_2(\ep)$. For $0 < \ep < \th$ we
define
\[
N_{\ep}
\definedas
( M_1 \setminus U_1(\ep) ) \cup ( M_2 \setminus U_2(\ep) )/ {\sim},
\]
and
\[
U^N_\ep (\th)
\definedas
(U(\th) \setminus U(\ep)) / {\sim}
\]
where ${\sim}$ indicates that we identify the point 
$x \in \partial U_1(\ep)$ in~$M_1$
with the corresponding point $w_2 \circ w_1^{-1} (x) \in \partial U_2(\ep)$
in~$M_2$. Hence
\[
N_\ep
=
(M \setminus U(\th) ) \cup U^N_\ep (\th).
\]

We say that $N_\ep$ is obtained from $M_1$, $M_2$ (and $\bar{w}_1$,
$\bar{w}_2$) by a connected sum along $W$ with parameter $\ep$.

The diffeomorphism type of $N_\ep$ is independent of $\ep$, hence we
will usually write $N = N_\ep$. However, in situations when dropping
the index causes ambiguities we will keep the notation $N_\ep$. For
example the function $r: M \to [0,\infty)$ gives a continuous function
$r_\ep: N_\ep \to [\ep, \infty)$ whose domain depends on $\ep$. It is
also going to be important to keep track of the subscript $\ep$ on
$U^N_\ep (\th)$ since crucial estimates on solutions of the Yamabe
equation will be carried out on this set.

The surgery operation on a manifold is a special case of taking
connected sum along a submanifold. Indeed, let $M$ be a compact
manifold of dimension~$n$ and let $M_1 = M$, $M_2 = S^n$,
$W = S^k$. Let $w_1 : S^k \times B^{n-k} \to M$ be an embedding
defining a surgery and let $w_2 :  S^k \times B^{n-k} \to S^n$ be the
standard embedding. Since $S^n \setminus w_2 (S^k \times B^{n-k})$ is
diffeomorphic to $B^{k+1} \times S^{n-k-1}$ we have in this situation
that $N$ is obtained from $M$ by performing surgery on~$w_1$, see
\cite[Section VI, 9]{kosinski:93}.

%%%%%%%%%%%%%%%%%%%%%%%%%%%%%%%%%%%%%%%%%%%%%%%%%%%%%%%%%%%%%%%%%%%%%%%%%
\section{The constants $\La_{n,k}$} \label{sec.La.const}
%%%%%%%%%%%%%%%%%%%%%%%%%%%%%%%%%%%%%%%%%%%%%%%%%%%%%%%%%%%%%%%%%%%%%%%%%

In Section \ref{sec.history} we defined the conformal Yamabe constant
only for compact manifolds. There are several ways to generalize the
conformal Yamabe constant to non-compact manifolds.  In this section
we define two such generalizations $\mukim$ and $\mu^{(1)}$, and also
introduce a related quantity called $\mu^{(2)}$. These invariants
will be needed to define the constants $\La_{n,k}$ and to prove their
positivity on our model spaces $\mH^{k+1}_c\times \mS^{n-k-1}$.

The definition of $\mu^{(2)}$ comes from a technical difficulty in the
proof of Theorem~\ref{main.strong} and is only relevant in the case
$k = n-3 \geq 3$, see Remark~\ref{rem.twoone}.

%%%%%%%%%%%%%%%%%%%%%%%%%%%%%%%%%%%%%%%%%%%%%%%%%%%%%%%%%%%%%%%%%%%%%%%%%
\subsection{The manifolds $\mH^{k+1}_c\times \mS^{n-k-1}$}
%%%%%%%%%%%%%%%%%%%%%%%%%%%%%%%%%%%%%%%%%%%%%%%%%%%%%%%%%%%%%%%%%%%%%%%%%

For $0\leq k<n$ and $c \in \mR$ we define the metric
$\eta^{k+1}_c \definedas e^{2 c t} \xi^k + dt^2$ on
$\mR^k \times \mR$ and we write
\[
\mH^{k+1}_c
\definedas
(\mR^k \times \mR, \eta^{k+1}_c) .
\]
This is a model of the simply connected complete manifold
of constant curvature~$-c^2$. We denote by
\[
G_c \definedas \eta^{k+1}_c + \sigma^{n-k-1}
\]
the product metric on $\mH^{k+1}_c \times \mS^{n-k-1}$.
The scalar curvature of $\mH^{k+1}_c\times \mS^{n-k-1}$ is
$\scal^{G_c} = -k(k+1) c^2+(n-k-1)(n-k-2)$.

\begin{proposition} \label{HconfS}
$\mH_1^{k+1}\times\mS^{n-k-1}$ is conformal to
$\mS^n \setminus \mS^k$.
\end{proposition}

\begin{proof}
Let $\mS^k$ be embedded in $\mS^n\subset \mR^{n+1}$ by setting the
last $n-k$ coordinates to zero and let $s \definedas d(\cdot,\mS^k)$
be the intrinsic distance to~$\mS^k$ in $\mS^n$. Then the function
$\sin s$ is smooth
and positive on $S^n\setminus S^k$. The points of maximal distance
$\pi/2$ to $\mS^k$ lie on an $(n-k-1)$-sphere, denoted by
$(\mS^{k})^{\perp}$. On $\mS^n\setminus (\mS^k\cup (\mS^{k})^\perp)$
the round metric is
\[
\si^n
=
(\cos s)^2 \si^k + ds^2 + (\sin s)^2 \si^{n-k-1}.
\]
Substitute $s\in (0,\pi/2)$ by $t\in(0,\infty)$ such that
$\sinh t = \cot s$. Then $\cosh t = (\sin s)^{-1}$ and
$\cosh t\,dt= -(\sin s)^{-2}\,ds$, so $\si^n$ is conformal to
\[
(\sin s)^{-2} \si^n
=
(\sinh t)^2 \si^k + dt^2 + \si^{n-k-1}.
\]
Here we see that the first two terms give a metric
\[
(\sinh t)^2 \si^k + dt^2
\]
on $S^k \times (0,\infty)$. This is just
the standard metric on $\mH_1^{k+1} \setminus \{p_0\}$
where $t = d( \cdot, p_0)$, written in polar normal coordinates.
In the case $k\geq 1$ it is evident that the conformal
diffeomorphism
$\mS^n\setminus(\mS^k\cup (\mS^k)^\perp)\to(\mH_1^{k+1}
\setminus \{p_0\})\times \mS^{n-k-1}$ extends to a conformal diffeomorphism
$\mS^n\setminus\mS^k\to \mH_1^{k+1} \times \mS^{n-k-1}$.

In the case $k=0$ we equip $s$ and $t$ with a sign, that is we let
$s>0$ and $t>0$ on one of the components of
$\mS^n \setminus(\mS^0 \cup (\mS^0)^\perp)$, and $s<0$ and $t<0$ on
the other component. The functions $s$ and $t$ are then smooth on
$\mS^n\setminus \mS^0$ and take values $s\in (-\pi/2,\pi/2)$ and $t\in
\mR$. Then the argument is the same as above.
\end{proof}

%%%%%%%%%%%%%%%%%%%%%%%%%%%%%%%%%%%%%%%%%%%%%%%%%%%%%%%%%%%%%%%%%%%%%%%%%
\subsection{Definition of $\La_{n,k}$ } %\label{subsec.def.La}
%%%%%%%%%%%%%%%%%%%%%%%%%%%%%%%%%%%%%%%%%%%%%%%%%%%%%%%%%%%%%%%%%%%%%%%%%

Let $(N,h)$ be a Riemannian manifold of dimension~$n$. For $i=1,2$ we
let $\Omega^{(i)}(N,h)$ be the set of non-negative $C^2$ functions $u$
that solve the Yamabe equation
\begin{equation} \label{eq.conf}
L^h u = \mu u^{p-1}
\end{equation}
for some $\mu = \mu(u) \in \mR$ and satisfy
\begin{itemize}
\item $u \not \equiv 0$,
\item $\|u\|_{L^p(N)} \leq 1$,
\item $u \in L^{\infty}(N)$,
\end{itemize}
together with
\begin{itemize}
\item $u \in L^2(N)$, for $i=1$,
\end{itemize}
or
\begin{itemize}
\item $\mu(u) \|u\|^{p-2}_{L^{\infty}(N)} \geq
 \frac{(n-k-2)^2(n-1)}{8(n-2)}$, for $i=2$.
\end{itemize}
For $i=1,2$ we set
\[
\mu^{(i)} (N,h) \definedas \inf_{u \in \Omega^{(i)}(N,h)} \mu(u).
\]
In particular, if $\Omega^{(i)}(N,h)$ is empty then
$\mu^{(i)}(N,h)=\infty$.

\begin{definition}
For integers $n \geq 3$ and $0 \leq k \leq n-2$ let
\[
\La^{(i)}_{n,k}
\definedas
\inf_{c \in [-1,1]}
\mu^{(i)} (\mH^{k+1}_c\times \mS^{n-k-1})
\]
and
\[
\La_{n,k}
\definedas
\min \left\{ \La^{(1)}_{n,k},\La^{(2)}_{n,k}\right\}.
\]
\end{definition}

Note that the infimum could just as well be taken over $c \in [0,1]$
since $\mH_{c}^{k+1}\times\mS^{n-k-1}$ and
$\mH_{-c}^{k+1}\times\mS^{n-k-1}$ are isometric. We are going to
prove that these constants are positive.
\begin{theorem} \label{La>0}
For all $n \geq 3$ and $0 \leq k \leq n-3$, we have
$\La_{n,k} >0$.
\end{theorem}
The condition $k \leq n-3$ is important, as this implies that $\mS^{n-k-1}$
has positive curvature.

To prove Theorem \ref{La>0} we have to prove that $\La^{(1)}_{n,k}> 0$
and that $\La^{(2)}_{n,k} > 0$. This is the object of the following two
subsections. In the final subsection we prove that
$\La_{n,0}=\mu(\mS^n)=n (n-1)\om_n^{2/n}$.

\begin{remark} \label{rem.twoone}
Suppose that either $k\leq n-4$ or $k=n-3\leq 2$. One can then use
methods similar to those used in Section \ref{wsbundles} to show that
any $L^p$-solution of \eqref{eq.conf} on the model spaces is also an
$L^2$-solution, see~\cite{ammann.dahl.humbert:p11b}. An analogous
argument also works in the case $(n,k)=(6,3)$, for model spaces with
$c<1$, and this allows similar conclusions, see
\cite{ammann.dahl.humbert:p12}. This implies that
$\La_{n,k}^{(2)} \geq \La_{n,k}^{(1)}$ if $k\leq n-4$ or $k=n-3\leq 3$,
and hence
\[
\La_{n,k}
=
\La_{n,k}^{(1)}.
\]
In the case $k=n-3\geq 4$ there are $L^p$-solutions of \eqref{eq.conf}
on $\mH_1^{k+1}\times \mS^{n-k-1}$ that are not $L^2$-solutions.
\end{remark}

%%%%%%%%%%%%%%%%%%%%%%%%%%%%%%%%%%%%%%%%%%%%%%%%%%%%%%%%%%%%%%%%%%%%%%%%%
\subsection{Proof of $\La^{(1)}_{n,k}>0$}  \label{subsec.La1.pos}
%%%%%%%%%%%%%%%%%%%%%%%%%%%%%%%%%%%%%%%%%%%%%%%%%%%%%%%%%%%%%%%%%%%%%%%%%

The proof proceeds in several steps. We first introduce a conformal
Yamabe constant for non-compact manifolds and show that it gives a
lower bound for $\mu^{(1)}$. We then conclude by studying this
conformal invariant.

Let $(N,h)$ be a Riemannian manifold that is not necessarily compact
or complete. We define the conformal Yamabe constant $\mukim$ of
$(N,h)$ following Schoen-Yau \cite[Section~2]{schoen.yau:88},
see also \cite{kim:00}, as
\[
\mukim(N,h) \definedas \inf J^h(u)
\]
where $J^h$ is defined in \eqref{def.J^g} and the infimum runs over the
set of all non-zero compactly supported smooth functions $u$ on
$N$. If $h$ and $\tilde{h}$ are conformal metrics on $N$ it follows
from \eqref{confJ} that $\mukim(N,h) = \mukim(N,\tilde{h})$.
\begin{lemma} \label{lemma1}
Let $0\leq k \leq n-2$. Then
\[
\mu^{(1)}(\mH_c^{k+1}\times\mS^{n-k-1})
\geq
\mukim(\mH_c^{k+1}\times\mS^{n-k-1})
\]
for all $c \in \mR$.
\end{lemma}
\begin{proof}
Suppose that $u \in \Omega^{(1)} (\mH_c^{k+1} \times \mS^{n-k-1})$
is a solution of \eqref{eq.conf} on $\mH_c^{k+1}\times\mS^{n-k-1}$
with $\mu = \mu(u) \in [\mu^{(1)}(\mH_c^{k+1}\times\mS^{n-k-1}),
\mu^{(1)}(\mH_c^{k+1}\times\mS^{n-k-1})+\ep]$.  Let~$\chi_\al$ be a
cut-off function on $\mH_c^{k+1}\times\mS^{n-k-1}$ depending only on
the distance $r$ to a fixed point, such that $\chi_\al(r) = 1$ for $r
\leq \al$, $\chi_\al(r) = 0$ for $r \geq \al+2$, and $| d\chi_\al |
\leq 1$. We are going to see that
\begin{equation} \label{limJ}
\begin{split}
\mukim(\mH_c^{k+1}\times\mS^{n-k-1})
&\leq
\lim_{\al \to \infty} J^{G_c}(\chi_\al u) \\
&=
\mu \| u \|_{L^p(\mH_c^{k+1}\times\mS^{n-k-1})}^{p-2} \\
&\leq
\mu \\
&\leq \mu^{(1)}(\mH_c^{k+1}\times\mS^{n-k-1}) + \ep.
\end{split}
\end{equation}
Integrating by parts and using Equations \eqref{eq.conf} and
\eqref{formula.dchiu} we get
\begin{equation*}
\begin{split}
\int_{ \mH_c^{k+1}\times\mS^{n-k-1}}
(\chi_\al u)  L^{G_c}(\chi_\al u) \, dv^{G_c}
&=
\int_{ \mH_c^{k+1}\times\mS^{n-k-1}}
\chi_\al^2 u L^{G_c} u \, dv^{G_c} \\
&\quad
+ \an \int_{ \mH_c^{k+1}\times\mS^{n-k-1}}
|d\chi_\al|^2 u^2 \, dv^{G_c} \\
&=
\mu \int_{ \mH_c^{k+1}\times\mS^{n-k-1}}
\chi_\al^2 u^p \, dv^{G_c} \\
&\quad
+ \an \int_{ \operatorname{Supp}(d\chi_\al)}
|d\chi_\al|^2 u^2 \, dv^{G_c}.
\end{split}
\end{equation*}
Since $u \in L^2( \mH_c^{k+1}\times\mS^{n-k-1})$ and
$|d \chi_\al| \leq 1$ the last integral
goes to zero as $\al\to \infty$ and we conclude that
\[
\lim_{\al\to \infty}  \int_{ \mH_c^{k+1}\times\mS^{n-k-1}}
(\chi_\al u)  L^{G_c}(\chi_\al u) \, dv^{G_c}
=
\mu \| u \|_{L^p(\mH_c^{k+1}\times\mS^{n-k-1})}^p.
\]
Going back to the definition of $J^{G_c}$ we easily get \eqref{limJ},
and Lemma~\ref{lemma1} follows.
\end{proof}

\begin{remark}\label{rem.equal}
It follows from~\cite[Theorem~13]{grosse:p09} and a straight-forward
cut-off argument that
\[
\mu^{(1)}(\mH_c^{k+1}\times\mS^{n-k-1})
=
\mukim(\mH_c^{k+1}\times\mS^{n-k-1})
\]
if the space $\mH_c^{k+1}\times\mS^{n-k-1}$ has
positive scalar curvature, i.\thinspace e.\ if we have
$(n-k-1)(n-k-2)>c^2k(k+1)$.
\end{remark}

We define
\[
\La^{(0)}_{n,k}
\definedas
\inf_{c \in [-1,1]} \mukim(\mH_c^{k+1}\times\mS^{n-k-1}).
\]
Then Lemma \ref{lemma1} tells us that $\La^{(1)}_{n,k} \geq
\La^{(0)}_{n,k}$, so we are done if we prove that
$\La^{(0)}_{n,k}>0$. To do this we need two lemmas.

\begin{lemma} \label{lemma2}
Let $0\leq k\leq n-2$. Then
\[
\mukim(\mH_1^{k+1}\times\mS^{n-k-1})
=
\mu(\mS^n).
\]
\end{lemma}
\begin{proof}
The inequality $\mukim(\mH_1^{k+1}\times\mS^{n-k-1}) \leq \mu(\mS^n)$
is completely analogous to \cite[Lemma 3]{aubin:76}. As we do not need
this inequality later, we skip the proof.
To prove the opposite inequality
$\mukim(\mH_1^{k+1}\times\mS^{n-k-1}) \geq \mu(\mS^n)$ we use
Proposition \ref{HconfS} and the conformal invariance of
$\mukim$, and we obtain
\[
\mukim(\mH_1^{k+1}\times\mS^{n-k-1})
=
\mukim(\mS^n \setminus \mS^k).
\]

Clearly $\mukim(\mS^n \setminus \mS^k) \geq \mu(\mS^n)$ as the infimum
defining the left hand side runs over a smaller set of functions,
see \cite[Lemma~2.1]{schoen.yau:88}.
\end{proof}

\begin{lemma} \label{lemma3}
Let $0 \leq k \leq n-2$ and $0 < c_0 \leq c_1$. Then
\[
\mukim(\mH_{c_0}^{k+1}\times\mS^{n-k-1})
\geq
\left(\frac{c_0}{c_1}\right)^{\frac{2(n-k-1)}n}
\mukim(\mH_{c_1}^{k+1}\times\mS^{n-k-1}) .
\]
\end{lemma}
\begin{proof}
Let $c > 0$. Setting $s = c t + \ln c$ we see that
\[
G_c
=
e^{2ct} \xi^k +  dt^2 + \si^{n-k-1}
=
\frac{1}{c^2} \left( e^{2s} \xi^k +  ds^2 \right) + \si^{n-k-1}.
\]
Hence $G_c$ is conformal to the metric
\[
\tilde{G}_c
\definedas
e^{2s} \xi^k +  ds^2 + c^2 \si^{n-k-1}
\]
and by the conformal invariance of $\mukim$ we get that
\[
\mukim( \mH_{c_i}^{k+1} \times \mS^{n-k-1})
=
\mukim (\mR^k \times \mR \times S^{n-k-1},\tilde{G}_{c_i})
\]
for $i = 0,1$. In these coordinates we easily compute that
$\scal^{\tilde{G}_{c_0}} \geq \scal^{\tilde{G}_{c_1}}$,
$|du|^2_{\tilde{G}_{c_0}} \geq |du|^2_{\tilde{G}_{c_1}}$, and
$dv^{\tilde{G}_{c_0}} = \left( \frac{c_0}{c_1} \right)^{n-k-1}
dv^{\tilde{G}_{c_1}}$. We conclude that
\[
J^{\tilde{G}_{c_0}}(u)
\geq
\left( \frac{c_0}{c_1} \right)^{\frac{2(n-k-1)}n}
J^{\tilde{G}_{c_1}}(u)
\]
for all functions $u$ on $\mR^k \times \mR \times S^{n-k-1}$ and Lemma
\ref{lemma3} follows.
\end{proof}

If we set $c_1 = 1$ and use Lemma \ref{lemma2} together with
\eqref{mu(S^n)} we get the following result.
\begin{cor}
For $0\leq k\leq n-2$ and $c_0 > 0$ we have
\[
\inf_{c\in[c_0,1]}
\mukim(\mH_{c}^{k+1}\times\mS^{n-k-1})
\geq
n(n-1)\,{\om_n}^{2/n} {c_0}^{4/n}.
\]
\end{cor}

Finally, we are ready to prove that $\La^{(0)}_{n,k}$ is
positive.
\begin{theorem} \label{la3>0}
Let $0 \leq k \leq n-3$. Then $\La^{(0)}_{n,k} > 0$.
\end{theorem}

For this theorem the restriction $k\leq n-3$ is necessary. The proof
needs the positive scalar curvature of $\mS^{n-k-1}$, and it
can be shown that the theorem no longer holds for $k=n-2$.

\begin{proof}
Choose $c_0 > 0$ small enough so that
$\scal^{G_{c_0}} >0$. We then have
$\scal^{G_c} \geq \scal^{G_{c_0}}$ for all $c \in [0,c_0]$.
Hence
\[
\mukim(\mH_{c}^{k+1}\times\mS^{n-k-1})
\geq
\inf
\frac{\int_{\mH_{c}^{k+1}\times\mS^{n-k-1} }
\left( \an |du|_{G_{c}}^2 + \scal^{G_{c_0}} u^2 \right)
\, dv^{G_{c}}}
{\| u \|^2_{L^p( \mH_{c}^{k+1}\times\mS^{n-k-1})}}
\]
where the infimum is taken over all non-zero smooth functions $u$
with compact support. By Hebey \cite[Theorem 4.6, page 64]{hebey:96},
there exists a constant $A >0$ such that for all $c \in [0,c_0]$ and
all smooth non-zero functions $u$ compactly supported in
$\mH_{c}^{k+1}\times\mS^{n-k-1}$ we have
\[
\| u \|_{L^p( \mH_{c}^{k+1}\times\mS^{n-k-1})}^2
\leq
A \int_{\mH_{c}^{k+1}\times\mS^{n-k-1} }
\left( |du|_{G_{c}}^2 + u^2 \right) \, dv^{G_{c}}.
\]
This implies that
\[
\mukim (\mH_{c}^{k+1}\times\mS^{n-k-1})
\geq
\frac{1}{A} \min \left\{ \an, \scal^{G_{c_0}} \right\} > 0
\]
for all $c \in [0,c_0]$, and together with Lemma \ref{lemma3} we
obtain that
\[
\inf_{c \in [0,1] }
\mukim (\mH_{c}^{k+1}\times\mS^{n-k-1})
> 0.
\]
Since $\mH_{c}^{k+1}\times\mS^{n-k-1}$ and
$\mH_{-c}^{k+1}\times\mS^{n-k-1}$ are isometric we have
\[
\La^{(0)}_{n,k}
=
\inf_{c \in [-1,1] }
\mukim(\mH_{c}^{k+1}\times\mS^{n-k-1})
> 0.
\]
This ends the proof of Theorem \ref{la3>0}.
\end{proof}

As an immediate consequence we obtain that $\La^{(1)}_{n,k}$ is
positive.
\begin{cor}
Let $0 \leq k \leq n-3$. Then $\La^{(1)}_{n,k} > 0$.
\end{cor}

%%%%%%%%%%%%%%%%%%%%%%%%%%%%%%%%%%%%%%%%%%%%%%%%%%%%%%%%%%%%%%%%%%%%%%%%%
\subsection{Proof of $\La^{(2)}_{n,k}>0$} \label{subsec.La2.pos}
%%%%%%%%%%%%%%%%%%%%%%%%%%%%%%%%%%%%%%%%%%%%%%%%%%%%%%%%%%%%%%%%%%%%%%%%%

\begin{theorem}
Let $0\leq k\leq n-3$. Then $\La^{(2)}_{n,k} >0$.
\end{theorem}

\begin{proof}
We prove this by contradiction. Assume that there exists a sequence
$(c_i)$ of $c_i \in [-1,1]$ for which
$\mu_i \definedas \mu^{(2)}(\mH^{k+1}_{c_i}\times \mS^{n-k-1})$
tends to a limit $l \leq 0 $ as $i \to \infty$. After removing the
indices $i$ for which $\mu_i$ is infinite we get for every~$i$ a
positive solution $u_i \in \Omega^2(\mH^{k+1}_{c_i}\times \mS^{n-k-1})$
of the equation
\[
L^{G_{c_i}} u_i = \mu_i u_i^{p-1} .
\]
By definition of $\Omega^{(2)}(\mH^{k+1}_{c_i}\times \mS^{n-k-1})$
we have
\begin{equation} \label{uilai}
\frac{(n-k-2)^2(n-1)}{8(n-2)}
\leq
\mu_i  \| u_i\|_{L^\infty}^{p-2},
\end{equation}
which implies that $\mu_i > 0$. We conclude that
$l \definedas \lim_i \mu_i = 0$. We cannot assume that
$ \| u_i \|_{L^\infty}$ is attained but we can choose points
$x_i \in  \mH^{k+1}_{c_i} \times \mS^{n-k-1}$ such that
$u_i(x_i) \geq \frac{1}{2}\| u_i\|_{L^\infty}$. Moreover, we can
compose the functions $u_i$ with isometries so that all the $x_i$ are
the same point $x$. From \eqref{uilai} we get
\[
\frac{1}{2}
\left( \frac{(n-k-2)^2(n-1)}{8(n-2)\mu_i} \right)^{\frac{1}{p-2}}
\leq
u_i(x).
\]
We define $m_i \definedas u_i(x)$. Since
$\lim_{i \to \infty} \mu_i = 0$ we have
$\lim_{i \to \infty} m_i = \infty$. Restricting to a subsequence we
can assume that $c \definedas \lim_i c_i\in [-1,1]$ exists.
Define $\tilde{g}_i \definedas m_i^\frac{4}{n-2} G_{c_i}$.  We
apply Lemma \ref{diffeom} with $\al = 1/i$,
$(V, \ga_\al) = \mH^{k+1}_{c_i}\times \mS^{n-k-1}$,
$(V, \ga_0) = \mH^{k+1}_{c}\times \mS^{n-k-1}$,
$q_\al = x_i = x$, and $b_\al = m_i^\frac{2}{n-2}$.
For $r>0$ we obtain diffeomorphisms
\[
\Th_i :
B^n(r)
\to
B^{G_{c_i}} (x, m_i^{-\frac{2}{n-2}}r)
\]
such that the sequence $\Th_i^* (\tilde{g}_i)$ tends to the flat
metric $\xi^n$ on $B^n(r)$. We let
$\tilde{u}_i \definedas m_i^{-1} u_i$. By \eqref{confL} we then have
\[
L^{\tilde{g}_i} \tilde{u}_i = \mu_i {\tilde{u}_i}^{p-1}
\]
on $B^{G_{c_i}} ( x_i, m_i^{-\frac{2}{n-2}} r)$ and
\begin{equation*}
\begin{split}
\int_ {B^{G_{c_i}} ( x_i, m_i^{-\frac{2}{n-2}} r)}
{\tilde{u}_i}^p \,dv^{\tilde{g}_i}
&=
\int_{B^{G_{c_i}} ( x_i, m_i^{-\frac{2}{n-2}} r)}
u_i^p \,dv^{G_{c_i}} \\
&\leq
\int_N u_i^p dv^{G_{c_i}} \\
& \leq  1.
\end{split}
\end{equation*}
Here we used $dv^{\tilde{g}_i} =
m_i^p \,dv^{G_{c_i}}$. The last inequality comes from the fact that
any function in $\Omega^{(2)}(\mH^{k+1}_{c_i}\times \mS^{n-k-1})$
has $L^p$-norm smaller than $1$. Since
\[\Th_i:
(B^n(r), \Th_i^* (\tilde{g}_i))
\to
(B^{G_{c_i}} ( x, m_i^{-\frac{2}{n-2}} r), \tilde{g}_i)
\]
is an isometry we redefine $\tilde{u}_i$ as
$\tilde{u}_i \circ \Th_i$ which gives us solutions of
\[
L^{\Th_i^*(\tilde{g}_i)} \tilde{u}_i
=
\mu_i \tilde{u}_i^{p-1}
\]
on $B^n(r)$ with
$\int_{B^n(r)} \tilde{u}_i^p \, dv^{\Th_i^*(\tilde{g}_i)}
\leq 1$.
Since $\| \tilde{u}_i \|_{L^\infty(B^n(r))}
= \tilde{u}_i (0) = 1$ we can
apply Lemma~\ref{lim_sol} with $V = \mR^n$, $\al=1/i$,
$g_\al = \Th_i^*(\tilde{g}_i)$, and $u_\al = \tilde{u}_i$
(we can apply this lemma since each compact set of $\mR^n$ is
contained in some ball $B^n(r)$). This shows that there exists a
non-negative $C^2$ function $u$ on $\mR^n$ that does not
vanish identically (since $u(0)=1$) and that satisfies
\[
L^{\xi^n} u
=
\an  \Delta^{\xi^n} u
=
\bar{\mu} u^{p-1}
\]
where $\bar{\mu} = 0$. By \eqref{normlr_lim} we further have
\[
\int_{ B^n(r)} u^p \, dv^{\xi^n}
=
\lim_{i \to \infty} \int_{ B^{G_{c_i}}
( x, m_i^{-\frac{2}{n-2}} r)} u_i^p
\,dv^{G_{c_i}}
\leq 1
\]
for any $r>0$. In particular,
\[
\int_{\mR^n} u^p  \, dv^{\xi^n}  \leq 1.
\]
Lemma \ref{limit_space=R^n} below then implies the contradiction
$0 = \bar{\mu}\geq \mu(\mS^n)$. This proves that $\La^{(2)}_{n,k}$ is
positive.
\end{proof}

%%%%%%%%%%%%%%%%%%%%%%%%%%%%%%%%%%%%%%%%%%%%%%%%%%%%%%%%%%%%%%%%%%%%%%%%%
\subsection{The constants $\La_{n,0}$} \label{subsec.La.null}
%%%%%%%%%%%%%%%%%%%%%%%%%%%%%%%%%%%%%%%%%%%%%%%%%%%%%%%%%%%%%%%%%%%%%%%%%

Now we show that 
  $$\La_{n,0}=\mu(\mS^n)=n(n-1)\om_n^{2/n}.$$ 
The
corresponding model spaces $\mH_c^1\times \mS^{n-1}$ carry the standard
product metric $dt^2 + \si^{n-1}$ of $\mR\times \mS^{n-1}$, independently
of $c \in [-1,1]$.  Thus $\La_{n,0}^{(i)}=\mu^{(i)}(\mR\times \mS^{n-1})$.
Proposition~\ref{HconfS} yields a conformal diffeomorphism from the
cylinder $\mR\times \mS^{n-1}$ to $\mS^n\setminus\mS^0$, the
$n$-sphere with north and south poles removed.

\begin{lemma}
\[
\La_{n,0}^{(i)}\leq\mu(\mS^n)= n (n-1)\om_n^{2/n}
\]
for $i=1,2$.
\end{lemma}
\begin{proof}
We use the notation of Proposition \ref{HconfS} with $k=0$. Then the
standard metric on $S^n$ is
\[
\si^n=(\sin s)^2(dt^2+\si^{n-1})= (\cosh t)^{-2}(dt^2+\si^{n-1}).
\]
It follows that $(\om_n)^{-2/n} (\cosh t)^{-2}(dt^2+\si^{n-1})$ is a
(non-complete) metric of volume~$1$ and scalar curvature
$n(n-1)\om^{2/n} = \mu(\mS^n)$ on
$\mH_c^1 \times \mS^{n-1} = \mR\times \mS^{n-1}$. This is equivalent
to saying that
\[
u(t)
\definedas \om_n^{-\frac{n-2}{2n}}(\cosh t)^{-\frac{n-2}{2}}
\]
is a solution of \eqref{eq.conf} with $\mu=\mu(\mS^n)$ and
$\|u\|_{L^p}=1$ on $\mH_c^1\times \mS^{n-1}$ equipped with the
product metric. Clearly we have $u\in L^2$, and
$\|u\|_{L^\infty} = \om_n^{-\frac{n-2}{2n}}<\infty$.
Thus $u\in \Om^{(1)}(\mH_c^1\times \mS^{n-1})$. As a consequence, we obtain
$\La_{n,0}^{(1)}\leq n (n-1)\om_n^{2/n}$.

Further, we have
\[
\mu(\mS^n) \|u\|_{L^\infty}^{p-2}
=
n(n-1)
>
\frac{(n-0-2)^2(n-1)}{8(n-2)},
\]
and thus $u\in \Om^{(2)}(\mH_c^1\times \mS^{n-1})$, which implies
$\La_{n,0}^{(2)}\leq n (n-1)\om_n^{2/n}$.
\end{proof}

\begin{lemma}\label{lemma.cyl}
Let $u\in C^2(\mR\times \mS^{n-1})$ be a solution of \eqref{eq.conf}
on $\mR\times \mS^{n-1}$ with $\|u\|_{L^p} \leq 1$,
$u \not\equiv 0$. Then $\mu\geq \mu(\mS^n)$.
\end{lemma}

\begin{proof}
As above $\si^n=(\sin s)^2 (dt^2+\si^{n-1})$. If $u$ solves
\eqref{eq.conf} with $h=dt^2+\si^{n-1}$ then
$\ti u\definedas(\sin s)^{-\frac{n-2}2}u$ solves
\[
L^{\si^n}\ti u = \mu \ti u^{p-1}.
\]
Further $\ti u^p\,dv^{\si^n}=u^p\,dv^h$, hence
$\nu\definedas\|\ti u\|_{L^p(S^n\setminus S^0,\si^n)}\leq 1$.
For $\al>0$ small, we choose a smooth cut-off function
$\chi_\al: S^n \to [0,1]$ which is $1$ on $S^n\setminus U_{\al}(S^0)$,
with support disjoint from $S^0$, and with
$|d\chi_\al|_{\si^n}\leq 2/\al$. Then using \eqref{formula.dchiu} in
Appendix~\ref{app.formula.dchiu} we see that
\[
\int_{\mS^n}
(\chi_\al \ti u) L^{\si^n}(\chi_\al \ti u) \,dv^{\si^n}
=
\mu \int_{\mS^n}
u^p\chi_\al^2\,dv^{\si^n}+ a\int_{\mS^n}|d\chi_\al|_{\si^n}^2 \ti u^2
\,dv^{\si^n}.
\]
The first summand tends to $\mu\nu^p$ as $\al\searrow 0$. By H\"older's
inequality the second summand is bounded by
\[
\frac{4a}{\al^2}
\|\ti u\|_{L^p(U_\al(S^0)\setminus S^0,\si^n)}^2
\Vol(U_\al(S^0) \setminus S^0,\si^n)^{2/n}
\leq
C \|\ti u\|_{L^p(U_\al(S^0) \setminus S^0, \si^n)}^2
\to 0
\]
as $\al\searrow 0$. Together with
$\lim_{\al\searrow 0} \|\chi_\al\ti u\|_{L^p(S^n\setminus S^0,\si^n)} =
\nu$ we obtain
\[
\mu(\mS^n)
\leq
J^{\si^n}(\chi_\al \ti u) \to \mu\nu^{p-2}
\leq
\mu
\]
as $\al\searrow 0$.
\end{proof}

This lemma obviously implies $\La_{n,0}^{(i)}\geq \mu(\mS^n)$ for
$i=1,2$, and thus we have
\[
\La_{n,0} = \La_{n,0}^{(1)} = \La_{n,0}^{(2)} = \mu(\mS^n).
\]

%%%%%%%%%%%%%%%%%%%%%%%%%%%%%%%%%%%%%%%%%%%%%%%%%%%%%%%%%%%%%%%%%%%%%%%%%
\subsection{The constants $\La_{n,k}$ for $1\leq k\leq n-3$}
\label{subsec.La.k-ge-2}
%%%%%%%%%%%%%%%%%%%%%%%%%%%%%%%%%%%%%%%%%%%%%%%%%%%%%%%%%%%%%%%%%%%%%%%%%

For $2\leq k\leq n-4$ we have found an explicit positive lower bound
on $\La^{(0)}_{n,k}$ which will be published in
\cite{ammann.dahl.humbert:13}. Together with Remark~\ref{rem.twoone}
we obtain a lower bound for $\La_{n,k}$,
see also~\cite{ammann.dahl.humbert:p11b}. For
$m \definedas k+1\in \{3,\ldots,n-3\}$ we conclude
\[
\Lambda_{n,m-1} \geq
n\,a_n\;
\left(\frac{Y_m}{m a_m}\right)^{\frac mn}
\left(\frac{Y_{n-m}}{(n-m)a_{n-m}}\right)^{\frac{n-m}n}
\]
A lower bound in the case $k=1$ and in the cases $(n,k)=(5,2)$
was established in~\cite{ammann.dahl.humbert:p12}.
These lower bounds are not optimal, but
% if our conjectures
%about $\La_{n,k}$ presented in the following subsection hold,
they are optimal up to a factor of at most $2$.
% in the sense
%the quotient of the right hand side by the left hand side in
%\eqref{ineq.lower.exp} is larger than $1/2$.

We collected all known and conjectured values for $\La_{n,k}$ for $n\leq 9$
in Figure~\ref{tab.lank}. In the table, $>0$ means that
% it is proven that $\La_{n,k}$ is positive, but
no explicit positive lower estimate has been worked out until now.

\def\abs{\\[.2cm]}

\begin{figure}
\[
\begin{array}{llccl} %\label{lista}
n  & k  & \mbox{known }\La_{n,k} & \mbox{conjectured }\La_{n,k}&  \ \ \mu(\mS^n)\\
\hline
3 & 0 & 43.82323 & 43.82323 & 43.82323\abs

4 & 0 & 61.56239 & 61.56239 & 61.56239\\
4 & 1 & 38.9     & 59.40481 & 61.56239\abs

5 & 0 & 78.99686 & 78.99686 & 78.99686\\
5 & 1 & 56.6     & 78.18644 & 78.99686\\
5 & 2 & 45.1     & 75.39687 & 78.99686\abs

6 & 0 & 96.29728 & 96.29728 & 96.29728\\
6 & 1 & >0       & 95.87367 & 96.29728\\
6 & 2 & 54.77904 & 94.71444 & 96.29728\\
6 & 3 & 49.98764 & 91.68339 & 96.29728\abs

7 & 0 & 113.5272 & 113.5272 & 113.5272\\
7 & 1 & >0       & 113.2670 & 113.5272\\
7 & 2 & 74.50435 & 112.6214 & 113.5272\\
7 & 3 & 74.50435 & 111.2934 & 113.5272\\
7 & 4 & >0       & 108.1625 & 113.5272\abs

8 & 0 & 130.7157 & 130.7157 & 130.7157\\
8 & 1 & >0       & 130.5398 & 130.7157\\
8 & 2 & 92.24278 & 130.1272 & 130.7157\\
8 & 3 & 95.76372 & 129.3551 & 130.7157\\
8 & 4 & 92.24278 & 127.9414 & 130.7157\\
8 & 5 & >0       & 124.7747 & 130.7157\abs

9 & 0 & 147.8778 & 147.8778 & 147.8778\\
9 & 1 & 109.2993 & 147.7507 & 147.8778\\
9 & 2 & 109.4260 & 147.4615 & 147.8778\\
9 & 3 & 114.3250 & 146.9519 & 147.8778\\
9 & 4 & 114.3250 & 146.1089 & 147.8778\\
9 & 5 & 109.4260 & 144.6521 & 147.8778\\
9 & 6 & >0       & 141.4740 & 147.8778

%10 & 0 & 165.0220 &  & 165.0220642\\
%10 & 1 & 102.6925 &  & 165.0220642\\
%10 & 2 & 126.4134 &  & 165.0220642\\
%10 & 3 & 132.0534 &  & 165.0220642\\
%10 & 4 & 133.3072 &  & 165.0220642\\
%10 & 5 & 132.0534 &  & 165.0220642\\
%10 & 6 & 126.4134 &  & 165.0220642\\
%10 & 7 & >0       &  & 165.0220642
\end{array}
\]
\caption{Known and conjectured lower estimates for $\La_{n,k}$.}\label{tab.lank}
\end{figure}

%%%%%%%%%%%%%%%%%%%%%%%%%%%%%%%%%%%%%%%%%%%%%%%%%%%%%%%%%%%%%%%%%%%%%%%%%
\subsection{Speculation about $\La_{n,k}$ for $k\geq 1$}
\label{subsec.La.speculation}
%%%%%%%%%%%%%%%%%%%%%%%%%%%%%%%%%%%%%%%%%%%%%%%%%%%%%%%%%%%%%%%%%%%%%%%%%

We want to speculate about two relations that seem likely to us
although we have no proof. Conformally, the model spaces
$\mH^{k+1}_c \times \mS^{n-k-1}$ can be viewed as an interpolation
between $\mR^{k+1}\times \mS^{n-k-1}$ (for $c=0$) and the sphere
$\mS^{n}$ (for $c=1$). Since the sphere has the largest possible value of
the conformal Yamabe constant we could hope that the function
$c\mapsto \mukim(\mH^{k+1}_c\times \mS^{n-k-1})$ is increasing for
$c\in [0,1]$, or in particular
\[
\mukim(\mR^{k+1}\times \mS^{n-k-1})
\leq
\mukim(\mH^{k+1}_c\times \mS^{n-k-1})
\]
for all $c\in [-1,1]$. This would imply
\[
\La_{n,k}
=
\mukim(\mR^{k+1}\times \mS^{n-k-1}).
\]

To formulate the second potential relation we define the following
variant of $\mukim(\mH^{k+1}_c\times \mS^{n-k-1})$:
\[
\mukim_{\mH^{k+1}_c}(\mH^{k+1}_c\times \mS^{n-k})
\definedas
\inf \{J^{G_c}(u) \,|\,
u\in C_0^\infty(\mH^{k+1}_c)\}.
\]
Here $J^{G_c}$ is the functional of $\mH^{k+1}_c\times \mS^{n-k-1}$,
but we only evaluate it for functions that are constant along the
sphere $ \mS^{n-k-1}$. We ask, similarly to the Question formulated in
the Introduction in \cite{akutagawa.florit.petean:07}, whether
\[
\mukim_{\mH^{k+1}_c}(\mH^{k+1}_c\times \mS^{n-k})
=
\mukim(\mH^{k+1}_c\times \mS^{n-k}).
\]
It seems likely to us that the answer is yes, if and only if
$|c|\leq 1$.

An affirmative answer for $|c|\leq 1$ would imply, using a reflection
argument, that we can restrict not only to functions that are constant
along the sphere, but even to radial functions. Here a radial function
is defined as a function of the form $u(x,y)=u(d^{\mH^{k+1}_c}(x))$
where $d^{\mH^{k+1}_c}(x)$ is the distance from $x$ to a fixed point in
$\mH^{k+1}_c$. The constants $\La_{n,k}$ could then be calculated
numerically. For example we would obtain
\[
\La_{4,1}
= \mukim(\mR^2\times \mS^2)
= 59.40481 \ldots
\]
and thus $\si(S^2\times S^2)\geq  59.40481\ldots \,$, which
should be compared to $\mu(\mS^4)= 61.56239\ldots$ and
$\mu(\mS^2\times \mS^2)=16\pi= 50.26548\ldots$

Using the handle reduction techniques of the proof of the
$h$-cobordism theorem, together with information about the spin
bordism groups in low dimensions, we would be able to conclude
the following lower
bounds on $\si(M)$ for simply connected spin manifolds of dimension~$n$
(and with vanishing index in the case $n=8$).
\[
\begin{array}{lc}
n  & \si(M) >\\
\hline
5 &  75.3968 \\
6 &  91.683 \\
7 &  108.162 \\
8 &  124.774 \\
\end{array}
\]
If $n = 5,6,7$ we use that $M$ is spin bordant to a sphere, for
$n = 8$ we have that $M$ is spin bordant to a number of copies of
$\mH P^2$. For the standard metric we have $\mu(\mH P^2)=144.959\ldots$.
In all four cases we would have $\si(M)/ \si(\mS^n) > 0.95$. Similar
conclusions can be drawn for non-spin manifolds.

These inequalities would imply for example that $\si(\mC P^3)$ is not
attained by the Fubini-Study metric, as $\mu(\mC P^3)=82.9864\ldots$ for
this conformal class.

%
%Other conclusions would be:
%All simply-connected compact spin manifolds of dimension~$5$ satisfy
%$\si(M)> 75.3968$.
%All simply-connected compact manifolds of
%dimension~$6$ (resp.\ dimension~$7$) satisfy $\si(M)> 91.683$
%(resp.\ $\si(M)> 108.162$).
%Using $\mu(\mH P^2)=144.959\ldots$, we would also see that
%any simply connected compact spin manifold of dimension~$8$ with
%$\hat{\cA}(M)=0$ satisfies $\si(M)>124.774$.
%In all four cases we would have
%$\si(M^n)/\si(\mS^n)>0.95$.
%
%These inequalities would imply e.g. that $\si(\mC P^3)$ is not attained
%by the Fubini-Study metric, as $\mu(CP^3)=82.9864\ldots$ for this
%conformal class.
%

%%%%%%%%%%%%%%%%%%%%%%%%%%%%%%%%%%%%%%%%%%%%%%%%%%%%%%%%%%%%%%%%%%%%%%%%%
\section{Limit spaces and limit solutions} \label{sec.limit}
%%%%%%%%%%%%%%%%%%%%%%%%%%%%%%%%%%%%%%%%%%%%%%%%%%%%%%%%%%%%%%%%%%%%%%%%%

In the proofs of the main theorems we will construct limit solutions
of the Yamabe equation on certain limit spaces. For this we need the
following two lemmas.

\begin{lemma} \label{diffeom}
Let $V$ be an $n$-dimensional manifold. Let $(q_{\al})$
be a sequence of points in $V$ that converges to a point $q$ as
$\al \searrow 0$. Let~$(\ga_\al)$ be a sequence of metrics defined on a
neighborhood $O$ of $q$ that converges to a metric $\ga_0$ in
the $C^2(O)$-topology. Finally, let $(b_\al)$ be a sequence of positive
real numbers such that $\lim_{\al \searrow 0} b_\al = \infty$. Then for
$r>0$ there exists for $\al$ small enough a diffeomorphism
\[
\Theta_\al:
B^n(r)
\to
B^{\ga_\al} (q_\al, b_\al^{-1} r)
\]
with $\Theta_\al(0)= q_\al$ such that the metric
$\Theta_{\al}^*(b_\al^2 \ga_\al)$ tends to the
flat metric $\xi^n$ in $C^2(B^n(r))$.
\end{lemma}

\begin{proof}
Denote by $\exp_{q_\al}^{\ga_\al}: U_\al \to O_\al$ the exponential
map at the point~$q_\al$ defined with respect to the metric $\ga_\al$.
Here $O_\al$ is a neighborhood of $q_\al$ in $V$ and $U_{\al}$ is a
neighborhood of the origin in $\mR^n$. We set
\[
\Theta_\al:
B^n(r) \ni x
\mapsto
\exp_{q_\al}^{\ga_\al}( b_\al^{-1} x) \in
B^{\ga_\al} (q_\al, b_\al^{-1} r)  .
\]
It is easily checked that $\Theta_\al$ is the desired diffeomorphism.
\end{proof}

\begin{lemma} \label{lim_sol}
Let $V$ be an $n$-dimensional manifold. Let $(g_\al)$ be a sequence
of metrics that converges to a metric $g$ in $C^2$ on all compact
sets $K \subset V$ as $\al \searrow 0$. Assume that $(U_\al)$ is an
increasing sequence of subdomains of $V$ such that
$\bigcup_{\al} U_\al = V$. Let $u_\al \in C^2(U_\al)$ be a
sequence of positive functions such that
$\| u_\al \|_{L^\infty(U_\al)}$ is bounded independently of~$\al$. We assume
\begin{equation} \label{eqal}
L^{g_{\al}} u_\al = \mu_\al u_\al^{p-1}
\end{equation}
where the $\mu_\al$ are numbers tending to $\bar{\mu}$. Then there
exists a non-negative function $u \in C^2(V)$,  satisfying
\begin{equation} \label{eq_lim}
L^g u
= \bar{\mu} u^{p-1}
\end{equation}
on $V$ and a subsequence of $u_\al$ that tends to $u$ in $C^1$ on
each open set $\Om \subset V$ with compact closure. In particular
\begin{equation} \label{norminf_lim}
\| u \|_{L^\infty(K)}
=
\lim_{\al \searrow 0} \| u_\al \|_{L^\infty(K)},
\end{equation}
and
\begin{equation} \label{normlr_lim}
\int_K u^r \,dv^g
=
\lim_{\al \searrow 0} \int_K u_\al^r \,dv^{g_\al}
\end{equation}
for any compact set $K$ and any $r \geq 1$.
\end{lemma}

\begin{proof}
Let $K$ be a compact subset of $V$ and let $\Om$ be an open set with
smooth boundary and compact closure in $V$ such that $K \subset \Om$.
{}From equation \eqref{eqal} and the boundedness of
$\| u_\al \|_{\infty}$ we see with standard results on elliptic
regularity (see for example \cite{gilbarg.trudinger:77}) that $(u_\al)$
is bounded in the Sobolev space $H^{2,2n}(\Om,g)$, that is all
derivatives of $u_\al|_\Om$ up to second order are bounded in
$L^{2n}(\Om)$). As this Sobolev space embeds compactly into
$C^1(\Om)$, a subsequence of $(u_\al)$ converges in $C^1(\Om)$ to a
function $u^\Om \in C^1(\Om)$, $u^\Om \geq 0$, depending on $\Om$.
Let $\phi \in C^{\infty}(\Om)$ be compactly supported in $\Om$.
Multiplying Equation \eqref{eqal} by $\phi$ and integrating over
$\Om$, we obtain that $u^\Om$ satisfies Equation \eqref{eq_lim}
weakly on $\Om$. By standard regularity results $u^\Om \in C^2(\Om)$
and satisfies Equation \eqref{eq_lim}.

As a next step we choose an increasing sequence of compact sets~$K_m$ satisfying
$\bigcup_m K_m = V$. Using the above arguments and taking successive
subsequences it follows that $(u_\al)$ converges to functions $u_m \in C^{2}(K_m)$
that solve Equation \eqref{eq_lim} and satisfy
$u_m \geq 0$ and $u_m |_{K_{m-1}} = u_{m-1}$. We define $u$ on $V$ by
$u = u_m$ on $K_m$. By taking a diagonal subsequence of $(u_\al)$ we
get that $(u_\al)$ tends to $u$ in $C^1$ on any compact set $K \subset
V$. This ends the proof of Lemma~\ref{lim_sol}.
\end{proof}

The next Lemma is useful when the sequence of metrics in Lemma~\ref{lim_sol}
converges to the flat metric $\xi^n$ on $\mR^n$.

%Lemma \ref{lim_sol} will be applied several times in the article, in
%most applications the limit space $\mR^n$ will be obtained. In this
%situation, the following lemma will be helpful.

\begin{lemma} \label{limit_space=R^n}
Let $\xi^n$ be the standard flat metric on $\mR^n$, and assume that
$u \in C^2(\mR^n)$, $u \geq 0$, $u \not\equiv 0$ satisfies
\begin{equation} \label{EqRn}
L^{\xi^n} u = \mu u^{p-1}
\end{equation}
for some $\mu \in \mR$. Assume in addition that
$u \in L^p(\mR^n)$ and that
\begin{equation*}
\| u \|_{L^p(\mR^n)} \leq  1.
\end{equation*}
Then $\mu \geq \mu(\mS^n)$.
\end{lemma}

\begin{proof}
The map $\phi:\mR\times \mS^{n-1}\to \mR^n\setminus\{0\}$,
$\phi(t,x) = e^t x$, is a conformal diffeomorphism with
\[
dt^2+\si^{n-1} = e^{-2t} \phi^* \xi^n.
\]
Thus if $u$ is a solution of \eqref{EqRn}, then
$\hat u \definedas e^{(n-2)t/2}u\circ \phi$ is a solution of
$L^{dt^2+\si^{n-1}} \hat u = \mu \hat u^{p-1}$ and
$\|\hat u\|_{L^p(\mR\times\mS^{n-1})} = \|u\|_{L^p(\mR^n)}\leq 1$.
The result now follows from Lemma \ref{lemma.cyl}.
\end{proof}

%%%%%%%%%%%%%%%%%%%%%%%%%%%%%%%%%%%%%%%%%%%%%%%%%%%%%%%%%%%%%%%%%%%%%%%%%
\section{$L^2$-estimates on \WS-bundles} \label{wsbundles}
%%%%%%%%%%%%%%%%%%%%%%%%%%%%%%%%%%%%%%%%%%%%%%%%%%%%%%%%%%%%%%%%%%%%%%%%%

Manifolds with a certain structure of a double bundle will appear in
the proofs of our main results. In this section we derive
$L^2$-estimates for solutions to a perturbed Yamabe equation on a \WS-bundle.

%%%%%%%%%%%%%%%%%%%%%%%%%%%%%%%%%%%%%%%%%%%%%%%%%%%%%%%%%%%%%%%%%%%%%%%%%
\subsection{Definition and statement of the result}
%%%%%%%%%%%%%%%%%%%%%%%%%%%%%%%%%%%%%%%%%%%%%%%%%%%%%%%%%%%%%%%%%%%%%%%%%

Let $n \geq 1$ and $0 \leq k \leq n-3$ be integers. Let $W$ be a
closed manifold of dimension~$k$ and let $I$ be an interval. By a
\emph{\WS-bundle} we will mean the product
$P \definedas I \times W \times S^{n-k-1}$ equipped with a metric of
the form
\begin{equation}\label{gthform}
\gWS
=
dt^2 + e^{2\phi(t)}h_t + \sigma^{n-k-1}
\end{equation}
where $h_t$ is a smooth family of metrics on $W$ depending on
$t \in I$ and $\phi$ is a function on $I$. The condition $k\leq n-3$
implies that the sphere $S^{n-k-1}$ carries positive scalar curvature,
which is an essential ingredient in the proof of Theorem~\ref{theo.fibest}.
Let $\pi : P \to I$ be the projection onto the first factor and let
$F_t \definedas \pi^{-1}(t) = \{ t \} \times W \times S^{n-k-1}$.
The metric induced on $F_t$ is
$g_t \definedas e^{2\phi(t)} h_t + \sigma^{n-k-1}$.
Let $H_t$ be the mean curvature of $F_t$ in~$P$, that is $H_t \pa_t$
is the mean curvature vector of $F_t$. We always use the sign
convention for the mean curvature vector for which it points in the
direction of decreasing volume of $F_t$. The mean curvature is given
by the formula
\begin{equation} \label{meancurvP}
H_t
=
- \frac{k}{n-1}\phi'(t) - e(h_t)
\end{equation}
with
$e(h_t) \definedas \frac{1}{2(n-1)}{\trace}_{h_t}(\partial_t h_t)$.
Clearly, $e(h_t) = 0$ if $t \mapsto h_t$ is constant.
The derivative of the volume element $dv^{g_t}$ of $F_t$ is
\[
\pa_t dv^{g_t} = -(n-1) H_t dv^{g_t}.
\]
It is straightforward to check that the scalar curvatures of $\gWS$
and $h_t$ are related by (see Appendix~\ref{app.bgm} for details)
\begin{equation} \label{scalP}
\begin{split}
\scal^{\gWS}
&=
e^{-2\phi(t)} \scal^{h_t} + (n-k-1)(n-k-2) \\
&\quad
- k(k+1) \phi'(t)^2
- 2k \phi''(t)
- (k+1) \phi'(t) \tr (h_t^{-1} \pa_t h_t ) \\
&\quad
+ \frac{3}{4} \tr( ( h_t^{-1} \pa_t h_t)^2 )
- \frac{1}{4} ( \tr(h_t^{-1} \pa_t h_t) )^2
- \tr({h_t}^{-1} \pa_t^2 h_t).
\end{split}
\end{equation}

\begin{definition}
We say that condition $(A_t)$ holds if the following assumptions
are true:
\Atbox{\ \
\begin{matrix}
1.)& t \mapsto h_t\mbox{ is constant},\hfill\\
2.)& e^{-2 \phi(t)} \inf_{x \in W} \Scal^{h_t}(x)
\geq -\frac{n-k-2}{32} \an ,\hfill\\
3.)& |\phi'(t)| \leq 1,\hfill\\
4.)& 0 \leq -2k \phi''(t) \leq \frac1{2} (n-1)(n-k-2)^2.
\end{matrix}}{(A_t)}
Similarly, we say that condition $(B_t)$ holds if the following
assumptions are true:
\Atbox{\ \
\begin{matrix}
1.) & t \mapsto \phi(t)\mbox{ is constant,}\hfill\\
2.) & \inf_{x\in F_t} \Scal^{\gWS}(x)
\geq \frac{1}{2} \Scal^{\si^{n-k-1}}
= \frac{1}{2} (n-k-1)(n-k-2),\hfill\\
3.) & \frac{(n-1)^2}{2} e(h_t)^2
+ \frac{n-1}{2} \partial_t e(h_t)
\geq
- \frac{3}{64} (n-k-2).\hfill
\end{matrix}}{(B_t)}
\end{definition}

Let $P$ be \WS-bundle equipped with a metric $G$ that is close to
$\gWS$ in a sense to be made precise later. Let $\al, \beta \in \mR$
be such that $[\al,\beta] \subset I$. Our goal is to derive an
estimate for the distribution of $L^2$-norm of a positive solution to
the Yamabe equation
\[
L^G u = \mu u^{p-1}.
\]
If we write this equation in terms of the metric $\gWS$ we get a
perturbed version of the Yamabe equation for $\gWS$. We assume that we
have a smooth positive solution $u$ of the equation
\begin{equation} \label{eqyamodif}
L^{\gWS} u
=
\an\Delta^{\gWS} u + \Scal^{\gWS} u
=
\mu u^{p-1} + d^* A(du) + Xu + \ep \pa_t u - su
\end{equation}
where $s, \ep \in C^\infty(P)$, $A\in \End(T^*P)$, and $X\in \Gamma(TP)$
are perturbation terms coming from the difference between $G$ and
$\gWS$. We assume that the endomorphism $A$ is symmetric and that $X$
and $A$ are vertical, that is $dt(X) = 0$ and $A(dt) = 0$.

\begin{theorem}\label{theo.fibest}
Assume that $P$ carries a metric $\gWS$ of the form
\eqref{gthform}. Let $\al,\beta \in \mR$ be such that
$ [\al,\beta] \subset I$. Assume further that for each $t \in I$
either condition $(A_t)$ or condition $(B_t)$ is true. We also assume
that $u$ is a positive solution of \eqref{eqyamodif} satisfying
\begin{equation} \label{assumpmaxu}
\mu \| u \|_{L^\infty(P)}^{p-2}
\leq
\frac{(n-k-2)^2(n-1)}{8(n-2)}.
\end{equation}
Then there exists $c_0>0$ independent of $\al$, $\beta$, and $\phi$,
such that if
\[
\| A \|_{L^\infty(P)},
\| X \|_{L^\infty(P)},
\| s \|_{L^\infty(P)},
\| \ep \|_{L^\infty(P)},
\| e(h_t) \|_{L^\infty(P)}
\leq
c_0
\]
then
\[
\int_{\pi^{-1} \left((\al + \ga ,\be - \ga)\right) }
u^2 \, dv^{\gWS}
\leq
\frac{4 \| u \|_{L^\infty}^2}{n-k-2}
\left(
\Vol^{g_\al} ( F_{\al }) + \Vol^{g_\be} ( F_{\be })
\right),
\]
where $\ga \definedas \frac{\sqrt{32}}{n-k-2}$.
\end{theorem}
Note that this theorem only gives information when
$\be - \al > 2\ga$.

%%%%%%%%%%%%%%%%%%%%%%%%%%%%%%%%%%%%%%%%%%%%%%%%%%%%%%%%%%%%%%%%%%%%%%%%
\subsection{Proof of Theorem \ref{theo.fibest}}
%%%%%%%%%%%%%%%%%%%%%%%%%%%%%%%%%%%%%%%%%%%%%%%%%%%%%%%%%%%%%%%%%%%%%%%%

For the proof of Theorem \ref{theo.fibest} we need the following lemma.
\begin{lemma}\label{lem_intw}
Let $T$ and $\gamma$ be positive numbers, and assume that  
$w:~[-T-\ga,~T+\ga]~\to~\mR$
is a smooth positive function satisfying
\begin{equation} \label{w''geq}
w''(t)
\geq
\frac{w(t)}{\ga^2}.
\end{equation}
Then
\begin{equation}\label{intwlemma}
\int_{-T}^{T} w(t)^m \,dt
\leq
\frac{\ga}{m} \Bigl( (w(T+\ga))^m+(w(-T-\ga))^m \Bigr)
\end{equation}
for all $m \geq 1$.
\end{lemma}

\begin{proof}
Assume that $w|_{[-T-\ga,T+\ga]}$ attains its minimum at $t_0$.
Since $w'' \geq w/\ga^2 >0$  we have $w'(t) > 0$ for $t \in (t_0,T+\ga)$,
and $w'(t) < 0$ for
$t \in (-T-\ga,t_0)$. We first study the case when $t_0 \in (-T,T)$.
We define $W(t) \definedas w(t) + \ga w'(t)$.
As $w$ and $w'$ are increasing on $(t_0,T+\gamma)$, we get
\begin{equation} \label{W+}
\begin{split}
W(T)
&=
w(T) + \int_T^{T+\ga} w'(T) \,dt \\
&\leq
w(T)+\int_T^{T+\ga}w'(t)\,dt \\
&=
w(T+\ga).
\end{split}
\end{equation}
{}From \eqref{w''geq} we see that $W'(t) \geq W(t)/\ga$, or
$\pa_t \ln W(t) \geq 1/\ga$. Integrating this relation between
$t \in (t_0,T)$ and $T$ we get
\[
W(t)
\leq
e^{-\frac{T-t}{\ga}} W(T).
\]
Using that $w \leq W$ on $(t_0,T)$ together with \eqref{W+} we obtain
\[
w(t)
\leq
W(t)
\leq
e^{-\frac{T-t}{\ga}} w(T+\ga),
\]
and hence
\[
w(t)^m
\leq
e^{- m\frac{T-t}{\ga} } (w(T+\ga))^m
\]
for all $t\in [t_0,T]$ and $m \geq 1$.
Integrating this relation over $t \in [t_0,T]$ we get
\begin{equation} \label{intw+}
\int_{t_0}^T w(t)^m\,dt
\leq
\frac{\ga (1 - e^{- m \frac{T-t_0}{\ga}})}{m} (w(T+\ga))^m
\leq
\frac{\ga}{m} (w(T+\ga))^m.
\end{equation}
Similarly we conclude that
\begin{equation} \label{intw-}
\int_{-T}^{t_0} w(t)^m \, dt
\leq
\frac{\ga}{m} (w(-T-\ga))^m.
\end{equation}
This proves relation \eqref{intwlemma} in this case. In the case that
$t_0 \leq -T$ relation~\eqref{intw+} remains valid. Using
\[
\int_{-T}^T w(t)^m \, dt
\leq
\int_{t_0}^T w(t)^m \, dt
\]
and
\[
(w(T+\ga))^m \leq (w(T+\ga))^m + (w(-T-\ga))^m,
\]
we obtain relation \eqref{intwlemma}. We proceed in a similar way
using \eqref{intw-} in case $t_0 \geq T$. This ends the proof of Lemma
\ref{lem_intw}.
\end{proof}

\begin{proof}[Proof of Theorem \ref{theo.fibest}]
The Laplacian $\Delta^{\gWS}$ on $P$ is related to the Laplacian
$\Delta^{g_t}$ on $F_t$ through the formula
\[
\Delta^{\gWS}= \Delta^{g_t} - \pa_t^2 + (n-1) H_t \pa_t,
\]
so
\begin{equation*}
\begin{split}
\int_{F_t} u \Delta^{\gWS} u \, dv^{g_t}
&=
\int_{F_t}
\left(
u \Delta^{g_t} u - u (\pa_t^2 u) + (n-1) H_t u(\pa_t u)
\right)
\, dv^{g_t} \\
&=
\int_{F_t}
\left(
|\dvert u|^2 - u (\pa_t^2 u) + (n-1) H_t u(\pa_t u)
\right)
\, dv^{g_t}.
\end{split}
\end{equation*}
Together with \eqref{eqyamodif} we get
\begin{equation*}
\begin{split}
\an\int_{F_t} u \pa_t^2 u \, dv^{g_t}
&=
\int_{F_t}
\Big(
\an |\dvert u|^2
+ \an (n-1) H_t u \pa_t u \\
&\qquad
- \< \dvert u, A(\dvert u) \>
- u Xu
- \ep u \pa_t u \\
&\qquad
+ (\Scal^{\gWS} + s) u^2
- \mu u^p
\Big)
\, dv^{g_t}.
\end{split}
\end{equation*}
In the following we denote by $\de(c_0)$ a positive constant that
goes to~$0$ if $c_0$ tends to $0$ and whose convergence depends only
on $n$, $\mu$, and $h$. We set
$S_t \definedas \inf_{F_t} \scal^{\gWS}$.
If we use the inequality $2 \int |ab| \leq \int (a^2 + b^2)$ to
simplify the terms involving $X$ and $\ep$ we obtain
\begin{equation*}
\begin{split}
\an\int_{F_t} u \pa_t^2 u \, dv^{g_t}
&\geq
\int_{F_t}
\Big(
(\an - \de(c_0)) |\dvert u|^2
+ \an (n-1) H_t u \pa_t u \\
&\qquad
- \de(c_0) (\pa_t u)^2
+ (S_t - \de(c_0)) u^2
- \mu u^p
\Big)
\, dv^{g_t}.
\end{split}
\end{equation*}
If $c_0$ is small enough so that $\an - \de(c_0) > 0$ we conclude that
\begin{equation} \label{eqconf1}
\begin{split}
\an \int_{F_t}
\Big(
u \pa_t^2 u - (n-1) H_t u(\pa_t u)
\Big)
\, dv^{g_t}
&\geq
(S_t - \de(c_0)) w(t)^2 \\
&\quad
- \int_{F_t}
\Big(\de(c_0) (\pa_t u)^2 + \mu u^p \Big) \, dv^{g_t},
\end{split}
\end{equation}
where we define
\[
w(t)
\definedas
\| u \|_{L^2(F_t)}
=
\left( \int_{F_t}u^2\,dv^{g_t} \right)^{1/2}.
\]
Differentiating this we get
\begin{equation} \label{eq.udtu}
\begin{split}
2 w'(t)w(t)
&=
\pa_t \int_{F_t} u^2\, dv^{g_t} \\
&=
\int_{F_t}
\Big(
2 u (\partial_t u)
- (n-1)H_t u^2
\Big)
\, dv^{g_t}.
\end{split}
\end{equation}

We now assume that $(A_t)$ holds. Then \eqref{meancurvP} tells us that
\[
H_t = - \frac{k}{n-1} \phi'(t),
\]
so \eqref{eq.udtu} becomes
\begin{equation} \label{eq.udtu.A}
w'(t)w(t)
=
\int_{F_t} u (\partial_t u) \, dv^{g_t}
+
\frac{k}{2} \phi'(t) w(t)^2.
\end{equation}
We differentiate this and obtain
\begin{equation*}
\begin{split}
w'(t)^2 + w''(t) w(t)
&=
\int_{F_t} (\partial_t u)^2 \, dv^{g_t} \\
&\quad
+ \int_{F_t}
\Big(u \partial_t^2 u - (n-1) H_t u \partial_t u \Big)
\, dv^{g_t} \\
&\quad
+ \frac{k}{2} \phi''(t) w(t)^2 + k \phi'(t) w'(t) w(t).
\end{split}
\end{equation*}
{}From \eqref{eqconf1} we get
\begin{equation} \label{main2.A}
\begin{split}
w'(t)^2 + w''(t) w(t)
&\geq
\left(1 - \frac{\de(c_0)}{\an}\right)
\int_{F_t} (\partial_t u)^2 \, dv^{g_t} \\
&\quad
+ \left(
\frac{1}{\an} \left( S_t - \de(c_0) \right) + \frac{k}{2} \phi''(t)
\right) w(t)^2 \\
&\quad
- \frac{1}{\an}  \int_{F_t} \mu u^p \, dv^{g_t}
+ k \phi'(t) w'(t) w(t).
\end{split}
\end{equation}
We now use Cauchy-Schwarz and \eqref{eq.udtu.A} to get
\begin{equation*}
\begin{split}
w(t)^2 \int_{F_t} (\partial_t u)^2 \, dv^{g_t}
&\geq
\left( \int_{F_t} u (\partial_t u) \, dv^{g_t} \right)^2 \\
&=
\left( w'(t)w(t) - \frac{k}{2} \phi'(t) w(t)^2 \right)^2,
\end{split}
\end{equation*}
and thus
\begin{equation} \label{dtu^2.A}
\int_{F_t} (\partial_t u)^2 \, dv^{g_t}
\geq
\left( w'(t) - \frac{k}{2} \phi'(t) w(t) \right)^2 .
\end{equation}
{}From assumption \eqref{assumpmaxu} it follows that
\begin{equation} \label{upterm.A}
\frac{\mu}{\an} \int_{F_t} u^p \,dv^{g_t}
\leq
\frac{(n-k-2)^2}{32} w(t)^2 .
\end{equation}
Inserting \eqref{dtu^2.A} and \eqref{upterm.A} into \eqref{main2.A} we
obtain
\begin{equation*}
\begin{split}
w'(t)^2 + w''(t) w(t)
&\geq
\left(1 - \frac{\de(c_0)}{\an}\right)
\left( w'(t) - \frac{k}{2} \phi'(t) w(t) \right)^2 \\
&\quad
+ \left(
\frac{1}{\an} \left( S_t - \de(c_0) \right) + \frac{k}{2} \phi''(t)
\right) w(t)^2 \\
&\quad
- \frac{(n-k-2)^2}{32} w(t)^2
+ k \phi'(t) w'(t) w(t),
\end{split}
\end{equation*}
or after some rearranging,
\begin{equation} \label{main3.A}
\begin{split}
w''(t) w(t)
&\geq
- \frac{\de(c_0)}{\an}
\left( w'(t) - \frac{k}{2} \phi'(t) w(t) \right)^2 \\
&\quad
+ \left(
\frac{1}{\an} \left( S_t - \de(c_0) \right) + \frac{k}{2} \phi''(t)
+ \frac{k^2}{4} \phi'(t)^2 - \frac{(n-k-2)^2}{32}
\right) w(t)^2.
\end{split}
\end{equation}
Next we estimate the coefficient of $w(t)^2$ in the last line of
\eqref{main3.A}. We denote this coefficient by $D$. Using \eqref{scalP}
and assumption 1.) of $(A_t)$, which tells us that $e(h_t) = 0$, we get
\begin{equation*}
\begin{split}
D
&=
\frac{1}{\an}
\left(
e^{-2\phi(t)} \inf_{x\in W}\scal^{h_t}(x) - k (k+1) \phi'(t)^2
- 2k \phi''(t) + (n-k-1)(n-k-2)
\right) \\
&\quad
- \frac{\de(c_0)}{\an}
+ \frac{k}{2} \phi''(t)
+ \frac{k^2}{4} \phi'(t)^2 - \frac{(n-k-2)^2}{32} \\
&=
\frac{1}{\an} e^{-2\phi(t)}  \inf_{x\in W}\scal^{h_t}(x)
+ \frac{1}{\an} \left( (n-k-1)(n-k-2) - \de(c_0) \right)
+ \frac{k}{2(n-1)} \phi''(t) \\
&\quad
- \frac{k}{4(n-1)} (n-k-2) \phi'(t)^2
- \frac{(n-k-2)^2}{32}.
\end{split}
\end{equation*}
{}From assumptions 2.) and 3.) of $(A_t)$ we obtain
\begin{equation*}
\begin{split}
D
&\geq
-\frac{n-k-2}{32}
+ \frac{1}{\an} \left( (n-k-1)(n-k-2) - \de(c_0) \right)
+ \frac{k}{2(n-1)} \phi''(t) \\
&\quad
- \frac{k}{4(n-1)} (n-k-2)
- \frac{(n-k-2)^2}{32} \\
&=
\frac{1}{4(n-1)} \left( (n-1)(n-k-2)^2 + 2k \phi''(t) \right) \\
&\quad
-\frac{n-k-2}{32} - \frac{(n-k-2)^2}{32} - \frac{\de(c_0)}{\an}.
\end{split}
\end{equation*}
Using assumption 4.) of $(A_t)$ and $n-k-2 \geq 1$ we further obtain
\begin{equation*}
\begin{split}
D
&\geq
\frac{1}{4(n-1)} \left( \frac{1}{2} (n-1)(n-k-2)^2 \right) \\
&\quad
-\frac{(n-k-2)^2}{32} - \frac{(n-k-2)^2}{32} - \frac{\de(c_0)}{\an} \\
&=
\frac{(n-k-2)^2}{16} - \frac{\de(c_0)}{\an}.
\end{split}
\end{equation*}
Inserting this in \eqref{main3.A} we get
\begin{equation*}
\begin{split}
w''(t) w(t)
&\geq
- \frac{\de(c_0)}{\an}
\left( w'(t) - \frac{k}{2} \phi'(t) w(t) \right)^2 \\
&\quad
+ \left( \frac{(n-k-2)^2}{16} - \frac{\de(c_0)}{\an} \right) w(t)^2 \\
&\geq
- \frac{2\de(c_0)}{\an} w'(t)^2 \\
&\quad
+ \left(
- \frac{2\de(c_0)}{\an} \frac{k^2}{4} \phi'(t)^2
+ \frac{(n-k-2)^2}{16} - \frac{\de(c_0)}{\an}
\right) w(t)^2,
\end{split}
\end{equation*}
where we also used the elementary inequality
$(a-b)^2 \leq 2a^2 + 2b^2$. Again using assumption 3.) of $(A_t)$ we
conclude
\begin{equation} \label{main4.A}
\begin{split}
w''(t) w(t)
&\geq
- \frac{2\de(c_0)}{\an} w'(t)^2 \\
&\quad
+ \left(
\frac{(n-k-2)^2}{16}
- \frac{\de(c_0)}{\an} \left( 1 + \frac{k^2}{2} \right)
\right) w(t)^2.
\end{split}
\end{equation}
Fix a small positive number $\hat{\de}$.
Choose $c_0$ small so that
$\de(c_0)$ is also small. Then \eqref{main4.A} tells us that
\begin{equation} \label{main5.A}
w''(t) w(t)
\geq
\frac{(n-k-2)^2}{32} w(t)^2
- \hat{\de} w'(t)^2.
\end{equation}
Define $v(t) \definedas w(t)^{1+\hat{\de}}$. This function satisfies
\begin{equation*}
\begin{split}
v''(t)
&=
(1+\hat{\de}) w''(t) w(t)^{\hat{\de}}
+
\hat{\de} (1+\hat{\de}) w'(t)^2 w(t)^{\hat{\de}-1}\\
&\geq
(1+\hat{\de})\frac{(n-k-2)^2}{32} w(t)^{1+\hat{\de}} \\
&\geq
\frac{(n-k-2)^2}{32} v(t).
\end{split}
\end{equation*}

Next we assume that $(B_t)$ holds. Then \eqref{meancurvP} becomes
\[
H_t = - e(h_t),
\]
and from \eqref{eq.udtu} we get
\begin{equation} \label{eq.udtu.B}
w'(t)w(t)
=
\int_{F_t}
\left(
u (\partial_t u) + \frac{n-1}{2} e(h_t) u^2
\right)
\, dv^{g_t}.
\end{equation}
Differentiating this we get
\begin{equation*}
\begin{split}
w'(t)^2 + w''(t) w(t)
&=
\int_{F_t} \Bigg(
(\partial_t u)^2 + (n-1) e(h_t) u \partial_t u \\
&\qquad
+ \left(
\frac{(n-1)^2}{2} e(h_t)^2 + \frac{n-1}{2} \partial_t e(h_t)
\right) u^2
\Bigg) \, dv^{g_t} \\
&\quad
+ \int_{F_t}
\left( u \partial_t^2 u - (n-1) H_t u \partial_t u \right)
\, dv^{g_t}.
\end{split}
\end{equation*}
Next we use \eqref{eqconf1} followed by assumptions 2.) and 3.) of~$(B_t)$
to obtain
\begin{equation*}
\begin{split}
w'(t)^2 + w''(t) w(t)
&\geq
\int_{F_t} \Bigg(
(\partial_t u)^2 + (n-1) e(h_t) u \partial_t u \\
&\qquad
+ \left(
\frac{(n-1)^2}{2} e(h_t)^2 + \frac{n-1}{2} \partial_t e(h_t)
\right) u^2 \\
&\qquad
-\frac{\de(c_0)}{\an} (\pa_t u)^2
-\frac{\mu}{\an}  u^p
\Bigg) \, dv^{g_t} \\
&\quad
+
\frac{1}{\an} (S_t - \de(c_0)) w(t)^2 \\
&\geq
\int_{F_t} \left(
\left( 1 - \frac{\de(c_0)}{\an} \right) (\pa_t u)^2
+ (n-1) e(h_t) u \partial_t u
-\frac{\mu}{\an}  u^p
\right) \, dv^{g_t} \\
&\quad
+
\left( \frac{1}{2\an}(n-k-1)(n-k-2) - \frac{3}{64}(n-k-2)
- \frac{\de(c_0)}{\an} \right) w(t)^2.
\end{split}
\end{equation*}
{}From \eqref{upterm.A} we further get, using $k\leq n-3$ in the
last step,
\begin{equation} \label{main3.B}
\begin{split}
w'(t)^2 + w''(t) w(t)
&\geq
\int_{F_t} \left(
\left( 1 - \frac{\de(c_0)}{\an} \right) (\pa_t u)^2
+ (n-1) e(h_t) u \partial_t u
\right) \, dv^{g_t} \\
&\quad
+ \Bigg(
\frac{1}{2\an}(n-k-1)(n-k-2) - \frac{3}{64}(n-k-2) \\
&\qquad
- \frac{1}{32} (n-k-2)^2
- \frac{\de(c_0)}{\an} \Bigg) w(t)^2 \\
&\geq
\int_{F_t} \left(
\left( 1 - \frac{\de(c_0)}{\an} \right) (\pa_t u)^2
+ (n-1) e(h_t) u \partial_t u
\right) \, dv^{g_t} \\
&\quad
+ \left( \frac{1}{32} (n-k-2)(n-k-3/2)
- \frac{\de(c_0)}{\an} \right) w(t)^2 \\
&\geq
\int_{F_t} \left(
\left( 1 - \frac{\de(c_0)}{\an} \right) (\pa_t u)^2
+ (n-1) e(h_t) u \partial_t u
\right) \, dv^{g_t} \\
&\quad
+ \left( \frac{1}{32} (n-k-2)^2 + \frac{1}{64}
- \frac{\de(c_0)}{\an} \right) w(t)^2.
\end{split}
\end{equation}
We set $E_t \definedas \sup_{F_t} |e(h_t)|$ and use \eqref{eq.udtu.B}
to compute
\begin{equation*}
\begin{split}
w(t)^2 \int_{F_t} (\partial_t u)^2 \, dv^{g_t}
&\geq
\left( \int_{F_t} u (\partial_t u) \, dv^{g_t} \right)^2 \\
&=
\left(
w'(t)w(t) - \frac{n-1}{2} \int_{F_t} e(h_t) u^2 \, dv^{g_t}
\right)^2 \\
&=
\left( w'(t) w(t) \right)^2
+
\left(\frac{n-1}{2} \int_{F_t} e(h_t) u^2 \, dv^{g_t} \right)^2 \\
&\quad
- (n-1) w'(t) w(t) \int_{F_t} e(h_t) u^2 \, dv^{g_t} \\
&\geq
w'(t)^2 w(t)^2
- \left( \frac{n-1}{2} \right)^2 E_t^2 w(t)^4 \\
&\quad
- (n-1) |w'(t)| w(t) \int_{F_t} |e(h_t)| u^2 \, dv^{g_t} \\
&\geq
w'(t)^2 w(t)^2
- \left( \frac{n-1}{2} \right)^2 E_t^2 w(t)^4 \\
&\quad
- (n-1) E_t |w'(t)| w(t)^3.
\end{split}
\end{equation*}
Next we divide by $w(t)^2$ and obtain
\begin{equation} \label{dtu^2.B}
\begin{split}
\int_{F_t} (\partial_t u)^2 \, dv^{g_t}
&\geq
w'(t)^2
- \left( \frac{n-1}{2} \right)^2 E_t^2 w(t)^2
- (n-1) E_t |w'(t)| w(t) \\
&\geq
w'(t)^2
- \left( \frac{n-1}{2} \right)^2 E_t^2 w(t)^2
- \frac{n-1}{2} E_t
\left( w'(t)^2 + w(t)^2 \right) \\
&=
\left( 1 - \frac{n-1}{2} E_t \right) w'(t)^2
- \left(
\frac{n-1}{2} E_t + \left( \frac{n-1}{2} \right)^2 E_t^2
\right) w(t)^2.
\end{split}
\end{equation}
Also
\begin{equation*}
\begin{split}
\left|\int_{F_t} e(h_t) u \partial_t u \, dv^{g_t}\right|
&\leq
\int_{F_t} | e(h_t) u \partial_t u | \, dv^{g_t} \\
&\leq
E_t \int_{F_t} | u \partial_t u | \, dv^{g_t} \\
&\leq
\frac{1}{2} E_t
\int_{F_t} \left( u^2 + (\partial_t u)^2 \right) \, dv^{g_t},
\end{split}
\end{equation*}
so
\begin{equation} \label{eudtu.B}
\int_{F_t}
(n-1)e(h_t) u \partial_t u
\, dv^{g_t}
\geq
- \frac{n-1}{2} E_t
\int_{F_t} \left( u^2 + (\partial_t u)^2 \right) \, dv^{g_t}.
\end{equation}
Fix a small number $\hat{\de} > 0$. We insert \eqref{dtu^2.B} and
\eqref{eudtu.B} in \eqref{main3.B} and choose~$c_0$ small enough
so that $\de(c_0)$ and $E_t$ are small. Then we get that $w(t)$
satisfies the same inequality \eqref{main5.A} as we obtained under the
assumption~$(A_t)$. We have showed that in both cases $(A_t)$ and
$(B_t)$ the function $v(t) = w(t)^{1+\hat{\de}}$ satisfies
\[
v''(t)
\geq
v(t) / \ga^2
\]
since $\frac{32}{(n-k-2)^2} = \ga^2$.

Now we apply Lemma \ref{lem_intw} to the function
$\tilde{v}(t) \definedas v(t+ \frac{\beta + \al}{2})$ with
$T= \frac{\beta-\al}{2} - \ga$ and $m= \frac{2}{1+\hat{\de}}$.
{}From this we obtain
\begin{equation} \label{intvtilde}
\frac{\ga}{m}
\left( (\tilde{v}(T+\ga))^m + (\tilde{v}(-T-\ga))^m \right)
\geq
\int_{-T}^T \tilde{v}^m \, dt.
\end{equation}
With $s=t +\frac{\beta + \al}{2}$ we further have
\[
\int_{-T}^T \tilde{v}^m \, dt
=
\int_{\al + \ga}^{\be - \ga}
\left( w (s) \right)^{(1+\hat{\de})m} \, dt
=
\int_{\al + \ga}^{\be - \ga} w^2 \, ds.
\]
{}From the definition of $w$ we obtain
\[
\int_{-T}^T \tilde{v}^m \, dt
=
\int_{\pi^{-1} \left( (\al + \ga ,\be - \ga) \right) }
u^2 \, dv^{\gWS}.
\]
In addition, we have
\begin{equation*}
\begin{split}
\left( (\tilde{v}(T+\gamma))^m+(\tilde{v}(-T-\gamma))^m \right)
&=
\int_{F_{\al}} u^2 \, dv^{g_{\al}}
+  \int_{F_{\be} }u^2 \, dv^{g_{\be}} \\
&\leq
\| u \|_{L^\infty(P)}^2 \left( \Vol^{g_\al}(F_{\al })
+
\Vol^{g_\be}(F_{\be })  \right).
\end{split}
\end{equation*}
Choosing $\hat{\de}$ small we may assume $m \geq \sqrt{2}$. This
together with \eqref{intvtilde} and $\ga = \frac{\sqrt{32}}{n-k-2}$
gives us
\[
\int_{\pi^{-1} \left((\al + \ga ,\be - \ga)\right)}
u^2 \, dv^{\gWS}
\leq
\frac{4 \| u \|_{L^\infty}^2}{n-k-2}
\left( \Vol^{g_\al} ( F_{\al}) + \Vol^{g_\be} ( F_{\be }) \right).
\]
This proves Theorem \ref{theo.fibest}.
\end{proof}

%%%%%%%%%%%%%%%%%%%%%%%%%%%%%%%%%%%%%%%%%%%%%%%%%%%%%%%%%%%%%%%%%%%%%%%%
\section{Proof of Theorem~\ref{main.weak}}
%%%%%%%%%%%%%%%%%%%%%%%%%%%%%%%%%%%%%%%%%%%%%%%%%%%%%%%%%%%%%%%%%%%%%%%%

%%%%%%%%%%%%%%%%%%%%%%%%%%%%%%%%%%%%%%%%%%%%%%%%%%%%%%%%%%%%%%%%%%%%%%%%
\subsection{Stronger version of Theorem~\ref{main.weak}}
%%%%%%%%%%%%%%%%%%%%%%%%%%%%%%%%%%%%%%%%%%%%%%%%%%%%%%%%%%%%%%%%%%%%%%%%

In this section we prove the following Theorem \ref{main.strong}. By
taking the supremum over all conformal classes Theorem
\ref{main.strong} implies Theorem \ref{main.weak}.

\begin{theorem} \label{main.strong}
Suppose that $(M_1,g_1)$ and $(M_2,g_2)$ are compact Riemannian
manifolds of dimension~$n$.  Let $N$ be obtained from $M_1$, $M_2$, by
a connected sum along~$W$ as described in Section
\ref{joining_man}. Then there is a family of metrics $g_\th$,
$\th\in(0,\th_0)$ on $N$
satisfying
\begin{eqnarray*}
\min \left\{ \mu(M_1 \amalg M_2, g_1\amalg g_2 ), \La_{n,k} \right\}
&\leq &
\liminf_{\th\searrow 0} \mu(N,g_\th)\\
 &\leq &\limsup_{\th\searrow 0} \mu(N,g_\th)\\
 &\leq &\mu(M_1 \amalg M_2,g_1 \amalg g_2).
\end{eqnarray*}
\end{theorem}

In the following we define suitable metrics $g_\th$, and then we show that
they satisfy these inequalities.

%%%%%%%%%%%%%%%%%%%%%%%%%%%%%%%%%%%%%%%%%%%%%%%%%%%%%%%%%%%%%%%%%%%%%%%%
\subsection{Definition of the metrics $g_\th$}
%%%%%%%%%%%%%%%%%%%%%%%%%%%%%%%%%%%%%%%%%%%%%%%%%%%%%%%%%%%%%%%%%%%%%%%%

We continue to use the notation of Section~\ref{joining_man}. In the
following, $C$ denotes a constant that might change its value between
lines. Recall that $(M,g) = (M_1 \amalg M_2, g_1 \amalg g_2)$. For
$i=1,2$ we define the metric $h_i$ as the restriction of $g_i$ to
$W_i' = w_i(W \times \{ 0 \})$, and we set
$h \definedas h_1 \amalg h_2$ on $W' =  W_1' \amalg W_2'$.
As already explained, the normal exponential map of $W'\subset M$
defines a diffeomorphism
\[
w_i: W \times B^{n-k}(\Rmax)
\to
U_i(\Rmax),\quad i=1,2,
\]
which decomposes $U(\Rmax) = U_1(\Rmax) \amalg U_2(\Rmax)$ as a
product $W' \times B^{n-k}(\Rmax)$.

%By possible shrinking $\Rmax$ we
%can and will assume without loss
%of generality, that the diffeomorphism $w_i$ extends to a diffeomorphism
%defined on  $W \times B^{n-k}(\Rmax+\alpha)$ for some $\alpha>0$, which will
%imply that several estimates below hold uniformly
%on $U(\Rmax)$, e.g.\ the estimate
%\eqref{normC0T}.

In general, the Riemannian metric~$g$ does not have a
corresponding product structure, and we introduce
an error term $T$ measuring the difference from the product metric. If
$r$ denotes the distance function to $W'$, then the metric $g$ can be
written as
\begin{equation} \label{metric=product}
g =  h + \xi^{n-k} + T = h + dr^2 + r^2 \sigma^{n-k-1} + T
\end{equation}
on $U(\Rmax)\setminus W'\cong W'\times (0,\Rmax)\times
S^{n-k-1}$. Here $T$ is a symmetric $(2,0)$-tensor vanishing on $W'$
(in the sense of sections of $(T^*M \otimes T^*M)|_{W'}$).
We also define the product metric
\begin{equation} \label{def.g'}
g' \definedas h + \xi^{n-k} = h + dr^2 + r^2 \si^{n-k-1},
\end{equation}
on $U(\Rmax) \setminus W'$. Thus $g = g' + T$. Since $T$ vanishes on
$W'$ we have for sufficiently small $r$
\begin{equation} \label{normC0T}
| T(X,Y) | \leq Cr | X |_{g'} | Y |_{g'}
\end{equation}
for any $X,Y \in T_x M$ where $x \in U(\Rmax)$. Since $T$ is smooth we
have for sufficiently small $r$
\[
|(\nabla_U T)(X,Y) |
\leq
C | X |_{g'} | Y |_{g'} | U|_{g'},
\]
and
\[
|(\nabla^2_{U,V}) T(X,Y) |
\leq
C | X |_{g'} | Y |_{g'} |U|_{g'}|V|_{g'},
\]
for $X,Y,U,V \in T_x M$. We define $T_i \definedas T|_{M_i}$ for
$i=1,2$.

For a fixed $R_0\in (0,\Rmax)$, $R_0<1$, and sufficiently small in the sense
of equation~\eqref{normC0T} and the following equations,
we choose a smooth positive function $F: M \setminus W' \to \mR$ such that
\[
F(x) =
\begin{cases}
1,
&\text{if $x \in M_i \setminus U_i(\Rmax)$;} \\
r_i(x)^{-1},
&\text{if $x \in U_i(R_0)\setminus W'$.}
\end{cases}
\]
Next we choose small numbers $\theta, \detwo \in (0,R_0)$ with
$\th> \detwo > 0$. Here ``small'' means that
%
%we first choose a
%sequence $\th=\th_j$ of positive numbers tending to zero, such
%that all following arguments hold for all $\th$.  Then
%
for a given small number $\theta$
we choose a number $\detwo = \detwo(\th) \in (0,\th)$ such that
all arguments that need $\detwo$ to be small will hold,
%Then, in a third step we choose  $\de_1 = \de_1(\th)\in (0,\detwo)$
% in a similar way,
see Figure~\ref{hier_var}.
%%%%%%%%%%%%%%%%%%%%%%%%%%%%%%%%%%%%%%%%%%%%%%%%%%
\begin{figure}
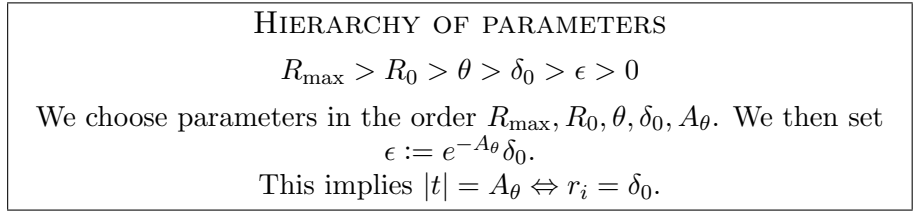

\begin{center}
$\framebox{\vbox{
{\sc\large Hierarchy of parameters}
\[
\Rmax> R_0 > \th> \detwo > \ep>0
\]
We choose parameters in the order
$\Rmax,R_0,\th,\detwo, A_{\th}$.
We then set
$\ep \definedas e^{-A_{\th}}\detwo$.\\
This implies $|t|=A_{\th}\Leftrightarrow r_i=\detwo$.
}}$

\caption{Hierarchy of parameters} \label{hier_var}
\end{center}
\end{figure}
%%%%%%%%%%%%%%%%%%%%%%%%%%%%%%%%%%%%%%%%%%%%%%%%%%
%
For any $\th>0$ and sufficiently small $\detwo$ there is
$A_\th \in (\th^{-1}, (\detwo)^{-1})$ and a family of smooth functions
$f= f_{\th, \detwo}: U(\Rmax) \to \mR$ depending only on the coordinate
$r$ such that

%
%%%%%%%%%%%%%%%%%%%%%%%%%%%%%%%%%%%%%%%%
\begin{figure}
\includegraphics[width=\textwidth]{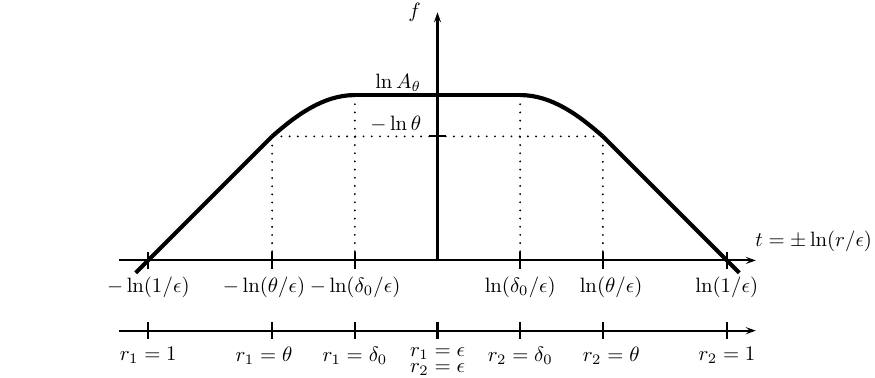}
\smallskip
\caption{The function $f$} \label{figure.f}
\end{figure}
%%%%%%%%%%%%%%%%%%%%%%%%%%%%%%%%%%%%%%%%
%
\[
f(x)  =
\begin{cases}
         -  \ln r(x),  &\text{if $x \in U(\Rmax) \setminus U(\th)$;} \\
\phantom{-} \ln A_\th, &\text{if $x \in U(\detwo)$,}
\end{cases}
\]
and such that
\begin{equation} \label{asumpf}
\left| r\frac{df}{dr} \right|
=
\left| \frac{df }{d(\ln r)} \right|
\leq 1,
\quad
\text{and}
\quad
\left\|r\frac{d}{dr}\left(r\frac{df}{dr}\right)\right\|_{L^\infty}
=
\left\|\frac{d^2f}{d^2(\ln r)}\right\|_{L^\infty}
\to 0
\end{equation}
as $\th\searrow 0$. See Figure~\ref{figure.f}.

We set $\ep = e^{-A_\th} \detwo$. We can and will assume that $\ep<1$.

Let $N$ be obtained from~$M$ by a connected sum along $W$
with parameter $\ep$, as described in Section~\ref{joining_man}.
In particular,
$U^N_\ep(s) = \left( U(s)\setminus U(\ep) \right) /{\sim}$ for all
$s \in [\ep,  \Rmax]$. On the set $U^N_\ep(\Rmax) =
\left( U(\Rmax) \setminus U(\ep) \right)/{\sim}$ we define
the variable $t$ by
\[
t \definedas
\begin{cases}
         -  \ln r_1 + \ln \ep, &
\text{on $U_1(\Rmax) \setminus U_1(\ep)$;} \\
\phantom{-} \ln r_2 - \ln \ep, &
\text{on $U_2(\Rmax) \setminus U_2(\ep)$.}
\end{cases}
\]
Note that $t \leq 0$ on $U_1(\Rmax) \setminus U_1(\ep)$ and $t \geq 0$
on $U_2(\Rmax) \setminus U_2(\ep)$, with $t=0$ precisely on the common
boundary $\pa U_1(\ep)$ identified with $\pa U_2(\ep)$ in $N$. It
follows that
\[
r_i = e^{|t|+ \ln \ep} = \ep e^{|t|}.
\]
We can arrange that
$t:U^N_\ep(\Rmax) \to \mR$  is smooth. Expressed in the variable $t$
we have
\[
F(x) = \ep^{-1}e^{-|t|}
\]
for $x \in U^N_\ep(R_0)$, or in other words if
$|t|+\ln \ep \leq \ln R_0$. Then Equation~\eqref{metric=product} tells
us that
\[
F^2 g =\ep^{-2} e^{-2|t|}(h+T) + dt^2 + \sigma^{n-k-1}
\]
on $U^N_\ep(R_0)$.
%%%%%%%%%%%%%%%%%%%%%%%%%%%%%%%%%%%%%%%%%%%%%%%%%
% The metrics g_\th Double picture
%%%%%%%%%%%%%%%%%%%%%%%%%%%%%%%%%%%%%%%%%%%%%%%%%
\begin{figure}
\begin{center}
\includegraphics[width=\textwidth]{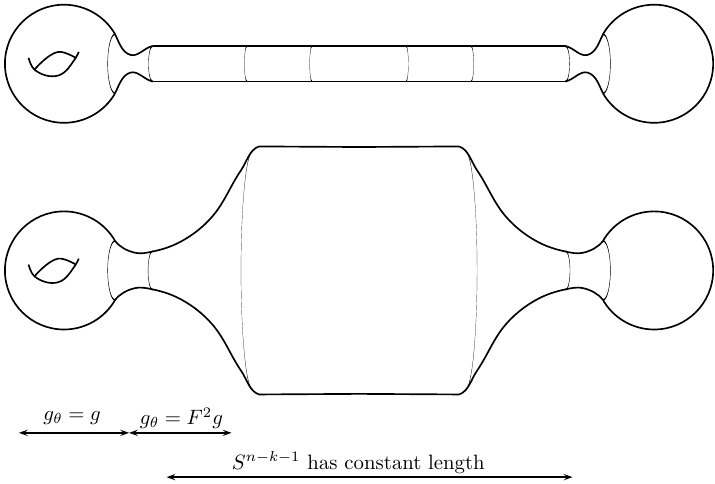}
\end{center}
\caption{The metrics $g_\th$. The horizontal direction in both drawings
corresponds to the $t$-variable. The vertical direction in the upper
drawing corresponds
to the projection to $S^{n-k-1}$, and in the lower drawing it corresponds
to the projection to $S^k$. In the lower drawing, the
curved parts close to the middle part are not drawn realistically.
Their curvature
tends to $0$ for $\th\to 0$, and the middle becomes huge in this limit, and
thus it would be too large for our picture.}\label{fig.gth}
\end{figure}
If we view $f$ as a function of $t$,
then
\[
f(t) =
\begin{cases}
-|t|-\ln\ep,
&\text{if $\ln\th- \ln \ep \leq |t| \leq \ln\Rmax - \ln\ep$;} \\
\ln A_\th,
&\text{if  $|t|\leq \ln\detwo-\ln\ep$;}
\end{cases}
\]
and $|df/dt|\leq 1$, $\|d^2f/dt^2\|_{L^\infty}\to 0$ as $\th$ tends to $0$.
We choose a cut-off function $\chi:\mR\to [0,1]$ such that $\chi=0$ on
$(-\infty,-1]$, $|d\chi| \leq 1$, and $\chi=1$ on $[1,\infty)$. With
these choices we define
\[
g_{\th}
\definedas
\begin{cases}
F^2 g_i,
&\text{on $M_i \setminus U_i(\th)$;} \\
e^{2f(t)}(h_i+T_i) + dt^2 + \sigma^{n-k-1},
&\text{on $U_i(\th)\setminus U_i(\detwo)$;} \\
\left.\begin{aligned}
&A_{\th}^2 \chi( t / A_{\th} )(h_2+T_2)\\
&+ A_{\th}^2 (1-\chi( t / A_{\th} ))(h_1+T_1)\\
&+ dt^2 + \sigma^{n-k-1},
\end{aligned}
\right\}
&\text{on $U^N_\ep(\detwo)$.}
\end{cases}
\]
On $U^N_\ep(R_0)$ we write $g_{\th}$ as
\[
g_\th= e^{2f(t)} \tilde{h}_t + dt^2 + \sigma^{n-k-1} + \widetilde{T}_t,
\]
where the metric $\tilde{h}_t$ is defined by
\[
\tilde{h}_t
\definedas
\chi(t / A_{\th}) h_2 + (1 - \chi(t / A_{\th})) h_1,
\]
for $t \in \mR$, and where the error term $\widetilde{T}_t$ is equal to
\[
\widetilde{T}_t
\definedas
e^{2f(t)}
\left(
\chi(t/A_{\th})T_2 + \left( 1 - \chi(t/A_{\th}) \right) T_1
\right).
\]
See also Figure~\ref{fig.gth}.
On $U^N_\ep(R_0)$ we also define the metric without error term
\begin{equation}\label{def.gth'}
g_\th'
\definedas
g_\th- \widetilde{T}_t
=
e^{2f(t)}\tilde{h}_t + dt^2 + \sigma^{n-k-1}.
\end{equation}
An upper bound for the error term $\ti T_t$ will be needed in the
following. We claim that
\begin{equation} \label{g'gth'}
|X|_{g'}
\leq
C e^{-f(t)} | X |_{g_\th'}
\end{equation}
for $X \in T_x N$, where $g'$ is the metric defined by \eqref{def.g'}.
To prove the claim, we decompose $X$ in a radial part, a part parallel
to $W'$, and a part parallel to $S^{n-k-1}$. This decomposition is
orthogonal with respect to both $g'$ and $g_\th'$. For
$X = \frac{\partial}{\partial t} =
\pm \ep e^{|t|} \frac{\partial}{\partial r} $ we have that
$1 = | X |_{g_\th'}$ and $| X |_{g'}= \ep e^{|t|} \leq e^{-f(t)}$
since $f(t) \leq -|t| - \ln \ep$. The argument is similar if $X$ is
parallel to $S^{n-k-1}$. If $X$ is tangent to $W'$, then
$|X|_g = |X|_h \leq C |X|_{\ti h_t} \leq C e^{-f(t)}|X|_{g_\th'}$,
and the claim follows.

The Relations \eqref{normC0T} and \eqref{g'gth'} imply
\begin{equation*}
\begin{split}
| \widetilde{T}_t(X,Y) |
&\leq
Ce^{2f(t)} |T(X,Y)| \\
&\leq
Ce^{2f(t)} r |X|_{g'} |Y|_{g'} \\
&\leq
C r | X |_{g_\th'}  | Y |_{g_\th'}
\end{split}
\end{equation*}
for all $X,Y$.
In other words this means
\begin{equation} \label{normC0Tt}
|\widetilde{T}_t |_{g_\th'} \leq Cr=C\ep e^{|t|}\leq Ce^{-f(t)}.
\end{equation}
Further, one can calculate that
\begin{equation} \label{normC1Tt}
|\nabla  \widetilde{T}_t |_{g_\th'} \leq Ce^{-f(t)},
\end{equation}
and
\begin{equation} \label{normC2Tt}
|\nabla^2  \widetilde{T}_t |_{g_\th'} \leq Ce^{-f(t)}.
\end{equation}
Here $\nabla$ denotes the Levi-Civita-connection with respect to
$g_\th'$.
In particular we see with Corollary \ref{cor.scal.diff}
\begin{equation} \label{scalgth}
|\Scal^{g_\th} - \Scal^{g_\th'}| \leq C e^{-f(t)}.
\end{equation}

%%%%%%%%%%%%%%%%%%%%%%%%%%%%%%%%%%%%%%%%%%%%%%%%%%%%%%%%%%%%%%%%%%%%%%%%
\subsection{Geometric description of the new metrics}
%%%%%%%%%%%%%%%%%%%%%%%%%%%%%%%%%%%%%%%%%%%%%%%%%%%%%%%%%%%%%%%%%%%%%%%%

In this subsection we collect some facts about the geometry of $F^2g$
and $g_\th'$ introduced in the previous subsection. Most of the
results are not needed for the proof of our result, but are useful to
understand the underlying geometric concept of the argument.
We will thus skip most of the proofs in this subsection.

The first proposition explains the special role of
$\mH^{k+1}\times \mS^{n-k-1}$.

\begin{proposition}
Let $x_i$ be a sequence of points in $M\setminus W$, converging to
$W$. Then the Riemann tensor of $F^2 g$ at $x_i$ converges to the
Riemann tensor of\/ $\mH^{k+1}\times \mS^{n-k-1}$. The covariant
derivative of the Riemann tensor of $F^2 g$ converges to zero. For any
fixed $R>0$ these convergences are uniform on balls (with respect to
the metric $F^2 g$) of radius $R$.  In particular, for any fixed $R>0$
the balls $(B^{F^2g}(x_i,R),x_i,F^2 g)$ converge to a ball of radius $R$
in $\mH^{k+1}\times \mS^{n-k-1}$ in the $C^{2,\al}$-topology of Riemannian
manifolds with base point.
\end{proposition}

The $C^{2,\al}$-topology of
Riemannian manifolds with base point  has its origins
in Cheeger's finiteness theorem \cite{cheeger:70} and in the work of
Gromov \cite{gromov:81}, \cite{gromov:99}. The article by Petersen
\cite[Pages 167--202]{petersen:97} is a good introduction to the
subject.

In the limit $r\searrow 0$ (or equivalently $t\to \infty$) the
$W$-component of the metric $F^2 g$ grows exponentially. The
motivation for introducing the function $f$ into the definition of
$g_\th$ is to slow down this exponential growth: the diameter of the
$W$-component with respect to $g_\th$ is then bounded by
$A_\th\diam(W,g)$, where $\diam(W,g)$ is the diameter of $W$ with
respect to $g$. This slowing down has to be done carefully in order to
get nice limit spaces. The properties claimed for~$f$ imply the
following result.

\begin{proposition}
Let $\th_i$ be a sequence of positive numbers tending to zero, and let
$x_i\in U_\ep^N(\Rmax)$ be a sequence of points such that the limit $c
\definedas \lim (\frac{\pa}{\pa t} f)(t(x_i))$ exists. Then the
Riemann tensor of $g_{\th_i}$ at $x_i$ converges to the Riemann tensor
of $\mH^{k+1}_c\times \mS^{n-k-1}$. The covariant derivative of the
Riemann tensor of $g_{\th_i}$ converges to zero. For any fixed $R>0$ these
convergences are uniform on balls (with respect to the metric $g_{\th_i}$)
of radius~$R$. In particular,  for any fixed $R>0$ the balls
$(B^{g_{\th_i} }(x_i,R),x_i,g_{\th_i})$ converge to a ball of radius~$R$ in
$\mH^{k+1}_c \times \mS^{n-k-1}$ in the $C^{2,\al}$-topology of
Riemannian manifolds with base point.
\end{proposition}

{}From this proposition it follows that the balls
$(B^{F^2g}(x_i,R),x_i,F^2 g)$ converge to a ball of radius $R$ in
$\mH_c^{k+1} \times \mS^{n-k-1}$ in the $C^{2,\al}$-topology of
Riemannian manifolds with base point. Thus, we get an explanation why
the spaces $\mH^{k+1}_c\times \mS^{n-k-1}$ appear as limit spaces.

The sectional curvature of $\mH^{k+1}_c$ is $-c^2$.  Hence the
sectional curvatures of the product $\mH^{k+1}_c\times \mS^{n-k-1}$
are in the interval $[-c^2,1]$. Using this fact we can prove the
following Proposition.

\begin{proposition}\label{scal.bdd}
The scalar curvatures of $g_\th$ and $g_\th'$ are bounded by a
constant independent of $\theta$.
\end{proposition}

\begin{proof}
The metric $g_\th'$ is the metric of a \WSk-bundle.  Hence
\eqref{scalP} is valid.  We calculate $\pa_t\ti h_t=(1/A_\th)
\chi'(t/A_\th)(h_2-h_1)$ and $\pa_t^2\ti h_t=(1/A_\th)^2
\chi''(t/A_\th)(h_2-h_1)$.  This implies
$|\tr^{\ti h_t}\pa_t \ti h_t|\leq C/A_\th$,
$|\tr (\ti h_t^{-1}\pa_t \ti h_t)^2 | \leq C/A_\th^2$,
and $|\tr^{\ti h_t}\pa_t^2 \ti h_t |\leq C/A_\th^2$.
{}From \eqref{scalP} it follows that $\scal^{g_\th'}$ is bounded.
Equation \eqref{scalgth} then implies that $\Scal^{g_\th}$ is bounded.
\end{proof}

The geometry close to the gluing of $M_1\setminus U_1(\ep)$ with
$M_2\setminus U_2(\ep)$ is described by the following simple
proposition.

\begin{proposition}
Let $H$ be the metric on $W\times (-1,1)$ given by $(\chi(t)h_2 +
(1-\chi(t))h_1) + dt^2$.  Then $(U^N_\ep(\detwo), g_\th')$ is isometric
to $(W\times (-1,1)\times S^{n-k-1}, A_\th^2 H +\si^{n-k-1})$.
\end{proposition}

%%%%%%%%%%%%%%%%%%%%%%%%%%%%%%%%%%%%%%%%%%%%%%%%%%%%%%%%%%%%%%%%%%%%%%%%
\subsection{Proof of Theorem~\ref{main.strong}}
%%%%%%%%%%%%%%%%%%%%%%%%%%%%%%%%%%%%%%%%%%%%%%%%%%%%%%%%%%%%%%%%%%%%%%%%

The metrics $g_\th$ are defined for small $\th>0$ as described
above. In order to prove Theorem \ref{main.strong} it is sufficient to
prove
\[
\min\left\{\mu(M,g),\La_{n,k} \right\}
\leq
\lim_{i\to \infty}\mu(N,g_{\th_i})
\leq
\mu(M,g)
\]
for any sequence $\th_i \searrow 0$ as $i \to \infty$ for which
$\lim_{i \to \infty} \mu(N,g_{\th_i})$ exists. Recall that
$(M,g) = (M_1\amalg M_2, g_1\amalg g_2)$.

The upper bound on $\lim_{i\to \infty}\mu(N,g_{\th_i})$ is
easy to prove. The proof of the lower bound is more complicated,
our arguments for this part are inspired by the
compact\-ness-concentration principle in analysis, see for example
\cite{evans:90}. In the case of a concentration, we will use blow-up analysis
in order to construct a non-trivial
solution to the Yamabe equation on some limit space. Here we follow
and generalize a similar construction of a blow-up limit
in lecture notes by Schoen, see~\cite[Chapter~V.2]{schoen.yau:94}.\nobreak

For each metric $g_\th$ we have a solution of the Yamabe equation
\eqref{Yamabe_equation}. We take a sequence of $\th$ tending to
$0$. Following the compactness-concentration principle, this sequence
of solutions can concentrate in points or converge to a non-trivial
solution or do both at the same time. The concentration in points can
be used to construct a non-trivial solution on a sphere by blowing up
the metrics.

In our situation we may have concentration in a fixed point (subcase~I.1)
or in a wandering point (subcase I.2), and we may have
convergence to a non-trivial solution on the original manifold
(subcase II.1.2) or in the attached part (subcases II.1.1 and
II.2). In each of these cases we obtain a different lower bound
for $\lim_{i\to \infty}\mu(N,g_{\th_i})$: In the subcases~I.1 and I.2
the lower bound is $\mu(\mS^n)$, in subcase II.1.2 it is $\mu(M,g)$,
and in the subcases II.1.1 and II.2 we obtain $\La_{n,k}^{(1)}$ and
$\La_{n,k}^{(2)}$ as lower bounds. Together these cases give the lower
bound of Theorem~\ref{main.strong}.

The cases here are not exclusive. For example it is possible that the
solutions may both concentrate in a point and converge to a
non-trivial solution on the original manifold.

In our arguments we will often pass to subsequences. To avoid
complicated notation we write $\theta\searrow 0$ for a sequence
$(\th_i)_{i\in \mN}$ converging to zero, and we will pass successively
to subsequences without changing notation. Similarly $\lim_{\th\searrow 0}
h(\th)$ should be read as $\lim_{i\to \infty} h(\th_i)$.

We set $\mu \definedas \mu(M,g)$ and
$\mu_\th\definedas \mu(N,g_\th)$. From Theorem~\ref{aubin} we have
\begin{equation}\label{ineq.mu.aubin}
\mu, \mu_{\th} \leq \mu(\mS^n).
\end{equation}
After passing to a subsequence, the limit
\[
\bar{\mu}
\definedas
\lim_{\th\searrow 0} \mu_{\th}\in [-\infty,\mu(\mS^n)]
\]
exists. Let $J \definedas J^g$ and $J_\th\definedas J^{g_\th}$ be
defined as in \eqref{def.J^g}.

We start with the easier part of the argument, namely with
\begin{equation} \label{easypart}
\bar{\mu} \leq \mu.
\end{equation}
For this let $\al>0$ be a small number. We choose a smooth cut-off
function~$\chi_\al$ on~$M$ such that $\chi_\al = 1$ on
$M \setminus U(2 \al)$, $|d\chi_\al| \leq 2/\al$, and $\chi_\al = 0$
on $U(\al)$. Let~$u$ be a smooth non-zero function such that
$J(u) \leq \mu + \de$ where $\de$ is a small positive number. On the
support of $\chi_\al$ the metrics $g$ and $g_\th$ are conformal since
$g_\th= F^2 g$ and hence by \eqref{confJ} we have
\[
\mu_\th
\leq J_\th\left(\chi_\al F^{-\frac{n-2}{2}} u
\right)
= J(\chi_\al u)
\]
for $\th< \al$. It is straightforward to compute that
$\lim_{\al \searrow 0} J(\chi_\al u) = J(u) \leq \mu + \de$.
{}From this Relation \eqref{easypart} follows.

Now we turn to the more difficult part of the proof, namely the inequality
\begin{equation} \label{diffpart}
\bar{\mu} \geq \min \left\{ \mu, \La_{n,k} \right\}.
\end{equation}
In the case $\bar{\mu} = \mu(\mS^n)$ this inequality follows trivially from
\eqref{ineq.mu.aubin}. Hence we assume $\bar{\mu} < \mu(\mS^n)$
in the following, which implies $\mu_\th< \mu(\mS^n)$ if $\th$
is sufficiently small. From Theorem~\ref{attained} we know that there
exist positive functions $u_\th\in C^2(M)$ such that
\begin{equation} \label{eqa}
L^{g_{\th}} u_\th= \mu_\th u_\th^{p-1},
\end{equation}
and
\[
\int_N u_\th^p \, dv^{g_\th} = 1.
\]

We begin by proving a lemma that yields a bound of the $L^2$-norm
of $u_\th$ in terms of the $L^\infty$-norm. This result is non-trivial
since $\vol(N,g_\th) \to \infty$ as $\th\searrow 0$.

\begin{lemma} \label{mbelow}
Assume that there exists $b >0$ such that
\[
\mu_\th\sup_{U^N_\ep(b)} u_\th^{p-2}
\leq
\frac{(n-k-2)^2 (n-1)}{8(n-2)}
\]
for $\th$ small enough. Then there exist constants $c_1,c_2>0$
independent of $\th$ such that
\[
\int_N u_\th^2 \, dv^{g_\th}
\leq
c_1 \|u_\th\|_{L^\infty(N)}^2 + c_2
\]
for all sufficiently small $\th$.
In particular, if $\|u_\th\|_{L^\infty(N)}$ is bounded, so is $\|
u_\th\|_{L^2(N)}$.
\end{lemma}

\begin{proof}
Let $\ti r \in (0, b)$  be fixed and set $P = U(\ti r)$. Then $P$ is a
\WS-bundle where, with the notation of Section~\ref{wsbundles},
$I= (\al,\beta)$ with $\al = -\ln \ti r + \ln \ep$ and
$\beta = \ln \ti r - \ln \ep$. On $P$ we have two natural metrics:
$g_\th$ and $\gWS = g'_\th= g_\th- \widetilde{T}_t$.
The metric $\gWS$ has exactly the form \eqref{gthform} with $\phi = f$
and $h_t = \tilde{h}_t$. Let $\th$ be small enough and let
$t \in (-\ln \ti r + \ln \ep, -\ln \detwo  + \ln \ep )  \cup
(\ln \detwo - \ln \ep, \ln \ti r - \ln \ep )$.
Then assumption $(A_t)$ of Theorem \ref{theo.fibest} is true.
Now, again if $\th$ is small enough, we have for all
$t \in (-\ln \detwo + \ln \ep, \ln \detwo - \ln \ep)$ the relation
$\Scal^{\gWS} = \Scal^{\si^{n-k-1}} + O(1/A_\th)$.
The error term $e(\ti h_t)$ from $(B_t)$ in this case satisfies
\[
2(n-1)|e(\ti h_t)|
\leq
\left| \trace^{\ti h_t}\pa_t \ti h_t\right|
=
\left|\trace^{\ti h_t}
\left(\chi'(t/A_\th)\frac{h_2-h_1}{A_\th}\right) \right|
\leq
\frac{C}{A_\th},
\]
and
\[
2(n-1)|\pa_t e(\ti h_t)|
=
\left|\trace\left(
{\ti h_t}^{-1}(\pa_t \ti h_t){\ti h_t}^{-1}(\pa_t \ti h_t)
\right)\right|+\left|\trace^{\ti h_t} \pa_t^2 \ti h_t\right|
\leq \frac{C}{A_\th^2}.
\]
Because of $1/A_\th\leq \th$ condition $(B_t)$ is true.
Equation~\eqref{eqa} is written in the metric $g_\th$.
Using the expression of the Laplacian in local coordinates,
\[
\Delta^{g_\th} u
=
-\sum_{i,j} (\det g_\th)^{-1/2} \partial_{i}
\left(g_\th^{ij} (\det g_\th)^{1/2} \partial_{j} u \right),
\]
one can check that if we write Equation \eqref{eqa} in the metric
$\gWS$ we obtain an equation of the form \eqref{eqyamodif} with
$\mu = \mu_\th$. Together with \eqref{normC0Tt}, \eqref{normC1Tt} and
\eqref{scalgth}, one verifies that the error terms satisfy
\[
| A(x) |_{\gWS},
| X(x) |_{\gWS},
| s(x) |_{\gWS},
| \ep(x) |_{\gWS} \leq C e^{-f(t)},
\]
where $| \cdot |_{\gWS}$ denotes the pointwise norm at a point in
$U^N_\ep(R_0)$, and where $C$ is a constant independent of $\th$.  In
particular for any $c_0>0$, we obtain
\[
| A(x) |_{\gWS},
| X(x) |_{\gWS},
| s(x) |_{\gWS},
| e(\ti h_t)(x) |_{\gWS},
| \ep(x) |_{\gWS} \leq c_0
\]
on $U^N_\ep(\th)$ for small $\th$. These estimates allow us to apply
Theorem~\ref{theo.fibest}. By the assumptions of Lemma~\ref{mbelow},
if $\ti r \in (0,b)$ is small enough, Assumption~\eqref{assumpmaxu}
of Theorem~\ref{theo.fibest} is true.  Thus, all hypotheses of
Theorem~\ref{theo.fibest} hold for
$\al \definedas -\ln \ti r + \ln \ep$,
$\be \definedas \ln \ti r - \ln \ep$,
and hence
\[
\int_{P'} u_\th^2 \, dv^{\gWS}
\leq
\frac{4 \|u_\th\|_{L^\infty}^2}{n-k-2}
\left( \Vol^{g_\al}(F_{\al}) + \Vol^{g_\be}(F_{\be}) \right).
\]
where $P'\definedas U^N_\ep(\ti r e^{-\ga})$. Now observe that
\[
C
\definedas
\frac{4}{n-k-2}
\left( \Vol^{g_\al}(F_{\al}) + \Vol^{g_\be}(F_{\be}) \right)
\]
does not depend on $\th$ (since $F_{\al}$ and $F_{\be}$ correspond to
the hypersurface $r=\ti r$). This implies that
\[
\int_{P'} u_\th^2 \, dv^{\gWS} \leq C \| u_\th\|_{L^\infty(N)}^2
\]
where $C>0$ is independent of $\th$.
Since if $\ti r$ is small enough, we clearly have
\[
dv^{g_\th} \leq 2 dv^{\gWS},
\]
and we obtain that
\[
\int_{P'} u_\th^2 \, dv^{g_\th}
\leq
c_1 \| u_\th\|_{L^\infty(N)}^2
\]
where $c_1 \definedas 2C > 0$ is independent of $\th$. Now observe
that $\Vol^{g_\th} (N \setminus P')$ is bounded by a constant
independent of $\th$. Using the H\"older inequality we obtain
\begin{equation*}
\begin{split}
\int_{N} u_\th^2 \, dv^{g_\th}
&=
\int_{P'} u_\th^2 \, dv^{g_\th}
+ \int_{N\setminus P'} u_\th^2 \, dv^{g_\th} \\
&\leq
c_1 \| u_\th\|_{L^\infty(N)}^2
+ \Vol^{g_\th} (N  \setminus P')^{\frac{2}{n}}
{\left(\int_{N \setminus P'} u_\th^p \, dv \right)}^{\frac{n-2}{n}}.
\end{split}
\end{equation*}
Since $\|u_\th\|_{L^p(N)}=1$, this proves Lemma~\ref{mbelow} with
$c_1$ as defined above and with
$c_2 \definedas \Vol^{g_\th} (N  \setminus P')^{\frac{2}{n}}$.
For small $\th$, the metric $g_\th|_{N\setminus P'}$ is independent of
$\th$, and thus $c_2$ does not depend on $\th$.
\end{proof}

\begin{corollary}
\[
\liminf_{\th\searrow 0} \| u_\th\|_{L^{\infty}(N)}
>0.
\]
\end{corollary}

\begin{proof}
We set $m_\th\definedas \| u_\th\|_{L^{\infty}(N)}$. In order to prove the
corollary by contradiction we assume $\lim_{\th\searrow 0} m_\th= 0$.
Then since $\mu_\th\leq \mu(\mS^n)$ the assumption of
Lemma~\ref{mbelow} is satisfied for all $\th >0$ sufficiently small,
and for all $b>0$ for which $U^N_\ep(b)$ is defined. We get the
contradiction
\[
1
=
\int_N u_\th^p dv^{g_\th}
\leq
m_\th^{p-2} \int_N u_\th^2 dv^g
\leq
m_\th^{p-2}(c_1 m_\th^2+c_2)
\to 0
\]
as $\th\searrow 0$.
\end{proof}

\begin{corollary}
\[
\bar{\mu}
=
\lim_{\th\searrow 0} \mu_\th
>
-\infty.
\]
\end{corollary}

\begin{proof}
Choose $x_\th$ as above. We then have $\Delta^{g_\th} u_\th(x_\th)
\geq 0$, which together with \eqref{eqa} gives us
\[
\Scal^{g_\th} (x_\th) \| u_\th\|_{L^{\infty}(N)}
\leq
\mu_\th  \| u_\th\|_{L^{\infty}(N)}^{p-1}.
\]
Proposition~\ref{scal.bdd} and the previous corollary then imply
that $\mu_\th$ is bounded from below.
\end{proof}

In addition, by Theorem~\ref{aubin}, $\mu_\th$ is bounded from above
by $\mu(\mS^n)$. It follows that $\bar{\mu} \in \mR$. The rest of the
proof is divided into cases.

\pagebreak[3]
%%%%%%%%%%%%%%%%%%%%%%%%%%%%%%%%%%%%%%%%%%%%%%%%%%%%%%%%%%%%%%
\begin{caseI}
$\limsup_{\th\searrow 0}\|u_\th\|_{L^\infty(N)} = \infty$.
\end{caseI}
%%%%%%%%%%%%%%%%%%%%%%%%%%%%%%%%%%%%%%%%%%%%%%%%%%%%%%%%%%%%%%

As before, we set $m_\th\definedas \|u_\th\|_{L^\infty(N)}$ and
choose $x_\th\in N$ with $u_\th(x_\th) = m_\th$. After again taking 
a subsequence, we can assume that $\lim_{\th\searrow 0} m_\th=\infty$.
We consider two subcases.

%%%%%%%%%%%%%%%%%%%%%%%%%%%%%%%%%%%%%%%%%%%%%%%%%%%%%%%%%%%%%%
\begin{subcaseI.1}
There exists $b>0$ such that $x_\th\in N \setminus U^N_\ep(b)$
for an infinite number of $\th$.
\end{subcaseI.1}
%%%%%%%%%%%%%%%%%%%%%%%%%%%%%%%%%%%%%%%%%%%%%%%%%%%%%%%%%%%%%%

We recall that $N_\ep \setminus U^N_\ep(b) =
M_1 \amalg M_2\setminus U(b)$.
By taking a subsequence we can assume
that there exists $\bar{x} \in M_1 \amalg M_2 \setminus U(b)$ such
that $\lim_{\th\searrow 0} x_\th= \bar{x}$. We let
$\tilde{g}_\th\definedas m_\th^{\frac{4}{n-2}} g_\th$.
In a neighborhood $U$ of $\bar{x}$ the metric $g_{\th}= F^2 g$ does
not depend on $\th$. We apply Lemma~\ref{diffeom} with $O=U$, $\al = \th$,
$q_\al = x_\th$, $q = \bar{x}$, $\ga_\al = g_\th=F^2 g$, and
$b_\al = m_\th^{\frac{2}{n-2}}$. Let $r>0$. For $\th$ small enough
Lemma~\ref{diffeom} gives us a diffeomorphism
\[
\Th_\th:
B^n(r)
\to
B^{g_\th} ( x_\th, m_\th^{-\frac{2}{n-2}} r)
\]
such that the sequence of metrics
$(\Th_\th^* (\tilde{g}_\th))$ tends to the flat
metric $\xi^n$ in $C^2(B^n(r))$. We let
$\tilde{u}_\th\definedas m_\th^{-1} u_\th$. By \eqref{confL} we then
have
\[
L^{\tilde{g}_\th} \tilde{u}_\th= \mu_\th{\tilde{u}_\th}^{p-1}
\]
on $B^{g_\th} ( x_\th, m_\th^{-\frac{2}{n-2}} r)$ and, using the fact that
$dv^{\tilde{g}_\th} = m_\th^p \,dv^{g_\th}$, we have
\begin{equation*}
\begin{split}
\int_ {B^{g_\th} ( x_\th, m_{\th}^{-\frac{2}{n-2}} r)}
{\tilde{u}_\th}^p \,dv^{\tilde{g}_\th}
&=
\int_{B^{g_\th} ( x_\th, m_\th^{-\frac{2}{n-2}} r)}
u_\th^p \,dv^{g_\th} \\
&\leq
\int_N u_\th^p dv^{g_\th} \\
&= 1.
\end{split}
\end{equation*}
Since
\[
\Th_\th:
(B^n(r), \Th_\th^* (\tilde{g}_\th))
\to
(B^{g_\th} ( x_\th, m_\th^{-\frac{2}{n-2}} r), \tilde{g}_\th)
\]
is an isometry we can consider $\tilde{u}_\th$ as a solution of
\[
L^{\Th_\th^*(\tilde{g}_\th)} \tilde{u}_\th
=
\mu_\th\tilde{u}_\th^{p-1}
\]
on $B^n(r)$ with
$\int_{B^n(r)} \tilde{u}_\th^p \, dv^{\Th_\th^*(\tilde{g}_\th)}
\leq 1$.
Since $\| \tilde{u}_\th\|_{L^\infty(B^n(r))}
= |\tilde{u}_\th(0)| = 1$ we can
apply Lemma~\ref{lim_sol} with $V = \mR^n$, $\al=\th$,
$g_\al = \Th_\th^*(\tilde{g}_\th)$, and $u_\al = \tilde{u}_\th$
(we can apply this lemma since each compact set of $\mR^n$
is contained in some ball $B^n(r)$). This shows that there
exists a non-negative function $u \not\equiv 0$ (since
$u(0)=1$) of class $C^2$ on $(\mR^n,\xi^n)$ that satisfies
\[
L^{\xi^n} u
=
\an  \Delta^{\xi^n} u
=
\bar{\mu} u^{p-1}.
\]
By \eqref{normlr_lim} we further have
\[
\int_{ B^n(r)} u^p \, dv^{\xi^n}
=
\lim_{\th\searrow 0} \int_{ B^{g_\th}
( x_\th, m_\th^{-\frac{2}{n-2}} r)} u_\th^p
\,dv^{g_\th}
\leq 1
\]
for any $r>0$. In particular,
\[
\int_{\mR^n} u^p  \, dv^{\xi^n}  \leq 1.
\]
{}From Lemma \ref{limit_space=R^n}, we get that
$\bar{\mu} \geq \mu(\mS^n) \geq \min \{ \mu, \La_{n,k} \}$. We have
proved Inequality~\eqref{diffpart} in this subcase.

%%%%%%%%%%%%%%%%%%%%%%%%%%%%%%%%%%%%%%%%%%%%%%%%%%%%%%%%%%%%%%
\begin{subcaseI.2}
For all $b>0$ it holds that $x_\th\in U^N_\ep(b)$ for $\th$ sufficiently
small.
\end{subcaseI.2}
%%%%%%%%%%%%%%%%%%%%%%%%%%%%%%%%%%%%%%%%%%%%%%%%%%%%%%%%%%%%%%

The subset $U^N_\ep(b)$ is diffeomorphic to $W \times I \times S^{n-k-1}$
where $I$ is an interval. We identify
\[
x_\th= (y_\th, t_\th, z_\th)
\]
where $y_\th\in W$,
$t_\th\in (-\ln R_0  + \ln \ep, -\ln \ep + \ln R_0 )$, and
$z_\th\in S^{n-k-1}$. By taking a subsequence we can assume that
$y_\th$, $\frac{t_\th}{A_{\th}}$, and $z_\th$ converge respectively to
$y \in W$, $T \in [-\infty, +\infty]$, and $z \in S^{n-k-1}$.
First we apply Lemma~\ref{diffeom} with
$V = W$, $\al=\th$, $q_\al= y_\th$, $q=y$, $\ga_\al= \tilde{h}_{t_\th}$,
$\ga_0 = \tilde{h}_T$ (we define $\tilde{h}_{-\infty} = h_1$ and
$\tilde{h}_{+\infty} = h_2$), and
$b_\al = m_\th^{\frac{2}{n-2}} e^{f(t_\th)}$. The lemma provides
diffeomorphisms
\[
\Th_\th^y :
B^k(r)
\to
B^{ \tilde{h}_{t_\th}}
(y_\th,  m_\th^{- {\frac{2}{n-2}}} e^{-f(t_\th)} r)
\]
for $r>0$ such that $(\Th_\th^y)^* (m_\th^{\frac{4}{n-2}}
e^{2f(t_\th)} \tilde{h}_{t_\th})$ tends to the flat metric $\xi^k$
on $B^k(r)$ as $\th\searrow 0$. Second we apply Lemma~\ref{diffeom}
with $V= S^{n-k-1}$, $\al = \th$, $q_\al = z_\th$,
$\ga_\al = \ga_0 = \sigma^{n-k-1}$, and
$b_\al = m_\th^\frac{2}{n-2}$. For $r'>0$ we get diffeomorphisms
\[
\Th_\th^z :
B^{n-k-1}(r')
\to
B^{\sigma^{n-k-1}}(z_\th, m_\th^{-\frac{2}{n-2}} r')
\]
such that
$(\Th_\th^z)^* (m_\th^{\frac{4}{n-2}} \sigma^{n-k-1})$ converges
to $\xi^{n-k-1}$ on $B^{n-k-1}(r')$ as $\th\searrow 0$. For $r,r',r''>0$
we define
\begin{equation*}
\begin{split}
U_\th(r,r',r'')
&\definedas
B^{\tilde{h}_{t_\th}}
(y_\th,  m_\th^{-\frac{2}{n-2}} e^{-f(t_\th)} r)
\times
[t_\th- m_\th^{-\frac{2}{n-2}} r'',
t_\th+ m_\th^{-\frac{2}{n-2}} r''] \\
&\qquad
\times  B^{\sigma^{n-k-1}}(z_\th, m_\th^{-\frac{2}{n-2}} r'),
\end{split}
\end{equation*}
and
\[
\Th_\th:
B^k(r) \times [-r'',r''] \times B^{n-k-1}(r')
\to U_\th(r,r',r'')
\]
by
\[
\Th_\th(y,s,z) \definedas ( \Th_\th^y (y), t(s), \Th_\th^z(z) ),
\]
where $t(s) \definedas t_\th+ m_\th^{\frac{2}{n-2}} s$. By
construction $\Th_\th$ is a diffeomorphism, and we see that
\begin{equation} \label{diffe}
\begin{split}
\Th_\th^* (m_\th^\frac{4}{n-2} g_\th)
&=
(\Th_\th^y)^* (m_\th^\frac{4}{n-2} e^{2f(t)}  \tilde{h}_{t})
+ ds^2 \\
&\quad
+
(\Th_\th^z)^*(m_\th^\frac{4}{n-2} \sigma^{n-k-1})
+ \Th_\th^*(m_\th^\frac{4}{n-2} \widetilde{T}_t).
\end{split}
\end{equation}
Next we study the first term on the right-hand side of
\eqref{diffe}. Note that it is here evaluated at $t$, while we have
information above when evaluated at $t_{\th}$. By construction of
$f(t)$ one can verify that
\[
\lim_{\th\searrow 0}
\left\|
\frac{e^{f(t_\th)}}{e^{f(t)}} - 1
\right\|_{C^2( [t_\th- m_\th^{-\frac{2}{n-2}} r'',
t_\th+ m_\th^{-\frac{2}{n-2}} r''])}
= 0
\]
since $\frac{df}{dt}$ and $\frac{d^2f }{dt^2}$ are uniformly
bounded. Moreover it is clear that
\[
\lim_{\th\searrow 0}
\left\|
\tilde{h}_{t} - \tilde{h}_{t_\th}
\right\|_{C^2( B^{\tilde{h}_{t_\th}}
(y_\th, m_\th^{-\frac{2}{n-2}} e^{-f(t_\th)} r))}
= 0
\]
uniformly in $t \in [t_\th- m_\th^{-\frac{2}{n-2}} r'',
t_\th+ m_\th^{-\frac{2}{n-2}} r'']$. As a consequence
\[
\lim_{\th\searrow  0}
\left\|
(\Th_\th^y)^*
\left( m_\th^{\frac{4}{n-2}} \left(
e^{2f(t)} \tilde{h}_{t} - e^{2f(t_\th)} \tilde{h}_{t_\th}
\right) \right)
\right\|_{C^2(B^k(r))}
=0
\]
uniformly for $t \in [t_\th- m_\th^{-\frac{2}{n-2}} r'',
t_\th+ m_\th^{-\frac{2}{n-2}} r'']$. This implies that the sequence
$(\Th_\th^y)^* (m_\th^{\frac{4}{n-2}} e^{2f(t)} \tilde{h}_{t})$
tends to the flat metric $\xi^k$ in $C^2(B^k(r))$ uniformly in
$t$ as $\th\searrow 0$. Further, we also know that the sequence
$(\Th_\th^z)^* (m_\th^{\frac{4}{n-2}} \sigma^{n-k-1})$ tends to
$\xi^{n-k-1}$ in $C^2(B^{n-k-1}(r'))$ as $\th\searrow 0$.
Recalling from \eqref{def.gth'} that $g_\th' =  g_\th- \widetilde{T}_t$,
we have proved that $\Th_\th^*(m_{\th}^{\frac{4}{n-2}} g_\th')$ tends to
the flat metric in $C^2(B^k(r) \times [-r'',r''] \times B^{n-k-1}(r'))$.
Finally we are going to show that the last term of \eqref{diffe} tends
to zero in $C^2$. It follows from \eqref{normC0Tt} that
\begin{equation} \label{tto0}
\lim_{\th\searrow 0}
\left\|
\Th_\th^*(m_\th^\frac{4}{n-2} \widetilde{T}_t)
\right\|_{C^2(B^k(r) \times [-r'',r''] \times B^{n-k-1}(r'))}
= 0.
\end{equation}
Indeed, \eqref{normC0Tt} tells us that
\begin{equation*}
\begin{split}
\left| \Th_\th^*(m_\th^\frac{4}{n-2} \widetilde{T}_t) (X,Y) \right|
& =
m_\th^\frac{4}{n-2}
\left| \widetilde{T}_t({\Th_\th}_*(X), {\Th_\th}_*(Y)) \right| \\
&\leq
C r m_\th^\frac{4}{n-2}
| {\Th_\th}_*(X) |_{g_\th'}
| {\Th_\th}_*(Y) |_{g_\th'} \\
&\leq
Cr
| X |_{\Th_\th^*(m_{\th}^{\frac{4}{n-2}} g_\th')}
| X |_{\Th_\th^*(m_{\th}^{\frac{4}{n-2}} g_\th')},
\end{split}
\end{equation*}
and since $\Th_\th^*(m_{\th}^{\frac{4}{n-2}}g_\th')$ tends to the
flat metric we get \eqref{tto0}. Doing the same with
$\nabla \widetilde{T}_t$ and $\nabla^2\widetilde{T}_t$ using
\eqref{normC1Tt} and \eqref{normC2Tt}, we obtain that
\begin{equation}\label{tto0c2}
\lim_{\th\searrow 0} \Th_\th^*(m_\th^\frac{4}{n-2} \widetilde{T}_t) = 0
\end{equation}
in $C^2(B^k(r) \times [-r'',r''] \times B^{n-k-1}(r'))$. Returning to
\eqref{diffe} we see that the sequence
$\Th_\th^* (m_\th^{\frac{4}{n-2}} g_\th)$ tends to
$\xi^n = \xi^k + ds^2 + \xi^{n-k-1}$ on
$B^k(r) \times [-r'',r''] \times B^{n-k-1}(r')$.
We proceed as in Subcase~I.1 to show that
$\bar{\mu} \geq \mu(\mS^n) \geq \min \{\mu, \La_{n,k} \}$,
which proves Relation~\eqref{diffpart} in the present subcase. This
ends the proof of Theorem~\ref{main.strong} in Case I.

\pagebreak[2]
%%%%%%%%%%%%%%%%%%%%%%%%%%%%%%%%%%%%%%%%%%%%%%%%%%%%%%%%%%%%%%
\begin{caseII}
There exists a constant $C_1$ such that
$\| u_\th\|_{L^{\infty}(N)}\leq C_1$ for all $\th$.
\end{caseII}
%%%%%%%%%%%%%%%%%%%%%%%%%%%%%%%%%%%%%%%%%%%%%%%%%%%%%%%%%%%%%%

As in Case I we consider two subcases.

%%%%%%%%%%%%%%%%%%%%%%%%%%%%%%%%%%%%%%%%%%%%%%%%%%%%%%%%%%%%%%
\begin{subcaseII.1}
There exists $b > 0$ such that
\[
\liminf_{\th\searrow 0}
\left( \mu_\th\sup_{ U^N_\ep(b) } u_\th^{p-2} \right)
<
\frac{(n-k-2)^2(n-1)}{8(n-2)}.
\]
\end{subcaseII.1}
%%%%%%%%%%%%%%%%%%%%%%%%%%%%%%%%%%%%%%%%%%%%%%%%%%%%%%%%%%%%%%

By restricting to a subsequence we can assume that
\[
\mu_\th  \sup_{ U^N_\ep(b) } u_\th^{p-2}
<
\frac{(n-k-2)^2(n-1)}{8(n-2)}
\]
for all $\th$. Lemma \ref{mbelow} tells us that there is a constant
$A_0 > 0$ such that
\begin{equation} \label{uboundedl2}
\| u_\th\|_{L^2(N,g_\th)} \leq A_0.
\end{equation}
We split the treatment of Subcase II.1. into two subsubcases.

%%%%%%%%%%%%%%%%%%%%%%%%%%%%%%%%%%%%%%%%%%%%%%%%%%%%%%%%%%%%%%
\begin{subsubcaseII.1.1}
$\limsup_{b\searrow 0} \limsup_{\th\searrow 0} \sup_{U^N_\ep(b)} u_\th> 0$.
\end{subsubcaseII.1.1}
%%%%%%%%%%%%%%%%%%%%%%%%%%%%%%%%%%%%%%%%%%%%%%%%%%%%%%%%%%%%%%

We set $D_0 \definedas \frac{1}{2} \limsup_{b\searrow 0}
\limsup_{\th\searrow 0} \sup_{U^N_\ep(b)} u_\th> 0$. Then there are sequences
$(b_i)$ and $(\th_i)$ of positive numbers converging to $0$ such that
\[
\sup_{U^N_\ep(b_i)} u_{\th_i} \geq D_0,
\]
for all $i$. For brevity of notation we write $\th$ for $\th_i$ and
$b_\th$ for $b_i$. Let $x_\th' \in \overline{U^N_\ep(b_\th)}$ be such that
\begin{equation} \label{norminfbound}
u_\th(x_\th') \geq D_0.
\end{equation}
As in Subcase~I.2 above we write $x_\th' = (y_\th, t_\th, z_\th)$
where $y_\th\in W$,
$t_\th\in (-\ln R_0 + \ln \ep, -\ln \ep + \ln R_0)$, and
$z_\th\in S^{n-k-1}$.
By restricting to a subsequence we can assume that $y_\th$,
$\frac{t_\th}{A_{\th}}$, and $z_\th$ converge respectively
to $y \in W$, $T \in [-\infty, +\infty]$, and $z \in S^{n-k-1}$.
We apply Lemma~\ref{diffeom} with $V = W$, $\al = \th$,
$q_\al = y_\th$, $q=y$, $\ga_\al= \tilde{h}_{t_\th}$,
$\ga_0 = \tilde{h}_T$, and $b_\al = e^{f(t_\th)}$ and conclude that
there is a diffeomorphism
\[
\Th_\th^y :
B^k(r)
\to
B^{ \tilde{h}_{t_\th}}(y_\th, e^{-f(t_\th)} r)
\]
for $r>0$ such that
$({\Th_\th^y})^* (e^{2f(t_\th)} \tilde{h}_{t_\th})$
converges to the flat metric $\xi^k$ on $B^k(r)$.
For $r,r'>0$ we set
\[
U_\th(r,r')
\definedas
B^{\tilde{h}_{t_\th}}(y_\th, e^{-f(t_\th)} r)
\times
[t_\th- r', t_\th+ r'] \times S^{n-k-1},
\]
and we define
\[
\Th_\th:
B^k(r) \times [-r',r'] \times S^{n-k-1}
\to U_\th(r,r')
\]
by
\[
\Th_\th(y,s,z) \definedas (\Th_\th^y (y), t(s), z ),
\]
where $t(s) \definedas t_\th+ s$. By construction, $\Th_\th$ is a
diffeomorphism, and we see that
\begin{equation} \label{metric=}
\Th_\th^* ( g_\th)
=
\frac{ e^{2f(t)} }{ e^{2f(t_\th)} }
({\Th_\th^y})^* ( e^{2f(t_\th)} \tilde{h}_{t})
+ ds^2 + \sigma^{n-k-1} + \Th_\th^*(\widetilde{T}_t).
\end{equation}
We will now find the limit of $\Th_\th^* ( g_\th) $ in the
$C^2$ topology. We define
$c \definedas \lim_{\theta\searrow 0} f'(t_\theta)$, which can be assumed to
exist without loss of generality.
\begin{lemma}\label{lim_metric}
For fixed $r,r'>0$ the sequence of metrics $\Th_\th^* ( g_\th)$ tends
to $G_c = \eta^{k+1}_c + \sigma^{n-k-1} =
e^{2c s}\xi^k + ds^2 + \sigma^{n-k-1}$ in the topological space
$C^2(B^k(r) \times [-r',r'] \times S^{n-k-1})$.
\end{lemma}

As this lemma coincides with
\cite[Lemma 4.1]{ammann.dahl.humbert:09b} %This is tau.tex
we only sketch the proof.
\begin{proof}
The intermediate value theorem tells us that
\[
\left| f(t) - f(t_\th) - f'(t_\th)(t-t_\th) \right|
\leq
\frac{r'^2}{2}
\max_{s \in [t_\th-r',t_\th+r']}
\left| f''(s) \right|
\]
for all $t \in [t_\th-r',t_\th+r']$.
Because of \eqref{asumpf} we also have
$\| f'' \|_{L^\infty}\to 0$ for $\th\searrow 0$, and hence
\[
\lim_{\th\searrow 0}
\left\|
f(t) - f(t_\th) - f'(t_\th)(t-t_\th)
\right\|_{C^0([t_\th-r',t_\th+r'])}
= 0
\]
for $r'$ fixed. Further we have
\begin{equation*}
\begin{split}
\left|
\frac{d}{dt}
\Big( f(t) - f(t_\th) - f'(t_\th)(t-t_\th) \Big)
\right|
&=
\left| f'(t) - f'(t_\th) \right| \\
&=
\left|
\int_{t_\th}^t f'' (s) \, ds
\right| \\
&\leq
r' \max_{s \in [t_\th-r',t_\th+r']}
\left| f'' (s) \right| \\
&\to 0
\end{split}
\end{equation*}
as $\th\searrow 0$, and finally
\[
\left|
\frac{d^2}{dt^2}
\left( f(t) - f(t_\th) - f'(t_\th)(t-t_\th) \right)
\right|
=
\left| f''(t) \right|
\to 0
\]
as $\th\searrow 0$.
Together with $c=\lim_{\theta\searrow 0}f'(t_\theta)$
we have shown that
\[
\lim_{\th\searrow 0}
\left\|
f(t) - f(t_\th) - c(t-t_\th)
\right\|_{C^2([t_\th-r',t_\th+r'])}
= 0.
\]
Hence
\[
\lim_{\th\searrow 0}
\left\|
e^{f(t)-f(t_\th)} - e^{c(t-t_\th)}
\right\|_{ C^2([t_\th- r',t_\th+ r']) }
= 0.
\]
We now write
$e^{2f(t)} \tilde{h}_t
= e^{2f(t)} (\tilde{h}_t - \tilde{h}_{t_\th})
+ \frac{ e^{2f(t)} }{ e^{2f(t_\th)} }
e^{2f(t_\th)} \tilde{h}_{t_\th}$.
Using the fact that
\[
\lim_{\th\searrow 0}
\left\| \tilde{h}_{t} - \tilde{h}_{t_\th}
\right\|_{C^2( B^{\tilde{h}_{t_\th}}(y_\th, e^{-f(t_\th)} r))}
= 0
\]
holds uniformly for $t \in [t_\th- r', t_\th +  r']$ we obtain that the
sequence $\frac{ e^{2f(t)} }{ e^{2f(t_\th)} }
(\Th_\th^y)^* ( e^{2f(t_\th)} \tilde{h}_{t})$
tends to $e^{2c s} \xi^k$ in the $C^2(B^k(r))$-topology where as before
$s=t-t_\th\in [-r',r']$. Finally, proceeding
exactly as we did to get Relation \eqref{tto0c2}, we have that
\[
\lim_{\th\searrow 0} \Th_\th^*(\widetilde{T}_t)
= 0
\]
in $C^2 (B^k(r) \times [-r',r'] \times S^{n-k-1})$. Now Lemma~\ref{lim_metric}
follows from \eqref{metric=}.
\end{proof}

We continue with the proof of Subsubcase~II.1.1.
As in Subcases~I.1 and~I.2 we apply Lemma~\ref{lim_sol}
with $(V,g) = (\mR^{k+1} \times S^{n-k-1}, G_c)$, $\al = \th$,
and $g_\al = \Th_\th^* (g_\th)$ (we can apply this lemma
since any compact subset of $\mR^{k+1} \times S^{n-k-1}$ is
contained in some $B^k(r) \times [-r',r'] \times S^{n-k-1}$).
We obtain a $C^2$ function $u \geq 0$ that is a solution of
\[
L^{G_c} u = \bar{\mu} u^{p-1}
\]
on $\mR^{k+1} \times S^{n-k-1}$. From \eqref{normlr_lim} it follows
that
\[
\int_{\mR^{k+1} \times S^{n-k-1}}
u^p \,dv^{G_c}
\leq 1.
\]
{}From \eqref{norminf_lim} it follows that $u \in L^\infty(\mR^{k+1}
\times S^{n-k-1})$. With \eqref{norminfbound}, we see that $u(0) \geq D_0$
and thus, $u \not\equiv 0$.  By \eqref{uboundedl2}, we also get that $u
\in L^2(\mR^{k+1} \times S^{n-k-1})$. By the definition of
$\La^{(1)}_{n,k}$ we have that
$\bar{\mu} \geq \La^{(1)}_{n,k} \geq \La_{n,k}$. This ends the proof
of Theorem \ref{main.strong} in this subsubcase.

%%%%%%%%%%%%%%%%%%%%%%%%%%%%%%%%%%%%%%%%%%%%%%%%%%%%%%%%%%%%%
\begin{subsubcaseII.1.2}
$\lim_{b\searrow 0} \limsup_{\th\searrow 0} \sup_{U^N_\ep(b)} u_\th= 0$.
\end{subsubcaseII.1.2}
%%%%%%%%%%%%%%%%%%%%%%%%%%%%%%%%%%%%%%%%%%%%%%%%%%%%%%%%%%%%%

The proof in this subsubcase proceeds in several steps.

%%%%%%%%%%%%%%%%%%%%%%%%%%%%%%%%%%%%%%%%%%%%%%%%%%%%%%%%%%%%%
\begin{step} \label{step1}
We prove
$\lim_{b \searrow 0} \limsup_{\th  \searrow 0} \int_{U^N_\ep(b)} u_\th^p \, dv^{g_\th}
= 0$.
\end{step}
%%%%%%%%%%%%%%%%%%%%%%%%%%%%%%%%%%%%%%%%%%%%%%%%%%%%%%%%%%%%%

Let $b >0$. Using \eqref{uboundedl2} we have
\[
\int_{U^N_\ep(b) } u_\th^p \, dv^{g_\th}
\leq
A_0  \sup_{U^N_\ep(b)} u_\th^{p-2}
\]
where the constant $A_0$ is independent of $b$ and $\th$. Step
\ref{step1} follows.

%%%%%%%%%%%%%%%%%%%%%%%%%%%%%%%%%%%%%%%%%%%%%%%%%%%%%%%%%%%%%
\begin{step}\label{step2}
We show
$\liminf_{b\searrow 0} \liminf_{\th\searrow 0}
\int_{U^N_\ep(2b) \setminus U^N_\ep(b)} u_\th^2 \, dv^{g_\th} =0$.
\end{step}
%%%%%%%%%%%%%%%%%%%%%%%%%%%%%%%%%%%%%%%%%%%%%%%%%%%%%%%%%%%%%

Let
\[
d_0
\definedas
\liminf_{b\searrow 0} \liminf_{\th\searrow 0}
\int_{U^N_\ep(2b) \setminus U^N_\ep(b)} u_\th^2 \, dv^{g_\th}.
\]
We prove this step by contradiction and assume that $d_0 >0$.
Then there exists $\de >0$ such that for all $b \in (0,\de]$,
\[
\liminf_{\th\searrow 0} \int_{U^N_\ep(2b) \setminus U^N_\ep(b)  } u_\th^2
\, dv^{g_\th}
\geq \frac{d_0}{2}.
\]
For $m \in \mN$ we set $V_m \definedas U(2^{-m} \de) \setminus
U(2^{-(m+1)} \de)$. In particular we have
\[
\liminf_{\th\searrow 0} \int_{V_m} u_\th^2 \, dv^{g_\th}
\geq
\frac{d_0}{2}
\]
for all $m$. Let $N_0 \in \mN$. For $m \neq m'$ the sets $V_m$ and
$V_{m'}$ are disjoint. Hence we can write
\[
\int_N u_\th^2 \, dv^{g_\th}
\geq
\int_{\bigcup_{m=0}^{N_0} V_m}  u_\th^2 \, dv^{g_\th}
\geq
\sum_{m=0}^{N_0} \int_{V_m} u_\th^2 \, dv^{g_\th}
\]
for $\th$ small enough. This leads to
\begin{equation*}
\begin{split}
\liminf_{\th\searrow 0} \int_N u_\th^2 \, dv^{g_\th}
&\geq
\liminf_{\th\searrow 0}
\sum_{m=0}^{N_0} \int_{V_m} u_\th^2 \, dv^{g_\th}  \\
&\geq  \sum_{m=0}^{N_0}  \liminf_{\th\searrow 0} \int_{V_m}
u_\th^2 \, dv^{g_\th} \\
&\geq
(N_0 + 1) \frac{d_0}{2}.
\end{split}
\end{equation*}
Since $N_0$ is arbitrary, this contradicts that $(u_\th)$ is bounded
in $L^2(N)$ and proves Step \ref{step2}.

%%%%%%%%%%%%%%%%%%%%%%%%%%%%%%%%%%%%%%%%%%%%%%%%%%%%%%%%%%%%%
\begin{step} %\label{step3}
Conclusion.
\end{step}
%%%%%%%%%%%%%%%%%%%%%%%%%%%%%%%%%%%%%%%%%%%%%%%%%%%%%%%%%%%%%

Let $d_0 >0$. By Steps \ref{step1} and \ref{step2} we can find $b>0$
such that after passing to a subsequence, we have for all $\th$ close to $0$
\begin{equation} \label{largelp}
\int_{N \setminus U^N_\ep(2b)} u_\th^p \, dv^{g_\th}
\geq
1 - d_0
\end{equation}
and
\begin{equation}\label{smalll2}
\int_{U^N_\ep(2b)  \setminus U^N_\ep(b) } u_\th^2 \, dv^{g_\th}
\leq
d_0.
\end{equation}
Let $\chi \in C^{\infty} (M) $, $0 \leq \chi \leq 1$, be a cut-off
function equal to $0$ on $U^N_\ep(b)$ and equal to $1$ on
$N \setminus U^N_\ep(2b)$. Since the set $U^N_\ep(2b) \setminus U^N_\ep(b)$
corresponds to $ t \in [t_0-\ln 2,t_0] \cup [t_1,t_1+\ln 2]$ with
$t_0 = -\ln b +\ln \ep$ and $t_1 = \ln b - \ln \ep$ we can assume
that
\begin{equation} \label{dchi}
| d\chi |_{g_\th} \leq 2 \ln2.
\end{equation}

We will use the function $\chi u_\th$ to estimate $\mu$. This function
is supported in $N \setminus U^N_\ep(b)$. If $\th$ is smaller than $b$,
then $(N \setminus U^N_\ep(b), g_\th)$ is isometric to
$(M \setminus U^M(b), F^2 g)$. In other words
$(N \setminus U^N_\ep(b), g_\th)$ is conformally equivalent to
$(M \setminus U^M(b), g)$. Relation~\eqref{confJ} implies that
\begin{equation} \label{muleqJ}
\mu
\leq
J_{\th}(\chi u_\th)
=
\frac{\int_N ( \an |d (\chi u_\th)|_{g_\th}^2
+ \Scal^{g_\th} (\chi u_\th)^2) \, dv^{g_\th}}
{{\left(
\int_{N} (\chi u_\th)^p \, dv^{g_\th}
\right)}^{\frac{n-2}{n}}} .
\end{equation}
We multiply Equation \eqref{eqa} by $\chi^2 u_\th$ and integrate
over $N$. We can re-write the result using the following form of
\eqref{formula.dchiu},
\[
\int_N  |d (\chi u_\th)|_{g_\th}^2 \, dv^{g_\th}
=
\int_N \chi^2 u_\th\Delta^{g_\th} u_\th  \, dv^{g_\th}
+
\int_N  |d\chi|_{g_\th}^2 u_\th^2 \, dv^{g_\th},
\]
to obtain
\begin{equation*}
\begin{split}
\int_N \left( \an |d (\chi u_\th)|_{g_\th}^2
+ \Scal^{g_\th} (\chi u_\th)^2 \right) \,dv^{g_\th}
&=
\mu_\th\int_N u_\th^p \chi^2 \,dv^{g_\th}
+ \an \int_N  |d\chi|_{g_\th}^2 u_\th^2 \, dv^{g_\th} \\
&\leq
\mu_\th\int_N u_\th^p \,dv^{g_\th}
+ |\mu_\th|  \int_{U^N_\ep(2b)} u_\th^p \, dv^{g_\th} \\
&\quad
+ \an \int_N  |d\chi|_{g_\th}^2 u_\th^2 \, dv^{g_\th}.
\end{split}
\end{equation*}
Using \eqref{smalll2} and \eqref{dchi}, we have
\[
\int_N  |d\chi|_{g_\th}^2 u_\th^2 \,dv^{g_\th}
=
\int_{U^N_\ep(2b)\setminus U^N_\ep(b)} |d\chi|_{g_\th}^2 u_\th^2 \,dv^{g_\th}
\leq
4 (\ln2)^2 d_0.
\]
Relation~\eqref{largelp} implies $\int_{U^N_\ep(2b)} u_\th^p \, dv^{g_\th}
\leq d_0$. Together with
$\int_N u_\th^p \, dv^{g_\th}~=~1$, we have
\begin{equation} \label{numJ}
\int_N
( \an | d(\chi u_\th) |_{g_\th}^2 + \Scal^{g_\th} (\chi u_\th)^2)
\, dv^{g_\th}
\leq
\mu_\th+|\mu_\th| d_0 + 4   (\ln2)^2 \an d_0.
\end{equation}
In addition, by Relation \eqref{largelp},
\begin{equation} \label{denJ}
\int_{N} (\chi u_\th)^p \, dv^{g_\th} \geq 1 - d_0.
\end{equation}
Plugging \eqref{numJ} and \eqref{denJ} in \eqref{muleqJ} we get
\[
\mu
\leq
\frac{\mu_\th+|\mu_\th|d_0 +  4 (\ln2)^2 \an d_0 }
{(1-d_0)^{\frac{n-2}{n}}}
\]
for small $\th$. By taking the limit $\th\searrow 0$ we can replace
$\mu_\th$ by $\bar{\mu}$ in this inequality. Since $d_0$ can be chosen
arbitrarily small we finally obtain $\mu \leq \bar{\mu}$. This proves
Theorem \ref{main.strong} in Subcase II.1.

%%%%%%%%%%%%%%%%%%%%%%%%%%%%%%%%%%%%%%%%%%%%%%%%%%%%%%%%%%%%%
\begin{subcaseII.2} For all $b >0$, we have
\[
\liminf_{\th\searrow 0}
\left(
\mu_\th\sup_{ U^N_\ep( b  ) } u_\th^{p-2}
\right)
\geq
\frac{(n-k-2)^2(n-1)}{8(n-2)}.
\]
\end{subcaseII.2}
%%%%%%%%%%%%%%%%%%%%%%%%%%%%%%%%%%%%%%%%%%%%%%%%%%%%%%%%%%%%%

Hence, we can construct a subsequence of $\th$ and a sequence
$(b_\th)$ of positive numbers converging to $0$ with
\[
\liminf_{\th\searrow 0}
\left(
\mu_\th\sup_{ U^N_\ep( b_\th) } u_\th^{p-2}
\right)
\geq
\frac{(n-k-2)^2(n-1)}{8(n-2)}.
\]
Choose a point $x_\th'' \in \overline{U^N_\ep(b_\th)}$ such that
$u_\th(x_\th'') = \sup_{U^N_\ep(b_\th)} u_\th$. Since $\mu_\th  \leq
\mu(\mS^n)$, we have
\[u_\th(x_\th'') \geq
D_1
\definedas
\left(
\frac{(n-k-2)^2 (n-1)}{8 \mu(\mS^n)(n-2)}
\right)^{\frac{1}{p-2}}.
\]
With similar arguments as in Subcase II.1.1 (just replace $x_\th'$
by $x_\th''$ and $D_0$ by $D_1$), we get the existence of a $C^2$
function $u \geq 0$ that is a solution of
\[
L^{G_c} u
=
\bar{\mu} u^{p-1}
\]
on $\mH^{k+1}_c \times S^{n-k-1}$. As in Subsubcase II.1.1,
$u \not\equiv 0$, $u \in L^\infty(\mH^{k+1}_c \times S^{n-k-1})$, and
\[
\int_{\mR^{k+1} \times S^{n-k-1}}
u^p \,dv^{G_c}
\leq 1.
\]
Moreover, the assumption of Subcase II.2 implies that
\[
\bar{\mu} u^{p-2}(0)
=
\lim_{\th\searrow 0} \mu_\th u_\th^{p-2}(x_\th'')
\geq
\frac{(n-k-2)^2 (n-1)}{8(n-2)}.
\]
By the definition of $\La^{(2)}_{n,k}$, we have that
$\bar{\mu} \geq \La^{(2)}_{n,k} \geq \La_{n,k}$.

%%%%%%%%%%%%%%%%%%%%%%%%%%%%%%%%%%%%%%%%%%%%%%%%%%%%%%%%%%%%%%%%%
\appendix
%%%%%%%%%%%%%%%%%%%%%%%%%%%%%%%%%%%%%%%%%%%%%%%%%%%%%%%%%%%%%%%%%

%%%%%%%%%%%%%%%%%%%%%%%%%%%%%%%%%%%%%
\section{Some details}
%%%%%%%%%%%%%%%%%%%%%%%%%%%%%%%%%%%%%
%%%%%%%%%%%%%%%%%%%%%%%%%%%%%%%%%%%%%%%%%%%%%%%%%%%%%%%%%%%%%%%%%
\subsection{Scalar curvature}
%%%%%%%%%%%%%%%%%%%%%%%%%%%%%%%%%%%%%%%%%%%%%%%%%%%%%%%%%%%%%%%%%

In this section $U$ denotes an open subset of a manifold
and $q\in U$ a fixed point.

\begin{proposition}
Let $g$ be a Riemannian metric on $U$ and $T$ a symmetric $2$-tensor
such that $\ti g \definedas g+T$ is also a Riemannian metric. Then the
scalar curvature $\scal^{\ti g}(q)$ of $\ti g$ in $q\in U$ is a smooth
function of the Riemann tensor $R^g(q)$ of $g$ at $q$,  $T(q)$,
$\na^g T(q)$, and $(\na^g)^2 T(q)$. Moreover, the operator
$T \mapsto \scal^{g+T}(q)$ is a quasilinear partial differential
operator of second order.
\end{proposition}

\begin{proof}
The proof is straightforward; we will just give a sketch using
notation from \cite{ammann.grosjean.humbert.morel:08} which coincides
with that of \cite{hebey:97}. We denote the components of the
curvature tensors of $g$ and $\ti g$ by
\[
R_{ijkl}
=
g(R^g(\pa_k,\pa_l)\pa_j,\pa_i),
\qquad
\ti R_{ijkl}
=
\ti g(R^{\ti g}(\pa_k,\pa_l)\pa_j,\pa_i).
\]
We work in normal coordinates for the metric $g$ centered in $q$,
indices of partial derivatives in coordinates are added and separated
with a comma ``,'' and covariant ones with respect to $g$ separated
with a semi-colon ``;''. In particular $T=T_{ij}dx^i\,dx^j$,
\[
T_{kl;i}
=
(\na_i T)(\pa_k,\pa_l)
=
\pa_iT_{kl} - T_{ml}\Ga_{ik}^m -T_{km}\Ga_{il}^m.
\]
At the point $q$ we have $\ti g_{kl,i}=T_{kl;i}$. As explained in
\cite[Formula (13)]{ammann.grosjean.humbert.morel:08} we have
\[
\na_\al\Ga_{ij}^k
=
\pa_\al\Ga_{ij}^k
=
-\frac13 \left(R_{ik\al j}+R_{i\al kj}\right)
\]
at the point $q$. Hence in that point,
\begin{eqnarray*}
T_{kl;rs}
 &= &
(\na^2_{rs}T)(\pa_k,\pa_l)\\
 &= &
\pa_r \pa_s T_{kl}
+ \frac13 T_{ml}(R_{smrk}+R_{srmk})
+ \frac13 T_{mk}(R_{smrl}+R_{srml}).
\end{eqnarray*}
In order to calculate the scalar curvature $\Scal^{\ti g}(q)$ of
$\ti g$ in $q$ we use the curvature formula as in \cite{hebey:97}
and contract twice. We obtain
\begin{equation} \label{scal.formula}
\Scal^{\ti g}(q)
=
\ti g^{ik} \ti g^{jm} (\ti g_{km,ij}-\ti g_{ki,mj})
+ P(\ti g^{rm},\ti g_{ij,k})
\end{equation}
where $P$ is a polynomial expression in $\ti g^{-1}$ and
$\pa \ti g$ that is cubic in $\ti g^{-1}=\ti g^{rm}$ and quadratic
in $\ti g_{ij,k}$. Note that formula \eqref{scal.formula} holds for an
arbitrary metric in arbitrary coordinates. The polynomial $P$
vanishes for $T=0$ in normal coordinates for $g$.
\end{proof}

\begin{corollary}\label{cor.scal.diff}
Let $\cR\subset T^*_qM\otimes T^*_qM \otimes T^*_qM \otimes T_qM$
be a bounded set of curvature tensors. Then there is an $\ep>0$ and
$C \in \mR$ such that for all metrics $g$ on $U$ with $R^g|_q\in \cR$
we have the following: if
\[
\max_{i\in\{0,1,2\}} \left| (\na^g)^i T (q) \right| < \ep,
\]
then
\[
|\scal^{g+T}(q) - \scal^{g}(q)|
\leq
C\left(
\left| (\na^g)^2 T (q) \right|
+ \left| \na^g T (q) \right|^2
+ \left| T (q)  \right|
\right).
\]
\end{corollary}

%%%%%%%%%%%%%%%%%%%%%%%%%%%%%%%%%%%%%%%%%%%%%%%%%%%%%%%%%%%%%%%%%
\subsection{Details for equation \eqref{scalP}} \label{app.bgm}
%%%%%%%%%%%%%%%%%%%%%%%%%%%%%%%%%%%%%%%%%%%%%%%%%%%%%%%%%%%%%%%%%

We compute the scalar curvature of the metric
$dt^2 + e^{2\phi(t)} h_t$ on $I \times W$. This is a generalized
cylinder metric as studied in \cite{baer.gauduchon.moroianu:05}. In
the following computations we use the notation from
\cite{baer.gauduchon.moroianu:05}, so $g_t = e^{2\phi(t)}h_t$ and we
have
\[
\dot g_t
=
2\phi'(t)e^{2\phi(t)} h_t
+
e^{2\phi(t)} \pa_t h_t,
\]
and
\[
\ddot g_t
=
(2\phi''(t) + 4\phi'(t)^2) e^{2\phi(t)} h_t
+
4\phi'(t) e^{2\phi(t)}\pa_t h_t
+
e^{2\phi(t)} \pa_t^2 h_t.
\]
This implies that the shape operator $S$ of the hypersurfaces defined by having
constant value~$t$ is given by
\[
S = -\phi' \Id - \frac{1}{2} h_t^{-1}\pa_t h_t,
\]
so
\[
\tr (S^2)
=
k \phi'(t)^2
+ \phi'(t) \tr (h_t^{-1} \pa_t h_t )
+ \frac{1}{4} \tr( (h_t^{-1} \pa_t h_t)^2 ),
\]
and
\[
(\tr S)^2
=
k^2 \phi'(t)^2
+ k \phi'(t) \tr (h_t^{-1} \pa_t h_t )
+ \frac{1}{4} (\tr(h_t^{-1} \pa_t h_t) )^2.
\]
Further
\begin{equation*}
\begin{split}
\tr^{g_t} \ddot g_t
&=
(2\phi''(t) + 4\phi'(t)^2) k
+
4\phi'(t) \tr^{h_t} (\pa_t h_t)
+
\tr^{h_t} (\pa_t^2 h_t) \\
&=
(2\phi''(t) + 4\phi'(t)^2) k
+
4\phi'(t) \tr (h_t^{-1} \pa_t h_t )
+
\tr (h_t^{-1}\pa_t^2 h_t) .
\end{split}
\end{equation*}
{}From \cite[Proposition 4.1, (21)]{baer.gauduchon.moroianu:05} we have
\begin{equation*}
\begin{split}
\scal^{e^{2\phi(t)} h_t + dt^2}
&=
\scal^{e^{2\phi(t)} h_t}
+ 3 \tr (S^2) - (\tr S)^2 - \tr^{g_t} \ddot g_t \\
&=
e^{-2\phi(t)} \scal^{h_t} - k(k+1) \phi'(t)^2\\
& \quad - (k+1) \phi'(t) \tr (h_t^{-1} \pa_t h_t )
- 2k \phi''(t) + \frac{3}{4} \tr( ( h_t^{-1} \pa_t h_t)^2 ) \\
&\quad
- \frac{1}{4} (\tr(h_t^{-1} \pa_t h_t))^2
- \tr (h_t^{-1} \pa_t^2 h_t).
\end{split}
\end{equation*}
When we add the scalar curvature of $\si^{n-k-1}$ we get Formula
\eqref{scalP} for the scalar curvature of
$\gWS = dt^2 + e^{2\phi(t)} h_t + \si^{n-k-1}$.

%%%%%%%%%%%%%%%%%%%%%%%%%%%%%%%%%%%%%%%%%%%%%%%%%%%%%%%%%%%%%%%%%
\subsection{A cut-off formula} \label{app.formula.dchiu}
%%%%%%%%%%%%%%%%%%%%%%%%%%%%%%%%%%%%%%%%%%%%%%%%%%%%%%%%%%%%%%%%%

Here we state a formula used several times in the article.
Assume that $u$ and $\chi$ are smooth functions on a
Riemannian manifold $(N,h)$, and that $\chi$ has compact
support. Then
\begin{equation} \label{formula.dchiu}
\begin{split}
\int_N |d(\chi u)|^2 \, dv^h
&=
\int_N
\left(
u^2 |d\chi|^2 + \< u d\chi, \chi du \> + \< \chi du, d(\chi u) \>
\right)
\, dv^h \\
&=
\int_N
\left(
u^2 |d\chi|^2 + \chi u \< d\chi, du \> + \< du, \chi d(\chi u) \>
\right)
\, dv^h \\
&=
\int_N
\left(
u^2 |d\chi|^2 + \chi u \< d\chi, du \>
+ \< du, d(\chi^2 u) - \chi u d\chi \>
\right)
\, dv^h \\
&=
\int_N
\left(
u^2 |d\chi|^2 + \< du, d(\chi^2 u) \>
\right)
\, dv^h \\
&=
\int_N \left(u^2 |d\chi|^2 + \chi^2 u \De^h u\right) dv^h.
\end{split}
\end{equation}

%%%%%%%%%%%%%%%%%%%%%%%%%%%%%%%%%%%%%%%%%%%%%%%%%%%%%%%%%%%
%\bibliographystyle{amsplain}
%\bibliography{literatur}

\begin{thebibliography}{10}

\bibitem{akutagawa.botvinnik:03}
K.~Akutagawa and B.~Botvinnik, \emph{Yamabe metrics on cylindrical manifolds},
  Geom. Funct. Anal. \textbf{13} (2003), 259--333, MR1982146, Zbl 1161.53344.

\bibitem{akutagawa.florit.petean:07}
K.~Akutagawa, L.~Florit, and J.~Petean, \emph{On {Y}amabe constants of
  Riemannian products}, Comm. Anal. Geom. \textbf{15} (2007), 947--969, MR2403191, Zbl 1147.53032.

\bibitem{akutagawa.neves:07}
K.~Akutagawa and A.~Neves, \emph{3-manifolds with {Y}amabe invariant greater
  than that of {$\Bbb R\Bbb P\sp 3$}}, J. Differential Geom. \textbf{75}
  (2007), 359--386, MR2301449, Zbl~1119.53027.

\bibitem{ammann.dahl.humbert:09b}
B.~Ammann, M.~Dahl, and E.~Humbert, \emph{Surgery and the spinorial
  $\tau$-invariant}, Comm. Part. Diff. Eq., \textbf{34} (2009),
1147--1179, MR2581968, Zbl 1183.53041.

\bibitem{ammann.dahl.humbert:13}
\bysame, \emph{The conformal {Y}amabe constant of product manifolds}, {P}roc.
  {A}{M}{S} \textbf{141} (2013), 295--307,
  \url{http://arxiv.org/abs/1103.1826}.

\bibitem{ammann.dahl.humbert:p11b}
\bysame, \emph{Square-integrability of solutions of the {Y}amabe equation},
  Preprint, 2011, \url{http://arxiv.org/abs/1111.2780}.

\bibitem{ammann.dahl.humbert:p12}
\bysame, \emph{Low-dimensional surgery and the {Y}amabe invariant}, Preprint,
  2012, \url{http://arxiv.org/abs/1204.1197}.

\bibitem{ammann.grosjean.humbert.morel:08}
B.~Ammann, J.~F. Grosjean, E.~Humbert, and B.~Morel, \emph{A spinorial analogue of {A}ubin's inequality}, Math. Z. {\bf 260} (2008), 127--151, MR2413347, Zbl 1145.53039.

\bibitem{anderson:06}
M.~T. Anderson, \emph{Canonical metrics on 3-manifolds and 4-manifolds}, Asian J. Math. \textbf{10} (2006), 127--163, MR2213687, Zbl pre05170984.

\bibitem{aubin:76}
T.~Aubin, \emph{{{\'E}quations diff{\'e}rentielles non lin{\'e}aires et
  probl{\`e}me de Yamabe concernant la courbure scalaire.}}, J. Math. Pur.
  Appl., IX. Ser. \textbf{55} (1976), 269--296, MR0431287, Zbl 0336.53033.

\bibitem{baer.gauduchon.moroianu:05}
C.~B{\"a}r, P.~Gauduchon, and A.~Moroianu, \emph{Generalized cylinders in
  semi-riemannian and spin geometry}, Math. Z. \textbf{249} (2005), 545--580, MR2121740, Zbl 1068.53030.

\bibitem{botvinnik.rosenberg:02}
B.~Botvinnik and J.~Rosenberg, \emph{The {Y}amabe invariant for non-simply
  connected manifolds}, J. Differential Geom. \textbf{62} (2002), 175--208, MR1988502, Zbl~1071.53021.

\bibitem{bray.neves:04}
H.~L. Bray and A.~Neves, \emph{Classification of prime 3-manifolds with
  {Y}amabe invariant greater than {$\Bbb{RP}\sp 3$}}, Ann. of Math. (2)
  \textbf{159} (2004), 407--424, MR2081443, Zbl 1066.53077.

\bibitem{cheeger:70}
J.~Cheeger, \emph{Finiteness theorems for {R}iemannian manifolds}, Amer. J.
  Math. \textbf{92} (1970), 61--74, MR0263092, Zbl 0194.52902.

\bibitem{chrusciel.isenberg.pollack:05}
P.~Chru{\'s}ciel, J.~Isenberg, and D.~Pollack, \emph{Initial data engineering}, Comm. Math. Phys. \textbf{257} (2005), 29--42, MR2163567, Zbl 1080.83002.

\bibitem{corvino.eichmair.miao:p12}
J.~Corvino, M.~Eichmair, and P.~Miao, \emph{Deformation of scalar curvature and volume}, Preprint, 2012.

\bibitem{evans:90}
L.~C. Evans, \emph{Weak convergence methods for nonlinear partial differential equations}, CBMS Regional Conference Series in Mathematics, vol.~74, AMS, 1990, MR1034481, Zbl 0698.35004.

\bibitem{gilbarg.trudinger:77}
D.~Gilbarg and N.~Trudinger, \emph{Elliptic partial differential equations of
  second order}, Grundlehren der mathematischen Wissenschaften, no. 224,
  Springer-Verlag, 1977, MR0473443, Zbl 0361.35003.

\bibitem{gromov:81}
M.~Gromov, \emph{Structures m{\'e}triques pour les vari{\'e}t{\'e}s
  {R}iemanniennes}, CEDIC, Paris, 1981, MR0682063, Zbl 0509.53034.

\bibitem{gromov:99}
\bysame, \emph{Metric structures for {R}iemannian and non-{R}iemannian spaces}, Birkh\"auser Boston Inc., Boston, MA, 1999, MR1699320, Zbl 0953.53002.

\bibitem{gromov.lawson:80b}
M.~Gromov and H.~B. Lawson, \emph{The classification of simply connected
  manifolds of positive scalar curvature}, Ann. of Math. (2) \textbf{111}
  (1980), 423--434, MR0577131, Zbl 0463.53025.

\bibitem{grosse:p09}
N.~Gro\ss{}e, \emph{The {Y}amabe equation on manifolds of bounded geometry},
Preprint, 2009, \url{http://arxiv.org/abs/0912.4398v2}.

\bibitem{gursky.lebrun:98}
M.~J. Gursky and C.~LeBrun, \emph{Yamabe invariants and {${\rm Spin}\sp c$}
  structures}, Geom. Funct. Anal. \textbf{8} (1998), 965--977, MR1664788, Zbl 0931.53019.

\bibitem{hanke:08}
B.~Hanke, \emph{Positive scalar curvature with symmetry}, J. Reine Angew.
  Math. \textbf{614} (2008), 73--115, MR2376283, Zbl~1144.53048.

\bibitem{hebey:96}
E.~Hebey, \emph{Sobolev spaces on Riemannian manifolds}, {Lecture Notes in
  Mathematics, vol. 1635, Berlin: Springer}, 1996, MR1481970, Zbl~0866.58068.

\bibitem{hebey:97}
\bysame, \emph{{Introduction \`a l'analyse non-lin\'eaire sur les
  vari\'et\'es}}, Diderot Editeur, Arts et Sciences, out of print, no longer
  sold, 1997.

\bibitem{hebey.vaugon:93}
E.~Hebey and M.~Vaugon, \emph{Le probl\`eme de {Y}amabe \'equivariant}, Bull.
  Sci. Math. \textbf{117} (1993), 241--286, MR1216009, Zbl 0786.53024.

\bibitem{joyce:03}
D.~Joyce, \emph{Constant scalar curvature metrics on connected sums}, Int. J.
  Math. Math. Sci. (2003), 405--450, MR1961016, Zbl 1026.53019.

\bibitem{kim:00}
S.~Kim, \emph{An obstruction to the conformal compactification of {R}iemannian manifolds}, Proc. Amer. Math. Soc. \textbf{128} (2000), 1833--1838, MR1646195, Zbl~0956.53032.

\bibitem{kleiner.lott:08}
B.~Kleiner and J.~Lott, \emph{Notes on {P}erelman's papers}, Geom. Topol.
  \textbf{12} (2008), no.~5, 2587--2855.

\bibitem{kobayashi:85}
O.~Kobayashi, \emph{On large scalar curvature}, research report {\bf 11},
  {D}ept. {M}ath. {K}eio {U}niv., 1985.

\bibitem{kobayashi:87}
\bysame, \emph{Scalar curvature of a metric with unit volume}, Math. Ann.
  \textbf{279} (1987), 253--265, MR0919505, Zbl 0611.53037.

\bibitem{kosinski:93}
A.~A. Kosinski, \emph{Differential manifolds}, Pure and Applied Mathematics,
  vol. 138, Academic Press Inc., Boston, MA, 1993, MR1190010, Zbl 0767.57001.

\bibitem{lebrun:96}
C.~LeBrun, \emph{Four-manifolds without {E}instein metrics}, Math. Res. Lett.
  \textbf{3} (1996), 133--147, MR1386835, Zbl 0856.53035.

\bibitem{lebrun:99b}
\bysame, \emph{Einstein metrics and the {Y}amabe problem}, Trends in
  mathematical physics (Knoxville, TN, 1998), AMS/IP Stud. Adv. Math., vol.~13, Amer. Math. Soc., Providence, RI, 1999, pp.~353--376, MR1708770, Zbl 1050.53032.


\bibitem{lebrun:99a}
\bysame, \emph{Kodaira dimension and the {Y}amabe problem}, Comm. Anal. Geom.
  \textbf{7} (1999), 133--156, MR1674105, Zbl 0996.32009.

\bibitem{lebrun:p08}
\bysame, \emph{Einstein metrics, complex surfaces, and symplectic
  $4$-manifolds}, \emph{Einstein metrics, complex surfaces, and symplectic 4-manifolds}, Math. Proc. Cambridge Philos. Soc. \textbf{147}, (2009), 1--8, MR2507306, Zbl 1175.57014.

\bibitem{lee.parker:87}
J.~M. Lee and T.~H. Parker, \emph{{The Yamabe problem}}, Bull. Am. Math. Soc., New Ser. \textbf{17} (1987), 37--91, MR0888880, Zbl 0633.53062.


\bibitem{madani:10}
F.~Madani, \emph{Equivariant {Y}amabe problem and {H}ebey--{V}augon conjecture}, J. Func. Anal. \textbf{258} (2010), 241--254, MR2557961, Zbl 1180.53041.

\bibitem{madani:12}
F.~Madani, \emph{Hebey--{V}augon conjecture {II}}, C. R. Math. Acad. Sci. Paris
  \textbf{350} (2012), no.~17-18, 849--852.

\bibitem{mazzeo.pollack.uhlenbeck:95}
R.~Mazzeo, D.~Pollack, and K.~Uhlenbeck, \emph{Connected sum constructions for constant scalar curvature metrics}, Topol. Methods Nonlinear Anal. \textbf{6} (1995), 207--233, MR1399537, Zbl 0866.58069.

\bibitem{mazzieri:08}
L.~Mazzieri, \emph{Generalized connected sum construction for nonzero
constant scalar curvature metrics}, Comm. Part. Diff. Eq.  {\bf 33}
(2008), 1--17, MR2398217, Zbl~1156.53021.

\bibitem{mazzieri:p06b}
\bysame, \emph{Generalized connected sum construction for scalar flat metrics}, Manuscripta Math. \textbf{129} (2009), 137--168, MR2505799, Zbl 1230.53034.

\bibitem{petean:03}
J.~Petean, \emph{The {Y}amabe invariant of simply connected manifolds}, J.
  Reine Angew. Math. \textbf{523} (2000), 225--231, MR1762961, Zbl 0949.53026.

\bibitem{petean.yun:99}
J.~Petean and G.~Yun, \emph{Surgery and the {Y}amabe invariant}, Geom. Funct.
  Anal. \textbf{9} (1999), 1189--1199, MR1736933, Zbl 0976.53045.

\bibitem{petersen:97}
P.~Petersen, \emph{Convergence theorems in {R}iemannian geometry}, Comparison
  geometry (Berkeley, CA, 1993--94), Math. Sci. Res. Inst. Publ., vol.~30,
  Cambridge Univ. Press, Cambridge, 1997, pp.~167--202, MR1452874, Zbl 0898.53035.

\bibitem{rosenberg.stolz:01}
J.~Rosenberg, S.~Stolz,
\emph{Metrics of positive scalar curvature and connections with surgery},
Surveys on surgery theory, Vol. 2, 353--386,
Ann. of Math. Stud., 149, Princeton Univ. Press, Princeton, NJ, 2001, MR1818778, Zbl 0971.57003.


\bibitem{schoen:84}
R.~Schoen, \emph{{Conformal deformation of a Riemannian metric to constant
  scalar curvature}}, J. Diff. Geom. \textbf{20} (1984), 479--495, MR0788292, Zbl~0576.53028.

\bibitem{schoen:87}
\bysame, \emph{Recent progress in geometric partial differential equations},
  Proceedings of the International Congress of Mathematicians, Vol. 1, 2
  (Berkeley, Calif., 1986) (Providence, RI), Amer. Math. Soc., 1987,
  pp.~121--130, MR0934219, Zbl~0692.35046.

\bibitem{schoen:89}
\bysame, \emph{{Variational theory for the total scalar curvature functional
  for Riemannian metrics and related topics}}, {Topics in calculus of
  variations, Lect. 2nd Sess., Montecatini/Italy 1987, Lect. Notes Math. 1365, pp.~120--154}, 1989, MR0994021, Zbl~0702.49038.

\bibitem{schoen.yau:79c}
R.~Schoen and S.-T. Yau, \emph{On the structure of manifolds with positive
  scalar curvature}, Manuscripta Math. \textbf{28} (1979), 159--183, MR0535700, Zbl 0423.53032.

\bibitem{schoen.yau:88}
\bysame, \emph{{Conformally flat manifolds, Kleinian groups and scalar
  curvature.}}, Invent. Math. \textbf{92} (1988), 47--71, MR0954421, Zbl~0658.53038.

\bibitem{schoen.yau:94}
\bysame, \emph{{Lectures on differential geometry.}}, Conference Proceedings
  and Lecture Notes in Geometry and Topology, International
  Press, 1994, MR1333601, Zbl~0830.53001.

\bibitem{sung:06}
C.~Sung, \emph{Surgery and equivariant {Y}amabe invariant}, Differential Geom. Appl. \textbf{24} (2006), 271--287, MR2216941, Zbl 1099.53031.

\bibitem{sung:09}
\bysame, \emph{Surgery, {Y}amabe invariant, and {S}eiberg-{W}itten theory}, J. Geom. Phys. \textbf{59} (2009), 246--255, MR2492194, Zbl 1163.53023.

\bibitem{trudinger:68}
N.~S. Trudinger, \emph{{Remarks concerning the conformal deformation of
  Riemannian structures on compact manifolds}}, Ann. Sc. Norm. Super. Pisa,
  Sci. Fis. Mat., III. Ser. \textbf{22} (1968), 265--274, MR0240748, Zbl 0159.23801.

\bibitem{yamabe:60}
H.~Yamabe, \emph{{On a deformation of Riemannian structures on compact
  manifolds.}}, Osaka Math. J. \textbf{12} (1960), 21--37, MR0125546, Zbl 0096.37201.

\end{thebibliography}
%%%%%%%%%%%%%%%%%%%%%%%%%%%%%%%%%%%%%%%%%%%%%%%%%%%%%%%%%%%

\providecommand{\bysame}{\leavevmode\hbox to3em{\hrulefill}\thinspace}
\providecommand{\MR}{\relax\ifhmode\unskip\space\fi MR }
\providecommand{\MRhref}[2]{%
  \href{http://www.ams.org/mathscinet-getitem?mr=#1}{#2}
}
\providecommand{\href}[2]{#2}

%%%%%%%%%%%%%%%%%%%%%%%%%%%%%%%%%%%%%%%%%%%%%%%%%%%%%%%%%%%%%%%%%
\end{document}